\documentclass{amsart}
\usepackage{amssymb,amsmath,amsthm,epsf,epsfig,dsfont,bbm}

\usepackage{latexsym}
\usepackage{upref, eucal}
\usepackage[all]{xy}

\newcommand {\nc} {\newcommand}
\newcommand {\enm} {\ensuremath}

\def \d{\delta}

\nc {\bdm} {\begin{displaymath}}
\nc {\edm} {\end{displaymath}}

\newtheorem {theorem} {\bf{Theorem}}[section]
\newtheorem {lemma}[theorem] {\bf Lemma}
\newtheorem {proposition}[theorem] {\bf Proposition}

\newtheorem {corollary}[theorem] {\bf Corollary}
\numberwithin {equation}{section}

\newcommand\BB{\mathbb{B}}\newcommand\DD{\mathbf{D}}\newcommand\FF{\mathbb{F}}\newcommand\QQ{\mathbb{Q}}\newcommand\WW{\mathbb{W}}
\newcommand\ZZ{\mathbb{Z}}

     \newcommand\fp{\mathfrak{p}}

\nc{\J}{\enm{\mathcal{J} }}
\nc {\Z} {\enm{\mathbb{Z}}}
\nc {\form}[1] {\enm{\mbox{\underline{for}}}_{#1}}
\nc {\prol}[1] {\enm{\mbox{\underline{prol}}_{{#1}^*}}}

\nc {\stk} {\stackrel}

\newcommand{\map}{\rightarrow}

\newcommand{\beqar}{\begin{eqnarray*}}
\newcommand{\eeqar}{\end{eqnarray*}}
\newcommand{\inj}{\hookrightarrow}

\newcommand{\Pn}[2] {\ensuremath{ {\mathbb{P}}^{#1}_{#2}}}
\nc{\Quot}[3]{\enm{ {\mathfrak{Quot}_{ {#1}/{#2}/{#3}}}}}
\nc{\Hilb}[2]{\enm{ {\mathfrak{Hilb}_{ {#1}/{#2}}}}}
\newcommand{\mfrak}[1]{\mathfrak{#1}}
\newcommand{\mf}[1]{\mathfrak{#1}}
\newcommand{\bb}[1]{\mathbb{#1}}
\newcommand{\mcal}[1]{\mathcal{#1}}

\nc {\Coh}[4] {\ensuremath{H^{#1}(\Pn{#2}{},{#3}({#4}))}}
\nc {\Ch}[3] {\enm{H^{#1}(X_t,{#2}_t({#3}))}}
\nc {\Qphi}[4]{\enm{ {\mathfrak{Quot}^{~#4}_{ {#1}/{#2}/{#3}}}}}
\nc {\Gra}[4]{\enm{ {\mathfrak{Grass}_{#2}({#3},{#4})}}}
\nc {\HomA}[2]{\enm{\mathrm{Hom}_A{#1}{#2}}}
\nc {\tr}{\mathrm{tr}}

\nc {\C}[2]{\enm{\left(\begin{array}{l} {#1} \\ {#2} \end{array} \right)}}
\nc {\mat}[4]{\enm{\left(\begin{array}{ll}{#1} & {#2} \\ {#3} & {#4}
\end{array}\right)}}

\def \vp{\varphi}
\def \mb{\mbox}

 \def \Z{{\mathbb Z}}

   \def \h{\hat{\ }}

\def \d{\delta}

  \def \bX{{\bf X}} \def \bH{{\bf H}}

\def \bH{{\bf H}} \def \d{\delta}

\def \hG{\hat{\mathbb{G}}_{\mathrm{a}}} 

\def \hA{\hat{A}}

\def \R1{R((q))[q']\h}

\usepackage{xcolor}

\DeclareMathOperator{\Spec}{\mathrm{Spec}}
\DeclareMathOperator{\Spf}{\mathrm{Spf}}
\DeclareMathOperator{\Lie}{\mathrm{Lie}}
\DeclareMathOperator{\rk}{\mathrm{rk}}
\newcommand{\Hom}{\mathrm{Hom}}

\newcommand{\Ext}{\mathrm{Ext}}

\newcommand{\bI}{{\bf I}}

\newcommand{\longlabelmap}[1]{{\,\buildrel #1\over\longrightarrow\,}}
\newcommand{\longmap}{{\,\longrightarrow\,}}
\def\longisomap{{\,\buildrel \sim\over\longrightarrow\,}} 

\newcommand{\oldmarginpar}[1]{}

\newcommand{\ainj}[1]{\ar@{^{(}->}[{#1}]}
\newcommand{\asurj}[1]{\ar@{->>}[{#1}]}
\newcommand{\Hcr}{\mathbf{H}_{\mathrm{cris}}}
\newcommand{\Fcr}{F_{\mathrm{cris}}}
\newcommand{\fra}{\mathfrak{f}}
\newcommand{\cblue}[1]{\color{black}{#1}\color{black}}
\newcommand{\bHd}[1]{\mathbf{H}^\delta_{[{#1}]}}

\newcommand{\bDd}[1]{\mathbf{D}_{[{#1}]}}

\usepackage[margin=90pt]{geometry}
\usepackage{verbatim}
\usepackage[pagebackref]{hyperref}

\usepackage{hyperref}
\hypersetup{
	colorlinks   = true,
	citecolor    = teal,
	linkcolor    = purple
}

\usepackage{autonum}
\def \d{\delta}
\def \hGm{\hat{\mathbb{G}}_{\mathrm{m}}}
\usepackage{array}
\usepackage{soul}
\usepackage{color}
\usepackage[normalem]{ulem}

\newcommand{\mff}{\mfrak{f}}

\nc{\Wm}[1]{W_{[m]#1}}

\nc{\mt}{\bold{t}}
\nc{\mF}{\mathfrak{F}}
\nc{\bx}{\bold{x}}
\nc{\by}{\bold{y}}
\nc{\bz}{\bold{z}}
\nc{\ba}{\bold{a}}
\nc{\phm}{\Phi_{[m]}}
\nc{\lan}{\langle}
\nc{\Fm}{F_{[m]}}
\nc{\ran}{\rangle}
\nc{\Fd}{F^\dagger}
\nc{\Rd}{R^\dagger}
\nc{\mlow}{m_{\mathrm{l}}}
\nc{\mup}{\beta_0}

\nc{\ord}{\mb{ord }}
\nc{\bXp}{\bX_{\mathrm{prim}}}
\nc{\bPsi}{\bold{\Psi}}
\nc{\mult}{\mathrm{mult}}
\nc{\mbB}{\mathbbm{B}}
\nc{\mfor}[1]{{#1}^{\mathrm{for}}}
\nc{\latder}{\tilde{\partial}}
\nc{\bg}{{\bold h}}
\nc{\bff}{{\bold f}}
\nc{\bu}{{\bold u}}
\nc{\mO}{\mathcal{O}}
\nc{\tTheta}{\tilde{\Theta}}
\nc{\io}{\iota}
\nc{\cb}{\mathbf{b}}
\nc{\tF}{\tilde{F}}
\nc{\mg}{\mathfrak{g}}
\nc{\bc}{\mathbf{c}}
\nc{\Witt}{\text{$\mathbb{W}$itt}}
\nc{\mW}{\mathbb{W}}
\nc{\mH}{\mathbf{H}}
\nc{\Hdr}{\bH^1_{\mathrm{dR}}}
\nc{\Fc}{F_{\mathrm{cris}}}
\nc {\Hd}{\bH^{\d}}
\nc{\mn}{[m]n}

\nc{\Frac}{\text{Frac}}

\newcommand{\Ker}{\text{Ker}}

\nc{\bHnew}{\tilde{\bH}}
\nc{\bGa}{\bar{\mathbb{G}}_a}
\nc{\Iso}{\mathrm{FIso}}
\usepackage [english]{babel}
\usepackage [autostyle, english = american]{csquotes}
\MakeOuterQuote{"}

\title{Delta Characters and Ranks of Minimal Subisocrystals of Abelian Schemes}
\author{Rajat Kumar Mishra}
\date{}
\email{rmishra@uark.edu}
\address{University of Arkansas, Fayetteville, Arkansas, 72701, USA}

\author{Arnab Saha}
\email{arnab.saha@iitgn.ac.in}
\address{Indian Institute of Technology Gandhinagar, Gujarat 382355, India}



\usepackage{hyperref}
\usepackage{cite}
\begin{document}
	
\begin{abstract}
Using the framework of delta characters on arithmetic jet spaces, we construct a canonical sequence of filtered isocrystals associated to an abelian scheme $A$ of relative dimension $g$ over a $p$-adically complete discrete valuation ring. These isocrystals map injectively via a canonical morphism into the first (rational) crystalline cohomology group $\mathbf{H}^1_{\mathrm{cris}}(A)$ of the abelian scheme $A$, directly generalizing a prior construction due to Borger and Saha. Pandit showed that the canonical morphism is a map of isocrystals. This then implies that the image of this canonical embedding coincides with the minimal subisocrystal $F_{\mathrm{cris}}(\Omega_A)$ containing the space of invariant differential forms $H^0(A,\Omega_A)$. 

Berthelot and Ogus showed that the dimension of $F_{\mathrm{cris}}(\Omega_A)$ is $g$ if and only if $A$ has an endomorphism that lifts the absolute Frobenius on the special fiber of $A$. As a primary application, for a general abelian scheme $A$, we show that the minimal subisocrystal $F_{\mathrm{cris}}(\Omega_A)$ can be viewed as a module obtained from primitive delta characters of $A$ via appropriate twists by the canonical Frobenius morphism $\phi$ induced from the theory of arithmetic jet spaces. As a consequence, we show that the rank of $F_{\mathrm{cris}}(\Omega_A)$ depends on numerical invariants arising from the theory of delta characters of $A$. More precisely, the dimension of $F_{\mathrm{cris}}(\Omega_A)$ is  given by the sum of the exact orders of the primitive delta characters associated to the abelian scheme $A$. 
\end{abstract}

	\maketitle
	\section{Introduction}
Let $p$ be a prime number, and let $R = W(\mathbb{F}_q)$ be the ring of $p$-typical Witt vectors of a finite field $\mathbb{F}_q$ of cardinality $q=p^h$ for some $h$. Let $K = \mathrm{Frac}(R)$ denote its fraction field. Let $A$ be an abelian $p$-formal scheme of relative dimension $g$ over $S$ 
where $S = \Spf R$. For any $R$-module $M$, denote $M_K = M \otimes_R K$. Consider the first 
(rational) crystalline cohomology $\Hcr^1(A)_K$ of $A$
and the associated Hodge filtration $H^0(A,\Omega_A)_K \subset \Hcr^1(A)_K$. By functoriality,
the absolute $q$-power Frobenius on the special fiber $\bar{A}$ of $A$ induces an operator 
$\Fcr$ on $\Hcr^1(A)$. Berthelot and Ogus, in their fundamental Theorem 3.17 in \cite {BeOg}, 
showed that the 
operator $\Fcr$ preserves the subspace $H^0(A,\Omega_A)$ if and only if $A$ admits an 
endomorphism which lifts the absolute $q$-power Frobenius on the special fiber. Such an abelian scheme
$A$ admitting a lift of the absolute Frobenius is said to have a {\it canonical lift}. In the
case when $A$ has ordinary reduction, $A$ has a canonical lift precisely when its 
Serre-Tate parameter satisfies $q_{ij}(A) =1$ for all $i$ and $j$.

The subspace $H^0(A,\Omega_A)_K$ under all the compositional powers of the operator $\Fcr$ 
generates a subisocrystal $\Fcr(\Omega_A)$ of $\Hcr^1(A)_K$. Hence $\Fcr(\Omega_A)$ is the 
smallest subisocrystal of $\Hcr^1(A)_K$ containing $H^0(A,\Omega_A)_K$. In this article,
one of our main results computes the rank $\rk_K \Fcr(\Omega_A)$ in terms of numerical 
invariants arising from the theory of delta characters of $A$. 
We do so by giving a character-theoretic interpretation of $\Fcr(\Omega)$ for any
$\pi$-formal abelian scheme $A$ over $S$.

To elaborate further, the arithmetic jet spaces $J^nA$ of $A$, for at least $n \geq 2$, 
admit $\pi$-formal group scheme morphisms  to the additive $\pi$-formal group scheme $\hG$, 
which are called {\it delta characters} of $A$. 
The space of delta characters of \(A\), denoted by \(\mathbf X_\infty(A)\), naturally becomes an \(R\{\phi^*\}\)-module, where \(R\{\phi^*\}\) is the noncommutative polynomial ring with the relation \(\phi^*r=\phi(r)\phi^*\) for all $r \in R$. Here \(\phi^*\) is the operator on delta characters induced by the canonical Frobenius morphisms \(\phi:J^{n+1}A\to J^nA\).

A delta character $\Theta \in 
\bX_\infty(A)$ is said to have {\it exact order} $o$ if $o$ is the least non-negative integer 
such that $\Theta$ can be viewed as an additive character $\Theta:J^oA \map \hG$.
By \cite{Bu2} and \cite{PaSa}, $\bX_\infty(A)$ is generated by $g$ delta characters 
$\Theta_1,\dots , \Theta_g$ of exact orders $o_1,\dots, o_g$ respectively, where the 
exact orders of $\Theta_i$ satisfy $o_i \leq g+1$ for all $i=1,\dots ,g$. The invariant
defined as $\mup = \max\{o_1,\dots,o_g\}$ is called the upper splitting number of $A$.
The association $A \mapsto (o_1,\dots o_g) \in \ZZ_{\geq 1}^g$ with the normalization
$o_1\leq \cdots \leq o_g$ can be shown to be 
invariant under isogeny. It depends on the deformation theoretic data of 
$A$ viewed as a lift of its closed fiber $\bar{A}$. For example, $A$ has a canonical lift
if and only if $o_1 = \cdots = o_g =1$. Hence Berthelot and Ogus' result says that 
$\rk_K \Fcr(\Omega_A) = g$ if and only if $A$ has canonical lift.

One of our main results which is Theorem \ref{rk_H_delta-intro},
computes the rank of $\Fcr(\Omega_A)$ for a general abelian scheme 
$A$ over $S$
\begin{align}
\label{rkFcr1}
\rk_K \Fcr(\Omega_A) = o_1 + \cdots + o_g.
\end{align}
We briefly outline our methods here.
By the classical theory of vectorial extensions of group schemes due to Mazur and Messing \cite{Mazur} and the theory of generalized arithmetic 
	jet space kernels, we associate to $A$ a canonical system of filtered isocrystals 
	$\left\{ \bHd{m}(A)_K \supset \bDd{m}(A)_K \right\}_{m=0}^\infty$
over $K$. Moreover, for each $m$, there exists a transition map $\Phi^*_{[m]}: \bHd{m}(A)_K
\map \bHd{m+1}(A)_K$ which is an isomorphism of $K$-modules and is 
a map of filtered isocrystals as well.

	For $m=0$, this structure recovers the filtered module 
	introduced by Borger and Saha in \cite{BoSa1}, where 
	$\bDd{0}(A)_K = \bXp(A)_K \simeq K\langle \Theta_1, \dots, \Theta_g \rangle$ is a $g$-dimensional $K$-vector space spanned by the classes of primitive delta characters $\Theta_1, \dots, \Theta_g$ of $A$ of exact orders $o_1, \dots, o_g$. 
Also there exists a canonical map of $K$-modules $\Phi: \bHd{0}(A)_K \map \Hcr^1(A)$
that preserves the natural filtrations on both sides.
	Pandit in \cite{Pandit} shows that the  map $\Phi$
	extends to a morphism of filtered isocrystals.

As a consequence of our Theorem \ref{delta_prim_final},
 we obtain that the  map $\Phi$ induces an isomorphism between $\bHd{0}(A)_K$ and the 
minimal subisocrystal $\Fcr(\Omega_A)$. This is also shown independently by Pandit in
\cite{Pandit}. For all $m \geq \mup - 1$, we prove that
$\rk_K \bDd{m}(A)_K = o_1 + \cdots  + o_g$. We then show that the system of filtered 
isocrystals stabilizes in the following sense: for all $m \geq \mup -1$, we have 
$\bDd{m}(A)_K = \bHd{m}(A)_K$. Hence combining this with the above discussion we obtain
the isomorphism of $K$-modules $\Fcr(\Omega_A) \simeq \bDd{m}(A)_K$ for all $m \geq \mup-1$
and hence \eqref{rkFcr1} follows. We will now explain our results in greater detail.

\subsection{Geometric background and notations}
Inspired by differential algebraic geometry, Buium introduced the theory of arithmetic jet spaces in the category of $p$-adic formal schemes \cite{Bu2}. He studied additive characters of the arithmetic jet spaces associated with abelian schemes over a $p$-adic field; we refer to these as delta characters. Arithmetic jet spaces and delta characters encode arithmetic information about abelian varieties. In particular, the study of torsion points through delta characters and their kernels led to Buium's proof of the effective Manin--Mumford conjecture for curves \cite{Bu2}. More recently, these methods have contributed to explicit bounds for the Mordell--Lang and Zilber--Pink loci obtained by Dogra and Pandit \cite{DP25, DP26} and Dogra and Saha 
\cite{DS26}, and to a higher-dimensional unramified Manin--Mumford bound obtained by Miller and Morrow using $p$-adic methods \cite{MM25}. The theory has also been developed in the context of modular forms \cite{Bu3, Bar, BuSa1, BuSa2, Hurlburt} and applied by Buium and Poonen to establish finiteness results for modular--elliptic correspondences \cite{Bu-Po}. A systematic exposition is given in \cite{Bu4}, while further developments appear in \cite{BoGu, Bu-Mi-1, Bu-Mi-3}. 


Let $\mathcal{O}$ be a Dedekind domain of characteristic 0 with finite residue fields, and fix a nonzero prime ideal $\mathfrak{p} \subset \mathcal{O}$. Let $\kappa$ be the residue field at $\mathfrak{p}$ of cardinality $q=p^h$, and let $\pi$ be a uniformizer of the localized ring $\mathcal{O}_\mathfrak{p}$. Let $R$ be an $\mathcal{O}_\mathfrak{p}$-algebra constituting a $\pi$-adically complete discrete valuation ring with a perfect residue field, and assume the ramification index satisfies $v_\pi(p) \le p-2$. Fix a lift $\phi: R \to R$ of the $q$-power Frobenius, meaning $\phi(a) \equiv a^q \pmod{\pi R}$. We retain the notation $K = \mathrm{Frac}(R)$ and $S = \mathrm{Spf} R$. For any $R$-module $M$, we set $M_K := M \otimes_R K$ and $M_\phi := M \otimes_{R,\phi} R$.
	
	For any $\pi$-formal scheme $X$ over $S$, one defines the $n$-th arithmetic jet 
functor as
$$J^n X(B) := X(W_n(B))$$ 
where for any $\pi$-adically complete $R$-algebra $B$, 
$W_n(B)$ is the ring of $\pi$-typical Witt vectors of length $n+1$
	\cite{Bo1, Drn, Joyal}.
	This functor is representable by a $\pi$-formal $S$-scheme $J^n X$.
The canonical truncation map $T$ and Frobenius map $F$ on 
	Witt vectors naturally induce projection morphisms $u: J^n X \to J^{n-1} X$ and Frobenius morphisms $\phi: J^n X \to J^{n-1} X$ respectively, yielding the canonical prolongation sequence $J^* X = \{J^n X\}_{n=0}^\infty$.
	
Note that when $G$ is a $\pi$-formal group scheme over $S$, then by definition $J^nG$ has a
natural structure of a group object.
The $n$-th iteration of the projection map, $u^{\circ n}$, induces a surjective projection $u_m: J^{m+n}G \to J^m G$ for all $m,n \ge 0$. We denote the generalized kernel of this map by $N^{[m]n}G := \mathrm{Ker}(u_m)$. If we further assume that $G$ is a smooth $\pi$-formal scheme
over $S$, then we have the following short exact sequence of $\pi$-formal group schemes:
	\begin{equation}
		\label{kerm1}
		0 \longrightarrow N^{[m]n}G \xrightarrow{\iota_m} J^{m+n}G \xrightarrow{u_m} 
		J^m G \longrightarrow 0.
	\end{equation}
	
	By Theorem $1.3$ in \cite{PaSa}, the system $N^{[m]*}G := 
	\{N^{[m]n}G\}_{n=0}^\infty$ forms a prolongation sequence naturally isomorphic to $J^*(N^{[m]1}G)$. Furthermore, the associated Frobenius morphisms $\mathfrak{f}_m: N^{[m]n}G \to N^{[m](n-1)}G$ obey the crucial compatibility condition 
\begin{align}
\label{phiiota}
\phi^{\circ(m+n)} \circ \iota_m = \phi^{m+n-1} \circ \iota_m \circ \mathfrak{f}_m.
\end{align}
	
	The $R$-module of delta characters of $G$ of order at most $n$ is defined by $\mathbf{X}_n(G) := \Hom(J^n G, \widehat{\mathbb{G}}_a)$. By passing to the direct limit over $u^*$ pullbacks, we form $\mathbf{X}_\infty(G)$, upon which the pullback by $\phi^*$ acts as a semilinear operator. The module of primitive delta characters is defined as the quotient $\mathbf{X}_{\mathrm{prim}}(G) := \mathbf{X}_\infty(G) / \phi^* \mathbf{X}_\infty(G)$, which is a free $R$-module of rank $g$ admitting a basis derived from primitive delta characters. 
	
	For each $n \ge 1$, we define the quotient module:$$\mathbf{H}^\delta_{[m]n}(G) := \frac{\Hom(N^{[m]n}G, \widehat{\mathbb{G}}_a)}{((\phi^{*(m+1)} \circ \iota_m)^* \mathbf{X}_{n-1}(G))_{\phi^{m+1}}}$$and its limit $\mathbf{H}^\delta_{[m]}(G) := \varinjlim_{u^*} \mathbf{H}^\delta_{[m]n}(G)$. The pullback via $\mathfrak{f}_m$ induces a well-defined $R$-semilinear operator $\mathfrak{f}_m^*: \mathbf{H}^\delta_{[m]}(G) \to \mathbf{H}^\delta_{[m]}(G)$.
	
Let $A$ be an abelian $\pi$-formal scheme of relative dimension $g$ over $S$.
Applying the functor $\Hom(-,\hG)$ to \eqref{kerm1}, we obtain the following long exact 
sequence of $R$-modules
\begin{align}
0 \map \Hom(J^mA,\hG) \stk{u_m^*}{\longrightarrow} \Hom(J^{m+n}A,\hG) \stk{\iota_m^*}{\longrightarrow} 
\Hom(N^{[m]n}A,\hG) \stk{\partial}{\map} \Ext(J^mA,\hG).
\end{align}
where $\partial: \Hom(N^{[m]n}A, \widehat{\mathbb{G}}_a) \to \Ext(J^m A, \widehat{\mathbb{G}}_a)$ is the natural connecting map. By \eqref{phiiota}, the map $\partial$ descends to a 
map $\partial :\mathbf{H}_{[m]n} \map \Ext(J^mA,\hG)$ of $R$-modules. Hence passing to the 
limit we obtain a map of $R$-modules $\partial: \mathbf{H}^\delta_{[m]}(A) \to \Ext(J^m A, \widehat{\mathbb{G}}_a)$, producing the exact sequence of $R$-modules:
$$0 \longrightarrow \mathbf{D}_{[m]}(A) \longrightarrow \mathbf{H}^\delta_{[m]}(A) \xrightarrow{\,\partial\,} \mathbf{I}_m(A) \longrightarrow 0$$
where $\mathbf{I}_m(A) := \mathrm{image}(\partial)$ and $\mathbf{D}_{[m]}(A) := \mathrm{kernel}(\partial)$. 

A {\it filtered module} over $K$ is a $K$-module $M$ endowed with a system of decreasing 
$K$-submodules $M^\bullet =\{M^i\}_{i \in \ZZ}$, that is $M^{i+1} \subset M^i$ for all
$i \in \ZZ$. The category of such objects will be denoted by $\mathrm{Fil}_K$. We define
a {\it(non-degenerate) isocrystal} over $K$ to be a finite dimensional $K$-vector space 
$D$ equipped with a (bijective) Frobenius-semilinear endomorphism $\phi_D: D \map D$.
A {\it filtered isocrystal} over $K$ is a triple $(D,\phi_D,D^\bullet)$, where $(D,\phi_D)$
is a non-degenerate isocrystal over $K$ and $(D,D^\bullet)$ is an object in $\mathrm{Fil}_K$.

Consider the filtered isocrystal
$$
\mathrm{Fiso}(\bH^\delta_{[m]}(A)):=(\bH^\delta_{[m]}(A)_K,\fra^*_m,\bH^\delta_{[m]}(A)^\bullet_K),
$$
where $\bH^\delta_{[m]}(A)^\bullet_K$ is the filtration given by $\bH^\delta_{[m]}(A) 
\supset \DD_{[m]}(A)$.

Assume $R=W(\kappa)$ to be the infinite length $\pi$-typical Witt vector of the residue field
$\kappa$. We prove in Theorem \ref{non-deg}
that upon tensoring with $K$, the operator $\mathfrak{f}_m^*$ is bijective, making $\mathbf{H}^\delta_{[m]}(A)_K$ a non-degenerate isocrystal. Hence 
$\mathrm{Fiso}(\bH^\delta_{[m]}(A)$ is a filtered isocrystal for all $m$. 

\subsection{Statement of main results}

A fundamental result connecting the levels of these jet extensions relies on a canonical morphism $\Phi_{[m]}: N^{[m]*}A \to N^{[m-1]*}A$ constructed 
	in Section $5$ of \cite{Mi-Sa-1}, which generalizes the map of multiplication-by-$\pi$ on Witt vectors. This morphism  intertwines with the Frobenius structures via 
	\begin{equation}
		\label{phmfra1}
		\mathfrak{f}_{m-1} \circ \Phi_{[m]} = \Phi_{[m]} \circ \mathfrak{f}_m, 
	\end{equation}
	thus inducing a natural isocrystal morphism $\Phi_{[m]}^*: \mathbf{H}^\delta_{[m-1]}(A)_K \to \mathbf{H}^\delta_{[m]}(A)_K$.

In Theorem 
	\ref{phiExt-intro-1}, we prove that $\partial$  respects $\Phi_{[m]}^*$ and the canonical pullback $\phi^*$ on extensions.
	
	\begin{theorem}
		\label{phiExt-intro-1}
		For any abelian scheme $A$ over $S$ and any $n \ge 0$, the following diagram of $R$-modules commutes:$$\xymatrix{
			\Hom(N^{[m]n}A,\widehat{\mathbb{G}}_a) \ar@{->>}[r] & \mathbf{H}^\delta_{[m]}(A) \ar[r]^-\partial & \Ext(J^mA,\widehat{\mathbb{G}}_a)\\
			\Hom(N^{[m-1]n}A,\widehat{\mathbb{G}}_a) \ar@{->>}[r] \ar[u]^{\Phi^*_{[m]}} & \mathbf{H}^\delta_{[m-1]}(A) \ar[u]^{\Phi_{[m]}^*} \ar[r]^-\partial & \Ext(J^{m-1}A,\widehat{\mathbb{G}}_a) \ar[u]_{\phi^*}.
		}$$
	\end{theorem}
	Combining Theorem \ref{phiExt-intro-1} and equation (\ref{phmfra1}), 
	$\phm^*$ hence induces a map of filtered isocrystals
	$$
	\phm^* : \Iso(\bHd{m-1}(A)) \map \Iso(\bHd{m}(A)),
	$$
that is, we have the following morphism of short exact sequences of $K$-modules
\begin{align}
\label{ComFiliso}
\xymatrix{
0 \ar[r] & \mathbf{D}_{[m]}(A) \ar[r] & \mathbf{H}^\delta_{[m]}(A) \ar[r]^{\partial} & 
\mathbf{I}_{m}(A) \ar[r] & 0\\
0 \ar[r] & \mathbf{D}_{[m-1]}(A) \ar[u]^{\phi^*} \ar[r] & \mathbf{H}^\delta_{[m-1]}(A) 
\ar[u]^{\Phi^*_{[m]}} 
\ar[r]^{\partial} & \mathbf{I}_{m-1}(A) \ar[u]_{\phi^*} \ar[r] & 0.
}
\end{align}

Theorem \ref{intro-Phm-1} characterizes the 
transition maps $\phm^*$ across the exact sequences of filtered isocrystals:
	
	\begin{theorem}
		\label{intro-Phm-1}
		For all $m \ge 1$, the following structural properties hold:
		\begin{enumerate}
			\item 
			The total map $\Phi_{[m]}^*: \mathbf{H}^\delta_{[m-1]}(A)_K \to 
			\mathbf{H}^\delta_{[m]}(A)_K$ is surjective.
			\item
			The induced map $\phi^*: \mathbf{I}_{[m-1]}(A)_K \to \mathbf{I}_{[m]}(A)_K$ is surjective.
			\item 
			The map $\phi^*: \mathbf{D}_{[m-1]}(A)_K \to \mathbf{D}_{[m]}(A)_K$ is  injective.
			
			\item
			The sequence stabilizes: for all $m \ge \beta_0$, where $\beta_0$ represents the upper splitting number, all vertical transition maps in the diagram \eqref{ComFiliso}
become exact isomorphisms of $K$-modules.
		\end{enumerate}
	\end{theorem}
	
We now turn to the principal application of our results to $p$-adic Hodge theory. Specifically, we compute the rank of the minimal isocrystal within the first crystalline cohomology that contains the invariant differentials of $A$, expressing this rank explicitly in terms of the exact orders of a basis of primitive delta characters of $A$.
	
	Before establishing this result, it is instructive to briefly recall the comparison theorems that bridge the delta-geometric objects introduced above with the first crystalline cohomology of the abelian scheme $A$ over $S$.
	
Assume $R=W(\kappa)$. In \cite{BoSa1}, Section 8, Borger and Saha constructed a canonical morphism of $K$-modules $\Phi: \bHd{0}(A)_K \map \Hcr^1(A)_K$, which induces a natural morphism of short exact sequences:
	$$
	\xymatrix{
		0 \ar[r] & \bXp(A)_K \ar[d]^\Upsilon \ar[r] & \bH^\delta_{[0]}(A)_K \ar[d]^{\Phi} \ar[r] & \bI_{[0]}(A)_K \ar[r] \ainj{d} & 0  \\
		0 \ar[r] & H^0(A,\Omega_A)_K \ar[r] & \Hcr^1(A)_K \ar[r] & H^1(A,\Omega_A)_K \ar[r] & 0,
	}
	$$
	where the bottom row constitutes the classical Hodge filtration sequence associated to $A$. Let $\Fcr(\Omega_A)$ denote the minimal subisocrystal inside $\Hcr^1(A)_K$ generated by the space of invariant differentials $H^0(A,\Omega_A)_K$ under $\Fcr$. By Theorem 1.1 of \cite{Delta-crys}, the restriction $\Upsilon$ is an isomorphism of $K$-modules, and hence implies
 $\Phi$ is injective.
	
	Furthermore, Theorem 7.6 of \cite{BoSa1} provides an  isomorphism of $K$-vector spaces $\bXp(A)_K \simeq K\langle \Theta_1,\dots, \Theta_g \rangle$, where each $\Theta_i$ for $i=1, \dots, g$ is a primitive delta character of $A$ of exact order $o_i$ satisfying the uniform bound $1 \leq o_i \leq g+1$. The assignment 
	$$
	\{\text{Abelian schemes of relative dimension } g \text{ over } S\} \longrightarrow \mathbb{N}^g
	$$
	defined by $A \mapsto (o_1,\dots, o_g)$ with the normalization $o_1 \leq o_2 \leq \dots \leq o_g$, is well-defined; that is, the tuple of exact orders is entirely independent of the chosen basis of primitive characters in $\bXp(A)$. Moreover, this assignment constitutes an isogeny invariant of $A$. 
	
	By composing the aforementioned isomorphism with $\Upsilon$, one obtains a canonical identification:
	$$
	H^0(A,\Omega_A)_K \simeq \bXp(A)_K \simeq K\langle \Theta_1,\dots, \Theta_g\rangle.
	$$
	This yields a deep character-theoretic interpretation of $H^0(A,\Omega_A)_K$. Under this correspondence, if an invariant $1$-form $\omega \in H^0(A,\Omega_A)_K$ is identified with a delta character $\Theta_\omega$, the map $\omega \mapsto o(\Theta_\omega) \in \mathbb{N}$ (where $o(\Theta_\omega)$ denotes the exact order of $\Theta_\omega$) endows the differential form $\omega$ with a canonical numerical invariant.
	
Pandit showed in \cite{Pandit}, Theorem 1.2, that $\Phi$ is a morphism of isocrystals and
hence becomes a morphism in the category of filtered isocrystals. Here, the target $\Hcr^1(A)_K$ is naturally equipped with the crystalline Frobenius operator $\Fcr$ and the descending Hodge filtration whose first step is $H^0(A,\Omega_A)_K$. As a 
 consequence of Theorem \ref{delta_prim_final}, we obtain that
 $\Phi$ induces an isomorphism of isocrystals:
\begin{align}	
\label{HdFcrO}
	\Phi: \bHd{0}(A)_K \longisomap \Fcr(\Omega_A).
\end{align}
The above is also independently shown by Pandit in \cite{Pandit}.	   
	Our subsequent main result, which will be proved as Corollary \ref{rkFcris}, 
	resolves the rank of this minimal isocrystal completely in terms of the 
	arithmetic differential invariants of $A$.
	
	\begin{theorem}
\label{rk_H_delta-intro}
		Let $A$ be an abelian scheme of relative dimension $g$ over $S$. For all integers $m\geq 0$, we have
		\begin{align}
			\rk_K(\mH^\delta_{[m]}(A)_K) = o_1+\cdots+o_g.  
		\end{align} 
		In particular, by \eqref{HdFcrO}, we have
		\begin{align}
			\rk_K(\Fcr(\Omega_A)) = o_1+\cdots+o_g.  
		\end{align} 
	\end{theorem}

	\section{Plan of the paper}
	
	In Section 3, we establish the foundational geometric framework by reviewing the theory of delta geometry. We give a brief overview of arithmetic jet spaces over $p$-adically complete discrete valuation rings and formally define the generalized kernels $N^{[m]n}G$ associated with the projection maps for a smooth commutative group scheme $G$. We then study their geometric and structural properties. This section serves as a preliminary foundation for the theories developed in the subsequent sections, collecting the necessary results from the literature.
	
	The purpose of Sections 4-6 is to establish several crucial structural statements that provide the key input for the rank computation in Section 7.
	
	In Section 4, we focus on the additive characters of these generalized kernels. Utilizing the theory of vectorial extensions, we construct the fundamental exact sequence of $R$-modules that connects delta characters with extensions of jet spaces.
	
	In Section 5, we prove that, upon tensoring with the fraction field $K$, this sequence is endowed with a bijective semilinear Frobenius operator, yielding a non-degenerate filtered isocrystal.
	
	In Section 6, we analyze the transition morphisms $\Phi_{[m]}$ between different levels of the jet extensions. The primary focus of this section is to prove Theorem 1.1 and Theorem 1.2, establishing that these transition maps strictly intertwine with the Frobenius structures and the connecting homomorphisms. We demonstrate the stabilization of this sequence of filtered isocrystals for sufficiently large $m$.  
	
	In Section 7, we present our main application to $p$-adic Hodge theory. We bridge our delta-geometric constructions with the first crystalline cohomology of the abelian scheme $A$. By exploiting the properties of the level $m=0$ isocrystal and its relation to invariant differential forms, we conclude by proving Theorem 1.3. This yields an explicit formula for the rank of the minimal subisocrystal $\Fcr(H^0(A,\Omega_A)_K)$ purely in terms of the exact orders of the primitive delta characters of $A$.  
	
	Finally, in Section 8, we illustrate our results in low-dimensional cases by considering elliptic curves and abelian surfaces, and explicitly compute the associated filtered isocrystals in each case.

		%

	\section{Preliminaries on Delta Geometry}\label{AJS}
	Let $\mO$ be a Dedekind domain of characteristic zero with finite residue fields, and let $\mathfrak{p}$ be a non-zero prime ideal with residue field $\kappa$. We assume that the cardinality of $\kappa$ is $q=p^h$ for some prime $p$ and integer $h\geq 1$. Let $\pi$ denote a fixed uniformizer of $\mathfrak{p}$. For an $\mO$-algebra $B$ and a $B$-algebra $A$ with
	$u:B \map A$ the structure map, we define a {\it $\pi$-derivation} 
	$\delta:B \map A$ as a set-theoretic map satisfying the following identities for all $x, y \in B$:
	
	\begin{enumerate}
		\item $\delta(1)=0$
		\item $\delta(x+y)=\delta(x)+\delta(y)+C_\pi(u(x),u(y))$
		\item $\delta(xy)=u(x)^q\delta y+u(y)^q\delta x+\pi\delta x\delta y $
	\end{enumerate}
	
	where 
	$$ C_\pi(X,Y)= \frac{X^q+Y^q-(X+Y)^q}{\pi}.
	$$
	
	Given such a $\pi$-derivation $\delta$, we define the associated lift of Frobenius by $\phi(x) := x^q + \pi\delta x$. This map $\phi$ constitutes a ring homomorphism satisfying the congruence $\phi(x) \equiv x^q \mod \fp$. 
	
	Throughout the following, we fix a $\pi$-adically complete $\mO$-algebra $R$ 
	of characteristic $0$
	equipped with a $\pi$-derivation $\delta$ and a perfect residue field $l$, and we denote $S = \Spf R$.

	\begin{lemma}\label{phi_bij}
		For any $\pi$-adically complete $\mO$-algebra $R$, which is a discrete valuation ring of characteristic 0 with a perfect residue field $l$ and a $\pi$-derivation $\delta$ on it, the endomorphism $\phi:R\rightarrow R$ is an automorphism.
	\end{lemma}

	\begin{proof}
		We need to show that $\phi$ is a bijection. We first establish injectivity.
		Let $\phi(x)=0$ for some $x\in R$. We have $x\in R$, hence $x=y\pi^e$ for some $y\in R^\times$ and $e=v_\pi(x)\geq 0$. If $e=0$, then $x$ is a unit and hence $\phi(x)$ is also a unit. Therefore $\phi(x)=0$ implies $x=0$. Let $e\geq 1$, then we have \begin{align}
			\delta x&= \frac{\phi(x)-x^q}{\pi}
			=\frac{\phi(\pi^ey)-(\pi^ey)^q}{\pi}\\
			&=\frac{\pi^e(\phi(y)-\pi^{eq-e}y^q)}{\pi}
			=\pi^{e-1}(\phi(y)-\pi^{e(q-1)}y^q).
		\end{align} 

		The element $\phi(y)$ is a unit and $R$ is $\pi$-torsion free. Hence $v_\pi(\delta x)=e-1$. We know that $\phi(x)=x^q+\pi\delta x=0$, and this implies $x^q=-\pi\delta x.$ Computing the valuations of both sides we obtain $eq=e$ and we know that $q\geq 1$, which is a contradiction if $x\neq 0$. Hence $x=0$ and that implies $\phi$ is an injection. 
		
		For surjectivity, we will use an inductive argument. Since $l$ is perfect and $\phi$ is an injection,  $\phi:l\simeq R/\pi R\rightarrow  R/\pi R\simeq l$ is an isomorphism. We claim  that $\phi:R/\pi^nR\rightarrow R/\pi^n R$ is an isomorphism and we will prove it by induction. It is clear for $n=1$. Assume that the statement holds for some $n\geq 1.$ Hence we have the following maps of exact sequences
		\begin{equation}\label{phi_R}
			\xymatrix{
				0\ar[r] &\frac{\pi^n R}{\pi^{n+1} R}\ar[r]\ar[d]_{\phi} & \frac{R}{\pi^{n+1} R}\ar[r]\ar[d]_{\phi} & \frac{R}{\pi^n R}\ar[r]\ar[d]_{\phi} & 0\\
				0\ar[r] &\frac{\pi^n R}{\pi^{n+1} R}\ar[r] & \frac{R}{\pi^{n+1} R}\ar[r] & \frac{R}{\pi^n R}\ar[r] & 0.
			}
		\end{equation}
		Note that $\pi^n R/\pi^{n+1} R\simeq R/\pi R\simeq l,$ therefore the map $\phi:\pi^n R/\pi^{n+1} R\simeq l\rightarrow \pi^n R/\pi^{n+1} R\simeq l$ is an automorphism and $\phi:R/\pi^n R\rightarrow R/\pi^n R$ is an automorphism by the induction hypothesis. Hence by the snake lemma,  $\phi:R/\pi^{n+1} R\rightarrow R/\pi^{n+1} R$ is also an automorphism. Because  $R$ is $\pi$-adically complete, $R=\varprojlim_n R/\pi^n R$, and hence $\phi:R\rightarrow R$ is an automorphism. 
	\end{proof}
	\subsection{$\pi$-typical Witt vectors}
	\label{pityp}
	We recall some basic facts about $\pi$-typical Witt vectors. The general theory of Witt vectors over Dedekind domains with finite residue fields was developed by Borger \cite{Bo1}. The case of local fields of characteristic $p$ was studied earlier by Drinfeld \cite{Drn}.
	
	Let $B$ be an $R$-algebra with structure map $R \stk{f}{\map} B$. For $n\geq 0$, let $B^{\phi^n}$ denote the $R$-algebra whose structure map is given by
	$R \stk{\phi^n}{\map} R \stk{f}{\map} B$.
	Given an $S$-scheme $X$, we define its $\phi^n$-twist by
	$X^{\phi^n}(B):=X(B^{\phi^n})$
	for every $R$-algebra $B$. This functor is represented by the base change of $X$ along the morphism $\phi^n:S\map S$; that is,
	$X^{\phi^n}=X\times_{S,\phi^n}S$.
	
	We define the \emph{ghost rings}
	$\prod_\phi^n B := B\times B^\phi\times\cdots\times B^{\phi^n}$
	and
	$\prod_\phi^\infty B := B\times B^\phi\times\cdots$.
	For $n\geq 1$, consider the \emph{restriction} (or \emph{truncation}) map
	$T_w:\prod_\phi^n B\map\prod_\phi^{n-1}B$
	given by
	$T_w(w_0,\ldots,w_n):=(w_0,\ldots,w_{n-1})$.
	
	We also consider the left-shift \emph{Frobenius} map
	$F_w:\prod_\phi^n B\map\prod_\phi^{n-1}B$
	defined by
	$F_w(w_0,\ldots,w_n):=(w_1,\ldots,w_n)$.
	Note that $T_w$ is an $R$-algebra morphism, whereas $F_w$ lies over the fixed Frobenius endomorphism $\phi$ of $R$.
	
	As a set, define $W_n(B):=B^{n+1}$. We consider the set-theoretic map
	$w:W_n(B)\map\prod_\phi^n B$
	given by
	$w(x_0,\ldots,x_n):=(w_0,\ldots,w_n)$,
	where, for $i\geq 0$,
	
	$$
	w_i=x_0^{q^i}+\pi x_1^{q^{i-1}}+\cdots+\pi^i x_i
	$$
	
	are the \emph{Witt polynomials}. We refer to $w$ as the \emph{ghost map}.
	
	The ring structure on $\pi$-typical Witt vectors is characterized by the following result, analogous to the classical construction in the $p$-typical case (cf. \cite[pp.~171--192]{Mum-66}).
	
	\begin{theorem}
		For every $n\geq 0$, there exists a unique functorial $R$-algebra structure on $W_n(B)$ such that $w$ is a natural transformation of functors of $R$-algebras.
	\end{theorem}
	
	We now recall some of the basic operators on $\pi$-typical Witt vectors.
	
	$(1)$ The \emph{restriction} (or \emph{truncation}) map
	$T:W_n(B)\map W_{n-1}(B)$ is given by
	$T(x_0,\ldots,x_n)=(x_0,\ldots,x_{n-1})$.
	It satisfies
	$w\circ T=T_w\circ w$.
	
	$(2)$ The \emph{Frobenius} map
	$F:W_n(B)\map W_{n-1}(B)$ is characterized by
	$w\circ F=F_w\circ w$ and is given by
	
	$$
	F(x_0,\ldots,x_n)=(x_0^q+\pi x_1,\ldots).
	$$
	
	As in the case of the Frobenius map on the ghost side, $F$ lies over the Frobenius endomorphism $\phi$ of $R$.
	
	We define the ring of Witt vectors of infinite length by
	$W(B):=\varprojlim_T W_n(B)$, where the inverse limit is taken with respect to the truncation maps $T$.
	
	The Frobenius maps $F$ induce a Frobenius lift on $W(B)$. In fact, there exists a unique $\pi$-derivation $\Delta:W(B)\map W(B)$ whose associated Frobenius lift is $F$, namely,
	
	$$
	F(x)=x^q+\pi\Delta(x).
	$$
	
	The ring $W(B)$ is characterized by the following universal property. Let $C$ be an $R$-algebra equipped with a $\pi$-derivation $\delta$ and let $f:C\map B$ be an $R$-algebra homomorphism. Then there exists a unique $R$-algebra homomorphism $g:C\map W(B)$ such that the diagram
	
	$$
	\xymatrix{
		W(B) \ar[d]_T & \\
		B & C \ar[l]^f \ar[lu]^g
	}
	$$
	
	commutes and
	$\Delta\circ g=g\circ\delta$.
	Consequently, $W$ is right adjoint to the forgetful functor from the category of $R$-algebras equipped with a $\pi$-derivation to the category of $R$-algebras. In particular, by the above universal property, $W(B)$ carries a natural structure of an $R$-algebra for every $R$-algebra $B$.
	
	For further details, we refer the reader to Section 1 of \cite{Bo1}. This approach is analogous to Joyal's construction in the case of $p$-typical Witt vectors \cite{Joyal}.

	\subsection{Prolongation sequences} 
	Let $X$ and $Y$ be $\pi$-formal schemes over $S$. We say that a pair $(u,\d)$ is a \emph{prolongation}, and write
	$Y\stk{(u,\d)}{\longrightarrow}X$,
	if $u:Y\map X$ is a morphism of $\pi$-formal $S$-schemes and $\d:\mO_X\map u_*\mO_Y$ is a $\pi$-derivation such that the following diagram commutes:
	
	$$
	\xymatrix{
		R \ar[r] & u_*\mO_Y \\
		R \ar[u]^\d \ar[r] & \mO_X \ar[u]_\d.
	}
	$$
	
	We now recall the notions of prolongation sequences and arithmetic jet spaces over $\pi$-formal schemes, following Buium. As in \cite[p.~103]{Bu3}, a \emph{prolongation sequence} is a sequence
	
	$$
	S\stk{(u_0,\d_0)}{\longleftarrow}T^0
	\stk{(u_1,\d_1)}{\longleftarrow}
	T^1
	\stk{(u_2,\d_2)}{\longleftarrow}\cdots,
	$$
	
	where each morphism
	$T^i\stk{(u_{i+1},\d_{i+1})}{\longleftarrow}T^{i+1}$
	is a prolongation satisfying
	
	$$
	u_i^*\circ\d_{i+1}=\d_i\circ u_{i+1}^*,
	$$
	
	where $u_i^*$ denotes the pullback morphism of sheaves induced by $u_i$. We denote a prolongation sequence by $T^*$ or ${T^n}_{n\geq0}$.
	
	Prolongation sequences naturally form a category: a morphism $f:T^*\to U^*$ is a collection of morphisms $f^n:T^n\to U^n$ compatible with the morphisms $u_i$ and $\d_i$ in the evident sense.
	
	Let $S^*=\{S^i\}_{i\geq0}$ be the prolongation sequence defined by
	$S^i:=\Spf R$, with $u_i:=\mathbbm{1}$ and $\d_i:=\d$, where $\d$ is the fixed $\pi$-derivation on $R$, for every $i$. We denote by $\mcal{C}_{S^*}$ the category of prolongation sequences over $S^*$.
	
	For a $\pi$-formal $S$-scheme $X$ and $n\geq0$, we define the $n$-th arithmetic jet space $J^nX$ by its functor of points:
	
	$$
	J^nX(B):=X(W_n(B))
	=\Hom_S(\Spf(W_n(B)),X),
	$$
	
	for every $\pi$-adically complete $R$-algebra $B$. The functor $J^nX$ is representable by a $\pi$-formal scheme over $S$; representability was established in \cite{Bo2} for schemes over $S=\Spec\mathbb{Z}$ and, more generally, in \cite{BePrSa} for arbitrary prolongation sequences.
	
	The system of $\pi$-formal schemes
	$J^*X:={J^nX}_{n\geq0}$
	forms a prolongation sequence, called the \emph{canonical prolongation sequence} of $X$, as in \cite{Bu3}. In particular, there are canonical morphisms
	$\phi:J^{n+1}X\map J^nX$
	given by the lifts of Frobenius.
	
	Throughout this paper, we use the following notation. For every $n\geq0$, we denote the canonical prolongation morphism
	$J^{n+1}X\longrightarrow J^nX$
	by $u$. More generally, for every $m\geq1$, we denote the natural morphism
	$J^{n+m}X\longrightarrow J^mX$
	by $u_m$.
	
	By Proposition 1.1 of \cite{Bu3}, the canonical prolongation sequence $J^*X$ satisfies the following universal property: for every $T^*\in\mcal{C}_{S^*}$ and every $\pi$-formal scheme $X$ over $S$, there is a natural bijection
	\begin{align}
		\label{univ}
		\Hom_S(T^0,X)=\Hom_{\mathcal{C}_{S^*}}(T^*,J^*X).
	\end{align}
	
	\subsection{Local Witt coordinates}
	\label{localcoordinates}
	Consider the $\pi$-formal affine line $H=\hat{\bb{A}}^1\simeq\Spf R[x]\h$. For any $\pi$-adically complete $R$-algebra $B$, the adjunction between jet spaces and the Witt vector functor \cite{BePrSa,Bo2} gives
	$$
	J^nH(B)=\Hom_R(R[x]\h,W_n(B))
	\simeq \Hom_R(R[x,x_1,\dots,x_n]\h,B).
	$$
	Explicitly, if $f\in\Hom_R(R[x]\h,W_n(B))$ is given by $f(x)=(b_0,\dots,b_n)$, with $b_i\in B$ for all $i$, then the corresponding morphism $\tilde f\in\Hom_R(R[x,x_1,\dots,x_n]\h,B)$ is determined by $\tilde f(x_i)=b_i$ for every $i$. Under the resulting identification
	$J^nH\simeq\Spf R[x,x_1,\dots,x_n]\h$, the coordinate functions $x,x_1,\dots,x_n$ are referred to as the \emph{Witt coordinates}.
	\cblue
	
	Suppose now that $G$ is a smooth commutative $\pi$-formal group scheme over $S$ of relative dimension $g$. There exists an affine open subscheme $U\subset G$ containing the identity section $e$ and an étale morphism of $\pi$-formal schemes $U\rightarrow H$, where
	$H=\hat{\bb{A}}^g=\Spf R[\bx_0]\h$
	and $\bx={x_{10},\dots,x_{g0}}$ is a system of coordinate functions. Proposition 1.4 of \cite{Bu2} gives an isomorphism of $\pi$-formal schemes
	$$
	J^nU\simeq U\times_HJ^nH
	$$
	for every $n$. From the description above, we have
	$J^nH\simeq\Spf R[\bx,\bx_1,\dots,\bx_n]\h$,
	where $\bx_i={x_{1i},\dots,x_{gi}}$ is a system of coordinates for each $i$. Consequently,
	$$
	J^nU\simeq U\times_HJ^nH
	\simeq U\times_H\Spf R[\bx,\dots,\bx_n]\h.
	$$
	It follows that $\{\bx_i\}_{i=0}^n$ gives an étale coordinate system on $J^nG$ in a neighborhood of the identity section. The space $N^{[m]n}G$ can now be realized as the following fiber product:
	$$\xymatrix{ N^{[m]n}G=J^{m+n}U\times_{J^mU}S\ar[r]^-{}\ar[d] & J^{m+n}U \ar[d]\\ S \ar[r]^-{e} & J^mU , }$$
	where the fiber product is taken with respect to the identity section $S\xrightarrow{e}J^mU$. Composing this section with the natural morphism $J^mU\rightarrow J^mH$, we obtain a section
	$S\xrightarrow{e}J^mU\rightarrow J^mH$
	of $J^mH$. On coordinate rings, this is induced by the $R$-algebra homomorphism
	$R[\bx,\bx_1,\dots,\bx_m]\h\rightarrow R$
	which sends $\bx_i$ to $0$ for every $0\leq i\leq m$. Therefore,
	$$
	\begin{aligned}
		N^{[m]n}G
		=J^{m+n}U\times_{J^mU}S
		&=(J^{m+n}H\times_{J^mH}J^mU)\times_{J^mU}S\\
		&=J^{m+n}H\times_{J^mH}S\\
		&\cong\Spf(R[\bx_{m+1},\dots,\bx_{m+n}]\h).
	\end{aligned}
	$$
	Thus, the functions ${\bx_{m+1},\dots,\bx_{m+n}}$ provide a coordinate system on the affine $\pi$-formal group scheme $N^{[m]n}G$.

	For all $m,n \geq 0$, we have the following extension of smooth commutative
	group schemes 
	\begin{align}
		\label{canexact}
		0 \longrightarrow N^{[m]n}G \stk{\iota_m}{\longrightarrow} J^{m+n}G \stk{u_m}{\longrightarrow} J^mG \longrightarrow 0
	\end{align}
	where $N^{[m]n}G = \ker( u_m: J^{m+n}G \map J^mG)$ for all $m,n\geq 0$. 
	Note that the morphism $u:J^{m+n}G\rightarrow J^{m-1+n}G$ induces a morphism of $\pi$-formal 
	$S$-schemes (still denoted
	by $u$) $u: N^{[m]n} G \map N^{[m]n-1}G$ for all $n \geq 2$. 
	
	Consider the system of $\pi$-formal group schemes $N^{[m]*}G =\{N^{[m]n}G\}_{n=1}^\infty$.
	By \cite{PaSa}, $N^{[m]*}G$ is a prolongation sequence whose associated lift of Frobenius $\fra_m: N^{[m]n}G \map N^{[m]n-1}G$ with 
	respect to $u$ is also called the {\it lateral Frobenius}. By 
	Theorem $4.3$ of \cite{BoSa1} or Theorem $5.3$ of \cite{PaSa}
	$\fra_m$ satisfies 
	\begin{align}
		\label{fphi}
		\phi^{\circ (m+1)} \circ \iota = \phi^{\circ m} \circ \iota \circ \fra_m.
	\end{align}
	
	We recall Theorem $1.3$ from \cite{PaSa}:
	
	\begin{theorem}
		\label{NandJ}
		Let $G$ be a $\pi$-formal group scheme over $S$. Then
		$$
		N^{[m]*}G \simeq J^*(N^{[m]1}G),
		$$
		as a canonical isomorphism of prolongation sequences of $\pi$-formal schemes 
		over $S^*$. In particular, for all $n \geq 1$ we have
		$$
		N^{[m]n}G \simeq J^{n-1}(N^{[m]1}G).
		$$
	\end{theorem}
	Hence combining the above theorem and (\ref{fphi}) we obtain the following:
	
	\begin{theorem}
		\label{NandJphi}
		For all $n \geq 2$ we have the following commutative diagram of $\pi$-formal 
		group schemes
		$$
		\xymatrix{
			J^{n-1}(N^{[m]1}G) \simeq N^{[m]n}G \ar[r]^-{\phi^{\circ m} \circ \iota_m} \ar[d]_{\fra_m} & J^{n-1}G 
			\ar[d]^\phi \\
			J^{n-2}(N^{[m]1}G) \simeq N^{[m]n-1}G \ar[r]^-{\phi^{\circ m} \circ \iota_m} & J^{n-2}G.
		}
		$$
	\end{theorem} 
	
Assume $v_\pi(p) \leq p-2$.	
Given that $G$ is a $\pi$-formal group scheme of relative dimension $g$ over $S$, Proposition 2.2 and Lemma 2.3 of \cite{Bu2} we have an isomorphism of $\pi$-formal group schemes $\Psi_1^{[m]}: N^{[m]1}G \rightarrow \hG^g$, defined analytically by
	\begin{align}
		\label{NmGa}
		\Psi_1^{[m]}(\mt) = \frac{1}{\pi^{m+1}} \log_G(\pi^{m+1}\mt).
	\end{align}
We will also use the morphism $\phm:N^{[m]*}G\longrightarrow N^{[m-1]*} G$ constructed in \cite[Theorem 4.4]{Mi-Sa-1}. It is compatible with the natural morphisms to the jet spaces, as expressed by the commutative diagram
	\begin{equation}
		\label{JnNm}	
		\xymatrix{N^{[m]n}G\ar[r]^-{\iota_m}\ar[d]^-{\phm}& J^{n+m}G\ar[d]^{\phi}\\
			N^{[m-1]n}G\ar[r]^{\iota_{m-1}}& J^{n+m-1}G.
		}
	\end{equation}
	
	We now recall the following result from \cite[Theorem 5.7]{Mi-Sa-1}, which describes the above morphism in terms of Witt schemes.
	
	\begin{theorem}
		\label{Psimiso}
Assume $v_\pi(p) \leq p-2$. Let $G$ be a $\pi$-formal smooth group scheme of relative dimension $g$ over 
		$S$. For all $m$ and $n$  we have the following 
		isomorphisms of $\pi$-formal group schemes
		\begin{equation}
			\label{Psimiso1}
			\xymatrix{
				N^{[m]n}G \ar[d]_-\phm \ar[r]_-\sim  &J^{n-1}(N^{[m]1}G)\ar[rr]^-{J^{n-1}(\Psi_1^{[m]})}_\sim \ar[d]^{\phm} & & J^{n-1}(\hG^g) \ar[d]^{J^n(\pi)} \ar[r]_\sim & \WW_{n-1}^g \ar[d]^{\pi}
				\\
				N^{[m-1]n}G \ar[r]^-\sim& J^{n-1}(N^{[m-1]1}G)  \ar[rr]_-{J^{n-1}(\Psi_1^{[m-1]})}^\sim & & J^{n-1}(\hG^g) \ar[r]^\sim & \WW_{n-1}^g, \\
			}
		\end{equation}
where the horizontal maps $J^{n-1}(\Psi^{[m-1]}_1)$
		and $J^{n-1}(\Psi^{[m]}_1)$ are isomorphisms.
	\end{theorem}

	For each index $m$, let
	$\Psi_m:N^{[m]n}G\rightarrow \WW_{n-1}^g$ denote  the composition of the 
	horizontal maps defined in diagram (\ref{Psimiso1}) in Theorem \ref{Psimiso}. 
	
	\begin{corollary}
		\label{phm-pi-com}
		Assume $v_\pi(p) \leq p-2$. There is a canonical commutative diagram of $R$-module homomorphisms:
		$$
		\xymatrix{
			\Hom(\WW_{n-1}^g,\hG)\ar[r]^-{\Psi_{m}^*}_-{\sim} & \Hom(N^{[m]n}G,\hG)\\
			\Hom(\WW_{n-1}^g,\hG)\ar[r]^-{\Psi_{m-1}^*}_-{\sim} \ar[u]^{\pi} & \Hom(N^{[m-1]n}G,\hG)\ar[u]^{\phm^*}.
		}
		$$
		Moreover, the vertical morphisms are injective, and upon tensoring with $K$, all morphisms in the diagram become isomorphisms.
	\end{corollary}
	
	\begin{proof}
		The commutativity of the above diagram and isomorphisms of the horizontal maps follow directly from  Theorem \ref{Psimiso}. Note that the multiplication by $\pi$ map from $\Hom(\WW_{n-1}^g,\hG)\rightarrow \Hom(\WW_{n-1}^g,\hG)$ is an injective homomorphism of $R$-modules. Hence  $\phm^*:\Hom(N^{[m-1]n}G,\hG)\rightarrow \Hom(N^{[m]n}G,\hG)$ is also  injective. When tensored with
		$K$, the multiplication-by-$\pi$ map becomes an isomorphism of $K$-modules. Consequently, the induced morphism
		$$\phm^* : \Hom(N^{[m-1]n}G,\hG)_K \longrightarrow \Hom(N^{[m]n}G,\hG)_K$$
		is also an isomorphism of $K$-modules.
	\end{proof}
	
	We next examine the jet space functor in two basic examples, namely the additive and multiplicative $\pi$-formal group schemes.
	
	\begin{enumerate}
		\item {\bf The additive group scheme $\hG$.}
		Let $\hG=\Spf R[x]\h$ denote the additive $\pi$-formal group scheme over $S$. It follows directly from the functorial description of arithmetic jet spaces that
		\begin{align}
			\label{JnGa}
			J^n(\hG)\simeq\bb{W}_n,
		\end{align}
		where $\bb{W}_n$ denotes the $\pi$-formal affine $(n+1)$-plane $\hat{\bb{A}}^{n+1}$ equipped with the additive group law of the $\pi$-typical Witt vectors.
		
		Under this identification, the projection
		$u:J^n(\hG)\map J^{n-1}(\hG)$
		corresponds to the restriction map $T$, while the Frobenius morphism
		$\phi:J^n(\hG)\map J^{n-1}(\hG)$
		is identified with the Frobenius map $F$ on $\pi$-typical Witt vectors.
		
		\item {\bf The multiplicative group scheme $\hGm$.}
		Now let $\hGm=\Spf R[x,x^{-1}]\h$ be the multiplicative $\pi$-formal group scheme over $S$. Once again, the functor of points description gives
		
		$$
		J^n(\hGm)\simeq\bb{W}_n^\times,
		$$
		
		where $\bb{W}_n^\times$ denotes the subfunctor of $\bb{W}_n$ consisting of multiplicatively invertible elements.
		
		We describe this jet space more explicitly. The morphism
		$\vp:\hGm\map\hGm$ defined by $x\map x^q$ is a lift of Frobenius compatible with the multiplicative group structure. Hence, by the universal property of the canonical prolongation sequence (\ref{univ}), $\vp$ determines a canonical splitting of $\pi$-formal group schemes
		$s:\hGm\map J^n(\hGm)$ of the exact sequence
		\begin{align}
			\xymatrix{ 0 \ar[r] & N^{[0]n}(\hGm) \ar[r]^\iota & J^n(\hGm) \ar[r]^-u & \hGm \ar@/_1pc/[l]_-s \ar[r] & 0. } \end{align}
		Consequently, there is an isomorphism of $\pi$-formal group schemes over $S$
		\begin{align}
			u\times(\mathbbm{1}-s\circ u):
			J^n(\hGm)\simeq\hGm\times N^{[0]n}(\hGm).
		\end{align}
		
		Consider the morphism
		\begin{align}
			\label{N1Ga}
			\frac{1}{\pi}\log_{{\bb{G}}^{\mathrm{for}}_{\mathrm{m}}}
			\circ(\phi\circ\iota):
			N^{[0]1}(\hGm)\longrightarrow\hG,
		\end{align}
		where $\log_{{\bb{G}}^{\mathrm{for}}_{\mathrm{m}}}$ denotes the formal logarithm identifying the formal multiplicative group law ${\bb{G}}^{\mathrm{for}}_{\mathrm{m}}$ with the additive formal group law ${\bb{G}}^{\mathrm{for}}_{\mathrm{a}}$, given by
		
		$$
		\log_{{\bb{G}}^{\mathrm{for}}_{\mathrm{m}}}(T)
		=\sum_{n=1}^{\infty}(-1)^{n+1}\frac{T^n}{n}.
		$$
		
		The morphism (\ref{N1Ga}) is an isomorphism of $\pi$-formal schemes. Therefore, Theorem \ref{NandJphi}, together with (\ref{JnGa}), yields
		
		$$
		N^{[0]n}(\hGm)\simeq\WW_{n-1}
		$$
		
		for every $n\geq1$, as $\pi$-formal group schemes over $S$. It follows that, for every $n\geq0$,
		\begin{align}
			\label{JnGm}
			J^n(\hGm)\simeq\hGm\times\WW_{n-1}
		\end{align}
		as $\pi$-formal group schemes over $S$. Upon evaluation at $\FF_p$, this recovers the familiar decomposition of groups
		$\ZZ_p^\times\simeq\FF_p^\times\times\ZZ_p$.
	\end{enumerate}

	\subsection{Delta characters of group schemes}
	We begin by recalling the basic theory of $\delta$-characters developed in Sections $7$ and $8$ of \cite{BoSa1}, which will be used in the construction of the filtered object $\Hd_{[m]}(G)$ over $K$. The case $m=0$, which we denote by $\mH_\d(G)$, has been studied extensively in \cite{BoSa1, PaSa, Pandit, Delta-crys}.
	
	Let $T^*$ be a prolongation sequence. For an integer $s\geq0$, its \emph{shifted prolongation sequence} is defined by
	$T^{*+s}:={T^{s+n}}_{n=0}^{\infty}$.
	A morphism $\Theta:J^nG\map\hG$ is called a \emph{$\delta$-function on $G$ of order at most $n$}. By the universal property of arithmetic jet spaces, such a morphism determines a morphism of prolongation sequences
	$\Theta:J^{*+n}G\map\hG$.
	
	A \emph{$\delta$-character of order at most $n$} is a $\delta$-function $\Theta:G\map\hG$ of order at most $n$ which is, in addition, a homomorphism of $\pi$-formal group schemes. By the universal property of arithmetic jet spaces, as in (\ref{univ}), this is equivalent to a homomorphism of $\pi$-formal group schemes over $S$
	$\Theta:J^nG\map\hG$.
	We write $\bX_n(G)$ for the group of $\delta$-characters of order at most $n$; thus
	
	$$
	\bX_n(G)=\Hom_S(J^nG,\hG).
	$$
	
	A $\delta$-character $\Theta$ is said to have \emph{exact order} $n$ if
	$\Theta\in\bX_n(G)\backslash u^*\bX_{n-1}(G)$,
	where
	$u^*$ denotes the pullback induced by the canonical prolongation morphism
	$u:J^nG\map J^{n-1}G$.
	We denote the exact order of $\Theta$ by $\mathrm{ord}(\Theta)$.
	
	For an $R$-module $M$, we use the notation
	
	$$
	M_{\phi}:=R\otimes_{\phi,R}M.
	$$
	
	Since $\hG$ naturally has the structure of an $R$-module $\pi$-formal scheme over $S$, each $\bX_n(G)$ inherits a natural $R$-module structure. Moreover, the inverse system
	
	$$
	\cdots\longrightarrow J^{n+1}G\xrightarrow{u}J^nG
	\xrightarrow{u}J^{n-1}G\longrightarrow\cdots
	$$
	
	induces a directed system of $R$-modules under pullback:
	
	$$
	\cdots\xrightarrow{u^*}\bX_n(G)
	\xrightarrow{u^*}\bX_{n+1}(G)
	\xrightarrow{u^*}\cdots.
	$$
	
	We therefore define the module of $\delta$-characters of arbitrary finite order by
	
	$$
	\bX_\infty(G):=\varinjlim\bX_n(G).
	$$
	
	Similarly, pulling back by the Frobenius map $\phi:J^{n+1}G\to J^nG$ induces a Frobenius operator
	$\bX_n(G)\to \bX_{n+1}(G)$. However since $\phi:J^{n+1}G\to J^nG$ is not a morphism over $\Spf R$ but
	instead lies over the Frobenius endomorphism $\phi$ of $\Spf R$, some care is required.
	Consider the relative Frobenius morphism $\phi_{G/R}$, defined to be the unique
	morphism making the following diagram commute:
	$$\xymatrix{
		J^{n+1}G \ar@{.>}^{\phi_{G/R}}[rd] \ar@/^/[rrd]^\phi \ar@/_/[ddr]& & \\
		& J^nG \times_{(\Spf R),\phi} \Spf R \ar[d] \ar[r] & J^nG \ar[d] \\
		& \Spf R \ar[r]_\phi & \Spf R 
	}$$
	Then $\phi_{G/R}$ is a morphism of $\pi$-formal schemes over $\Spf R$.
	
	Let $\Theta:J^nG\to\hG$ be a $\delta$-character. We define $\phi^*\Theta$ as the composite
	\begin{equation}
		J^{n+1}G \longlabelmap{\phi_{G/R}}
		J^nG\times_{(\Spf R),\phi}\Spf R
		\longlabelmap{\Theta\times\mathbbm{1}}
		\hG\times_{(\Spf R),\phi}\Spf R
		\longlabelmap{\iota}\hG,
	\end{equation}
	where $\iota$ denotes the isomorphism of $\pi$-formal schemes over $R$ arising from the fact that $\hG$ descends to $R$ as a $\pi$-formal scheme.
	
	On $B$-valued points, for any $R$-algebra $B$, this construction is given by
	
	$$
	G(W_{n+1}(B))
	\longlabelmap{G(F)}
	G(W_n(B)^\phi)
	\longlabelmap{\Theta_B^\phi}
	B^\phi
	\longlabelmap{b\mapsto b}
	B.
	$$
	
	The resulting composite $G(W_{n+1}(B))\map B$ is an $A$-module morphism, since the canonical map $B^\phi\map B$ is $A$-linear: indeed, the restriction of $\phi$ to $\hA$ is the identity.
	
	Thus, Frobenius induces an additive map
	
	$$
	\phi^*:\bX_n(G)\longmap\bX_{n+1}(G),
	\qquad
	\Theta\longmapsto\phi^*\Theta.
	$$
	
	This map is generally not $R$-linear. After twisting the target, however, we obtain an $R$-linear map
	
	$$
	\bX_n(G)\longmap\bX_{n+1}(G)_\phi,
	\qquad
	\Theta\longmapsto\phi^*\Theta.
	$$
	
	Passing to the direct limit over $n$ gives an $R$-linear endomorphism with twisted target
	
	$$
	\bX_\infty(G)\longmap\bX_\infty(G)_\phi,
	\qquad
	\Theta\longmapsto\phi^*\Theta.
	$$
	
	Consequently, $\bX_\infty(G)$ acquires the structure of a left module over the twisted polynomial ring $R\{\phi^*\}$, with commutation relation
	$$
	\phi^*r=\phi(r)\phi^*.
	$$
	
	We next apply the functor $\Hom(-,\hG)$ to the exact sequence (\ref{canexact}). This yields the exact sequence of $R$-modules
	\begin{align}
		\label{Hom_Ker}
		0\longmap\bX_m(G)
		\stk{u_m^*}\longrightarrow
		\bX_{m+n}(G)
		\stk{\iota_m^*}\longrightarrow
		\Hom(N^{[m]n}G,\hG)
		\stk{\partial}\longrightarrow
		\Ext(J^mG,\hG),
	\end{align}
	where $\Ext(J^mG,\hG)$ parametrizes the isomorphism classes of $\pi$-formal group schemes which are extensions of $J^mG$ by $\hG$, and $\partial$ denotes the associated connecting morphism.
	
	Set
	
	$$
	\bI_{[m]n}(G):=\operatorname{image}(\partial).
	$$
	
	The morphisms
	$u:J^{m+n}G\map J^{m+n-1}G$
	induce, in turn, morphisms between the corresponding kernels
	
	$$
	u:N^{[m]n}G\map N^{[m]n-1}G.
	$$
	
	Consequently, pullback gives an injective $R$-module homomorphism
	
	$$
	u^*:\Hom(N^{[m]n-1}G,\hG)\inj\Hom(N^{[m]n}G,\hG).
	$$
	
	It follows that
	
	$$
	\bI_{[m]n-1}(G)\subseteq\bI_{[m]n}(G).
	$$
	
	We may therefore form the direct limit
	\begin{align}
		\label{bI}
		\bI_{[m]}(G):=\varinjlim\bI_{[m]n}(G).
	\end{align}
	
	\subsection{Functorial transition maps and the direct limit}\label{Transition_maps}
	
	Consider the projection map $u \colon N^{[m]n}G \to N^{[m]n-1}G$ and the lateral Frobenius $\mf{f}_m \colon N^{[m]n}G \to N^{[m]n-1}G$ of the prolongation sequence $N^{[m]*}G$.
	
	Also consider the map $\Phi_{[m]} \colon N^{[m]*}G \to N^{[m-1]*}G$ from the preceding section. These maps induce, by precomposition, corresponding homomorphisms:
	\begin{align}
		u^* &\colon \Hom(N^{[m]n}G,\hG) \to \Hom(N^{[m]n-1}G,\hG), \\
		\mf{f}_m^* &\colon \Hom(N^{[m]n}G,\hG) \to \Hom(N^{[m]n-1}G,\hG), \\
		\phm^* &\colon \Hom(N^{[m-1]n}G,\hG) \to \Hom(N^{[m]n}G,\hG).
	\end{align}
	
	Since the truncation map $u$ commutes with both $\iota_m$ and $\phi$, we have the inclusion
	\begin{align}
		u^*\circ \iota^*\circ (\phi^{*(m+1)}(\bX_{n-1}(G))&=\iota_m^*\circ (\phi^*)^{m+1}\circ u^*(\bX_{n-1}(G))\subseteq \iota_m^*\circ (\phi^{*(m+1)}(\bX_{n}(G)),
	\end{align}
	which shows that $u^*$ induces a well-defined homomorphism of $R$-modules $u^* \colon \mH_{[m]n}^\delta(G) \to \mH^\delta_{[m]n+1}(G)$. We define the  $R$-module $\mH_{[m]}^\delta(G)$ via the direct limit
	\[ \mH_{[m]}^\delta(G):=\varinjlim_{n} \mH_{[m]n}^\delta(G). \]
	
	Similarly, for the lateral Frobenius operator $\fra_m$, it follows from 
	Theorem \ref{NandJphi} that
	\begin{align}
		\mf{f}_m^*\circ \iota_m^*\circ (\phi^{*(m+1)}(\bX_{n-1}(G))&=\mf{f}_m^*\circ \iota_m^*\circ \phi^{*(m)}\circ \phi^*(\bX_{n-1}(G)\\ &=\iota_m^*\circ \phi^{*(m+1)}\circ \phi^*(\bX_{n-1}(G))\\&\subseteq \iota_m^*\phi^{*(m+1)}(\bX_{n}(G)).
	\end{align}
	Consequently, $\mf{f}^*_m$ descends to a semilinear morphism of $R$-modules
	\begin{equation}
		\mf{f}_m^*:\mH_{[m]n}^\delta(G)\rightarrow \mH^\delta_{[m]n+1}(G).
	\end{equation}
	Taking the direct limit over $n$, this map induces a semilinear endomorphism of the limit module:
	\begin{equation}\label{Semilin_end}
		\mf{f}_m^*:\mH^\delta_{[m]}(G)\rightarrow \mH^\delta_{[m]}(G).
	\end{equation}
	
	We next consider the map $\phm^* \colon \Hom(N^{[m-1]n}G, \hG) \to \Hom(N^{[m]n}G, \hG)$. By (\ref{JnNm}), the identity $\phi \circ \iota_m = \iota_{m-1} \circ \phm$ holds, which implies the containment
	\begin{align}
		\phm^*\circ \iota_m^*\circ \phi^{*m}(\bX_{n-1}(G))=\iota_m^*\circ \phi^*\circ \phi^{*m}(\bX_{n-1}(G))\subseteq \iota_m^*\circ \phi^{*(m+1)}(\bX_{n-1}(G)).
	\end{align}
	Thus, $\phm^*$ induces a well-defined map on the quotients:
	\begin{equation}\label{phm*H_delta_mn}
		\phm^*:\mH_{[m-1]n}^\delta(G)\rightarrow \mH^\delta_{[m]n}(G).
	\end{equation} 
	Passing to the direct limit over $n$, we obtain the canonical homomorphism
	\begin{equation}\label{phm*H_delta_m}
		\phm^*:\mH^\delta_{[m-1]}(G)\longrightarrow\mH^\delta_{[m]}(G). 
	\end{equation} 
	We define $$\bXp^{[m]}(G):=\varinjlim_n\frac{\bX_{m+n}(G)}{\phi^{* (m+1)}\bX_{n-1}(G)_{\phi^{m+1}}}.$$
	As in Section 8 of \cite{BoSa1}, the exact sequence in (\ref{Hom_Ker}) yields
	\begin{equation}
		\label{Xprm}
		\xymatrix{
			0\ar[r]&\bX_m(G)\ar[r]^-{u_m^*}&\bXp^{[m]}(G)\ar[r]^-{\iota_m^*}& \Hd_{[m]}(G)\ar[r]^-\partial&\bI_{[m]}(G)\ar[r]& 0,
		}
	\end{equation}
	where $\bI_{[m]}(G)\subset \Ext(J^mG,\hG)$ is as described in (\ref{Exact_Hom_Ker}).
	
	By Theorem 8.7 of \cite{PaSa}, the $R$-module $\bXp^{[0]}(G)$ is free of rank $g$. Thus,
	\begin{align}
		\label{Xprimiso}
		\bXp^{[0]}(G)\simeq R\langle\Theta_1,\ldots,\Theta_g\rangle,
	\end{align}
	where $\mathbb{B}={\Theta_1,\ldots,\Theta_g}$ is a primitive basis for $\bX_\infty(G)$ in the sense of \cite[ p.~418]{BoSa1}. Let $o_i$ denote the exact order of the $\delta$-character $\Theta_i$, for $i=1,\ldots,g$. By Theorem 1.6 of \cite{PaSa}, there exists a nonzero integer $\mup$, called the \emph{upper splitting number}, such that
	$$
	\mup=\max\{o_1,\ldots,o_g\}\leq g+1.
	$$

	\subsection{Explicit basis construction for the character modules}\label{Basis_construction}
	
	Our next goal is to describe explicitly the structure of the delta characters and construct a canonical $R$-basis for the associated modules. Recall from \eqref{Hom_Ker} the exact sequence
	\[
	0\longrightarrow \bX_m(G)\xrightarrow{u_m^*} \bX_{m+n}(G)\xrightarrow{\iota_m^*} \Hom(N^{[m]n}G,\hG)\xrightarrow{\partial} \Ext(J^mA.\hG).
	\]
	By Theorem~\ref{Psimiso}, there is a canonical isomorphism $N^{[m]n}G \simeq \mW_{n-1}^g$, which implies that the $R$-module $\Hom(N^{[m]n}A, \hG)$ has rank $ng$. 
	
	Under our assumption on the ramification index, namely $v_\pi(p) \leq p-2$, 
	Proposition $2.2$ and Lemma $2.3$ of Buium \cite{Bu2} provide a canonical isomorphism of $\pi$-formal group schemes
	\begin{align}
		\label{NmGa}
		\Psi_1^{[m]}:N^{[m]1}G\rightarrow \hG^g. 
	\end{align}
	We identify $\Psi_1^{[m]}$ with the $g$-tuple of its components $\Psi_1^{[m]} = (\Psi^{[m]}_{11}, \dots, \Psi^{[m]}_{1g})$, where $\Psi_{1j}^{[m]} = 
	\mathrm{pr}_j \circ \Psi_1^{[m]}$ and $\mathrm{pr}_j \colon \hG^g \to \hG$ denotes the projection onto the $j$-th coordinate factor. For each $1 \leq i \leq n$ and $1 \leq j \leq g$, we define the character $\Psi_{ij}^{[m]} \colon N^{[m]n}G \to \hG$ to be the composition
	\begin{equation}
		\label{Psiijm}
		N^{[m]n}G\xrightarrow{\mff_m^{\circ (i-1)}}N^{[m]n-i+1}G\xrightarrow{u_m} N^{[m]1}G\xrightarrow{\Psi_{1j}^{[m]}}\hG.
	\end{equation}
	
	\begin{lemma}\label{lemma_lin_ind}
		For all $i$, consider the set $\BB_i^{[m]}=\{h_{i1}^{[m]},\cdots,h_{ig}^{[m]}\}$ where $h_{ij}^{[m]}\in k[\bx_{m+1},\cdots,\bx_{m+n}]$ satisfy $$h_{ij}^{[m]}\equiv x_{(m+1)j}^{q^{i-1}}\mod(\deg (q^{i-1}+1)).$$ Then the polynomials in the set $\BB_1^{[m]}\cup\cdots\cup\BB_n^{[m]}$ are $k$-linearly independent polynomials.
	\end{lemma}
	\begin{proof}
		We refer to Lemma 5.1 of \cite{BoSa1} for the proof.
	\end{proof}
	\begin{proposition}\label{Psi_ij}
		The family $\{\Psi_{ij}^{[m]}\}_{ij}$, for $1\leq i\leq n$ and $1\leq j\leq g$, forms an $R$-basis for $\Hom(N^{[m]n}G,\hG).$
	\end{proposition}
	\begin{proof}
		
We have $\rk_R(\Hom(N^{[m]n}G,\hG))=ng$, hence we need to show that $\{\Psi_{ij}^{[m]}\}_{ij}$ is linearly independent. By Nakayama's Lemma, it is enough to show that $\bar{\Psi}_{ij}^{[m]}$, that is the reduction modulo $\pi$ of the $\Psi_{ij}$, are $K$-linearly independent. From the general property of the logarithm of a formal group law, $\bar{\Psi}_{1j}^{[m]}$ satisfies (see the proof of Lemma 2.3 in \cite{Bu2}) $$\bar{\Psi}_{1j}^{[m]}(\bx_{m+1},\cdots,\bx_{m+n})\equiv x_{(m+1)j}\mod(\deg 2).$$ 
		Hence for $\bar{\Psi}_i^{[m]}=(\bar{\Psi}^{[m]}_{i1},\cdots, \bar{\Psi}^{[m]}_{ig})$ we have
		$$\bar{\Psi}_{ij}^{[m]}(\bx_{m+1},\cdots,\bx_{m+n})\equiv x_{(m+1)j}^{q^{i-1}}\mod(\deg (q^{i-1}+1)).$$
		Hence by Lemma \ref{lemma_lin_ind}, we have that $\{\bar{\Psi}_{ij}^{[m]}\}_{ij}$ are $K$-linearly independent.
	\end{proof}
	By Proposition \ref{Psi_ij}, any morphism $\Psi\in \Hom(N^{[m]n}G,\hG)$ can be represented as $$\Psi=\gamma_1.\Psi_1^{[m]}+\cdots+\gamma_n.\Psi_n^{[m]}$$ where $\gamma_i\in R^g$. Let $\Theta\in \bX_{m+n}(G)$, define $[\Theta]$ to be the class of $\Theta$ in $\bX_{m+n}(G)/u^*\bX_m(G)$. We have $\iota_m^*[\Theta]\in \Hom(N^{[m]n}G,\hG)$, hence 
	\begin{equation}\label{i*Theta}
		\iota_m^*[\Theta]=\gamma_1.\Psi_1^{[m]}+\cdots+\gamma_n.\Psi_n^{[m]}.
	\end{equation}
	This implies 
	\begin{equation}\label{Theta_exp}
		[\Theta]=\bx^{(m)}.g+\gamma_1.\Psi_1^{[m]}+\cdots+\gamma_n.\Psi_n^{[m]},
	\end{equation}
	where $\bx^{(m)}=(\bx_0,\cdots,\bx_m)$ and $g=(g_0,\cdots,g_{m})\in \mO(J^{m+n}U)^{m+1}$ for some affine open subscheme $U$ of $G$ containing the identity section. 
	\begin{lemma}\label{Psi_k_m}
		Let $m\leq k\leq m+j$ and $\Theta\in \bX_{m+j}(G)$, such that $$\iota_m^*\Theta=a_j.\Psi_j^{[m]}+\cdots +a_1.\Psi_1^{[m]}$$ for some $a_1,\cdots,a_j\in R^g$, then $$\iota^*_k\Theta=\pi^{k-m}\left(a_j.\Psi_{j-k+m}^{[k]}+\cdots+a_{m+1-k}.\Psi_{1}^{[k]}\right).$$
		
	\end{lemma}	
	\begin{proof}
		We have $$\iota_m^*\Theta=a_j.\Psi_j^{[m]}+\cdots+a_1.\Psi_1^{[m]},$$ this implies $$\Theta=a_j.\Psi_j^{[m]}+\cdots+a_1.\Psi_1^{[m]}+\bx^{(m)}.h,$$ where $h\in\mO(J^{m+j}U)^{m+1}$ for some affine open subscheme $U$ of $G$ containing the identity section and $\bx^{(m)}=(\bx_0,\cdots,\bx_m)$ are the local Witt coordinates of length $m+1$ around the identity section of $G$. For all $i\geq 1,$ from the definition of $\Psi_i^{[m]},$ we have $$\Psi_i^{[m]}(\bx_{m+1},\cdots,\bx_{m+i})=\pi^{i-1}\bx_{m+i}+\pi^{i-2}\bx_{m+i-1}^q+\cdots+\bx_{m+1}^{q^{i-1}}.$$ Hence this implies,
		$$\Theta=\pi^{j-1}a_j.\bx_{m+j}+\pi^{j-2}(a_j.\bx_{m+j-1}^q+a_{j-1}.\bx_{m+j-1})+\cdots+(a_j.\bx_{m+1}^{q^{j-1}}+\cdots+a_1.\bx_{m+1})+\bx^{(m)}.h.$$ We have $m\leq k$, therefore 
		\begin{align}
			\iota_k^*\Theta&=\Theta(0,\cdots,0,\bx_{k+1},\cdots,\bx_{m+j})\\
			&=\pi^{j-1}a_j\bx_{m+j}+\pi^{j-2}(a_j.\bx_{m+j-1}^q+a_{j-1}.\bx_{m+j-1})+\cdots+\pi^{k-m}(a_j.\bx_{k+1}^{q^{j-k+m-1}}+\cdots+a_{m+1-k}\bx_{k+1})\\
			&=\pi^{k-m}\left(\pi^{j-k+m-1}a_j.\bx_{m+j}+\cdots+(a_j.\bx_{k+1}^{q^{j-k+m-1}}+\cdots+a_{m+1-k}.\bx_{k+1})\right)\\
			&=\pi^{k-m}\left(a_j.(\pi^{j-k+m-1}\bx_{m+j}+\cdots+\bx_{k+1}^{q^{j-k+m-1}})+\cdots+a_1.(\bx_{k+1})\right)\\
			&=\pi^{k-m}\left(a_j.\Psi_{j-k+m}^{[k]}+\cdots+a_{m+1-k}.\Psi_1^{[k]}\right).
		\end{align}
	\end{proof}

	\

	\section{Rigidified Extensions of Abelian Schemes}
	\label{Iso}
	
	Throughout this section, let $\mO$ be a Dedekind domain of characteristic $0$, and let $R$ be a $\pi$-adically complete discrete valuation ring equipped with a $\pi$-derivation $\delta$ lifting the derivation on $\mO$. Let $l = R/\pi R$ denote the residue field of $R$, and let $K$ be its fraction field. We assume throughout that the ramification index satisfies $v_{\pi}(p) \leq p-2$. Set $S = \Spf R$, and let $G$ be a smooth $\pi$-formal group scheme over $S$ of relative dimension $g$.

	\color{black}

	\subsection{Extensions of arithmetic jet spaces of abelian schemes}\label{Extn_Jets}
	
	Let $A$ be an abelian scheme of relative dimension $g$ over $S$, and let us denote its formal completion along the special fiber simply by $A$. The generalized kernel sequence provides a canonical short exact sequence of $\pi$-formal group schemes:
	\begin{equation}\label{Ker_A}
		0\rightarrow N^{[m]n}A\xrightarrow{\iota_m} J^{m+n}A\xrightarrow{u_m} J^mA\rightarrow 0.
	\end{equation}
	Applying the functor $\Hom(-,\hG)$ yields the following exact sequence of $R$-modules:
	\begin{equation}\label{Hom_Ker_A}
		0\rightarrow \bX_m(A)\xrightarrow{u^*_m} \bX_{m+n}(A)\xrightarrow{\iota_m^*} \Hom(N^{[m]n}A,\hG)\xrightarrow {\partial}\Ext(J^mA,\hG).
	\end{equation}
	
	Any additive character $\Psi\in \Hom(N^{[m]n}A,\hG)$ canonically determines a push-out of the exact sequence \eqref{Ker_A} along $\Psi$. This operation yields an extension of $J^mA$ by $\hG$, which we denote by $(J^mA)_{\Psi}^*$. Explicitly, the formal group scheme $(J^mA)_{\Psi}^*$ is realized as the quotient
	\[(J^mA)_{\Psi}^*=\frac{J^{m+n}A\times \hG}{\Gamma_\Psi(N^{[m]n}(A))},\]
	where $\Gamma_{\Psi}(N^{[m]n}(A))$ is the graph-like subgroup defined by
	\[\Gamma_{\Psi}(N^{[m]n}(A))=\{(\iota_m(z),-\Psi(z))\mid z\in N^{[m]n}(A)\}.\]
	This construction yields the following commutative diagram with exact rows:
	\begin{equation}\label{Extn}
		\xymatrix{
			0\ar[r] & N^{[m]n}A\ar[r]^{\iota_m}\ar[d]_{\Psi} & J^{m+n}A\ar[r]^{u_m}\ar[d]_{g_\Psi} & J^mA\ar[r]\ar@{=}[d] & 0\\
			0\ar[r] & \hG\ar[r] & (J^mA)_{\Psi}^*\ar[r] & J^mA\ar[r] & 0,
		}
	\end{equation}
	where the vertical morphism $g_{\Psi}$ is defined by the natural inclusion $g_{\Psi}(x)=[(x,0)]\in (J^mA)_{\Psi}^*$.
	
	To analyze the induced map on Lie algebras, let $\bx_0=(x_{01},\dots,x_{0g})$ denote a fixed system of local \'etale coordinates for $A$ at the identity section. This choice canonically determines a system of local \'etale coordinates $(\bx_0,\dots,\bx_n)$ for $J^{n}A$ at the identity section, which in turn furnishes a basis for the free $R$-module $\Lie J^{n}A$.
	
	We define a canonical projection map on the Lie algebra level, $s_{\Witt_m}:\Lie J^{m+n}A\rightarrow \Lie N^{[m]n}A$, given explicitly in coordinates by
	\begin{equation}\label{S_witt}
		s_{\Witt_m}(\bx_0,\dots, \bx_{m+n})=(0,\dots,0,\bx_{m+1},\dots, \bx_{m+n}).
	\end{equation}
	Similarly, we define an inclusion map $v_{m.n}:\Lie J^mA\rightarrow\Lie J^{m+n}A$ by
	\begin{equation}\label{v_m_Lie}
		v_{m,n}(\bx_0,\dots,\bx_m)=(\bx_0,\dots,\bx_m,0,\dots,0).
	\end{equation}
	These choices naturally split the Lie algebra sequence, providing a split short exact sequence of $R$-modules:
	\begin{equation}
		\xymatrix{
			0\ar[r] & \Lie N^{[m]n}A\ar[r]^{D\iota_m} & \Lie J^{m+n}A\ar@/^1pc/[l]^{s_{\Witt_m}}\ar[r]^{Du_m} & \Lie J^mA\ar@/^1pc/[l]^{v_{m,n}}\ar[r] & 0. 
		}
	\end{equation}
	We verify that the operators behave as expected:
	\begin{align}
		(\mathbf{1}-v_{m,n}\circ Du_m)(\bx_0,\dots,\bx_{m+n})&=(\bx_0,\dots,\bx_{m+n})-v_{m,n}(\bx_0,\dots,\bx_m)\\
		&=(\bx_0,\dots,\bx_{m+n})-(\bx_0,\dots,\bx_m,0,\dots,0)\\
		&=(0,\dots,0,\bx_{m+1},\bx_{m+2},\dots, \bx_{m+n})\\
		&=s_{\Witt_m}(\bx_0,\dots,\bx_{m+n}).
	\end{align}
	
	Applying the Lie algebra functor to the extension diagram \eqref{Extn}, we obtain the following diagram of Lie algebras, where the lower exact sequence is equipped with a specific splitting:
	\begin{equation}
		\xymatrix{
			0\ar[r] & \Lie N^{[m]n}A\ar[d]_{D\Psi}\ar[r]^{D\iota_m} & \Lie J^{m+n}A\ar[d]^{Dg_{\Psi}}\ar@/^1pc/[l]^{s_{\Witt_m}}\ar[r]^{Du_m} & \Lie J^mA\ar@/^1pc/[l]^{v_{m,n}}\ar[r]\ar@{=}[d] & 0\\
			0\ar[r] & \Lie \hG\ar[r] &\Lie (J^mA)_{\Psi}^*\ar[r] \ar@/^1pc/[l]^{s_{\Psi_m}}& \Lie J^mA\ar[r] & 0. 
		}
	\end{equation}
	Here, the Lie algebra of the push-out extension takes the form
	\[\Lie (J^mA)_{\Psi}^*=\frac{\Lie J^{m+n}A\times \Lie \hG}{\Lie \Gamma_{\Psi}(N^{[m]n}A)}.\]
	The splitting $s_{\Psi_m}:\Lie (J^mA)^*_{\Psi}\rightarrow \Lie \hG$ is constructed universally. Consider the auxiliary map $\tilde{s}_{\Psi_m}$:
	$$
	\xymatrix{
		\Lie J^{m+n}A\times \Lie \hG\ar[rr]^-{\tilde{s}_{\Psi_m}}\ar[dr]&  &\Lie \hG\\
		& \Lie (J^mA)_{\Psi}^*\ar[ur]_{s_{\Psi_m}}
	}
	$$
	defined by $\tilde{s}_{\Psi_m}(\bx,y)=D\Psi(s_{\Witt_m}(\bx))+y$. To ensure $s_{\Psi_m}$ is well-defined on the quotient, suppose $(\bx,y)\in \Lie(\Gamma_\Psi(N^{[m]n}(A)))$. Then there exists an element $z\in \Lie(N^{[m]n}A)$ such that $\bx=D\iota_m(z)$ and $y=-D\Psi(z)$. Evaluating $\tilde{s}_{\Psi_m}$ on this pair yields:
	\begin{align}
		\tilde{s}_{\Psi_m}(\bx,y)&=D\Psi(s_{\Witt_m}(\bx))+y\\
		& =D\Psi(s_{\Witt_m}(D\iota_m(z)))-D\Psi(z)\\
		&=D\Psi(z)-D\Psi(z)=0.
	\end{align}
	The above implies that $s_{\Psi_m}$ descends to a well-defined $R$-linear splitting map $s_{\Psi_m}:\Lie (J^mA)_{\Psi}^*\map \Lie \hG$ given by the class evaluation $s_{\Psi_m}([\bx,y])=D\Psi(s_{\Witt_m}(\bx))+y$.

	\subsection{Rigidified extensions and $\Ext^\sharp$}
	
	Following Mazur and Messing \cite{Mazur}, we define $\Ext^{\sharp}(G,\hG)$ for any $\pi$-formal group scheme $G$ as the set parameterizing the isomorphism classes of extensions of $G$ by $\hG$, equipped with a specified splitting of the associated short exact sequence of their Lie algebras. Since $\hG$ is an $R$-module $\pi$-formal scheme, $\Ext^\sharp(G,\hG)$ 
canonically inherits an $R$-module structure.
	
	The preceding construction demonstrates that the pair $((J^mA)_{\Psi}^*,s_{\Psi_m})$ inherently defines an element of $\Ext^{\sharp}(J^m A,\hG)$. This structural assignment yields a canonical mapping:
	\begin{equation}\label{Lambda_m}
		\Lambda_{[m]}^n:\Hom(N^{[m]n}A,\hG)\longrightarrow \Ext^{\sharp}(J^mA,\hG)
	\end{equation} 
	which maps $\Psi\mapsto ((J^mA)_{\Psi}^*,s_{\Psi_m})$.
	
	Let $G$ be an arbitrary smooth commutative $\pi$-formal group scheme over $S$, and let $E$ be an extension of $G$ by $\hG$:
	$$
	\xymatrix{
		0\ar[r]& \hG\ar[r] & E\ar[r]& G\ar[r]& 0.
	}
	$$
	Following the classical theory of Mazur and Messing \cite{Mazur}, we define the subgroup of vanishing differential characters as
	\[\Hom^\triangledown(G,\hG):=\{f\in\Hom(G,\hG)\mid Df:\Lie G\rightarrow\Lie \hG\simeq R \text{ is the zero map}\}.\]
	It is evident that $\Hom^\triangledown(G,\hG)$ forms an $R$-submodule of $\Hom(G,\hG)$. The theory developed in \cite[Page 71, Equation (4.2)]{Mazur} yields the fundamental exact sequence of $R$-modules:
	$$
	0\rightarrow\Hom^\triangledown(G,\hG)\rightarrow \Hom(G,\hG)\rightarrow\Lie(G)^\vee\rightarrow \Ext^{\sharp}(G,\hG)\rightarrow\Ext(G,\hG).
	$$
	Specializing to $G=J^mA$, we obtain the exact sequence governing the extensions of the arithmetic jet spaces:
	\begin{equation}
		0\rightarrow\Hom^\triangledown(J^mA,\hG)\rightarrow \Hom(J^mA,\hG)\rightarrow\Lie(J^mA)^\vee\rightarrow \Ext^{\sharp}(J^mA,\hG)\rightarrow\Ext(J^mA,\hG).
	\end{equation}

	
	\begin{lemma} 
		$\Hom^\triangledown(J^mA,\hG)=\{0\}$ for any $m\geq 0$.
	\end{lemma}
	\begin{proof}
		Let us consider the following exact sequence
		$$
		0\rightarrow\bX_0(A)=\{0\}\xrightarrow{u_0^*}\bX_{m}(A)\xrightarrow{\iota_0^*}\Hom(N^{[0]m}A,\hG)\xrightarrow{\partial}\Ext(A,\hG).
		$$
		Let $\Theta\in \bX_m(A)$, then $\iota_0^*\Theta\in \Hom(N^{[0]m}A,\hG)$ and  $$\iota_0^*\Theta=a_1\Psi_1^{[0]}+\cdots + a_m\Psi_m^{[0]},$$ for some $a_i\in R^g,$ as we showed in (\ref{i*Theta}), where $\Psi_{i}^{[0] }\in \Hom(N^{[0]m}A,\hG) $ is a character of order $i$ for all $1\leq i\leq m$. As $A$ is an abelian scheme, we have $\bX_0(A)=\{0\}$ and hence  $\iota_0^*:\bX_{m}(A)\rightarrow \Hom(N^{[0]m}A,\hG)$ is an injection.
		Therefore \begin{equation}\label{phi_Psi}
			\Theta=\bx_0h+a_1\Psi_1^{[0]}+\cdots + a_m\Psi_m^{[0]}
		\end{equation}
		for some $h\in\mathcal{O}(J^mU)^g$ where $U$ is an open affine subscheme of $A$ containing the identity section and $\bx_0=(x_{01},\cdots,x_{0g})$ are \'etale coordinates of $A$ around the identity section $e$ as we have $\iota_0^*$ mapping $\bx_0\mapsto 0$. By \cite{Bu2}, Lemma 2.3,  We have 
		\begin{equation}\label{psi_i0}
			\Psi_{i}^{[0]}=\pi^{i-1}\bx_i+\cdots+\pi\bx_{2}^{q^{i-2}}+\bx_1^{q^i-1},\end{equation}
		where $\{\bx_0,\cdots,\bx_m\}$ are local Witt coordinates of length $m+1$ of $A$ around $e$. Let us consider the map $D\Theta:\Lie J^mA\rightarrow \Lie \hG\simeq R$. We know that $\bx:=(\bx_0,\cdots,\bx_m)$ gives us a basis of $\Lie J^{m}A$ and $D\Theta$ can be defined as $$D\Theta=\left(\frac{\partial\Theta}{\partial\bx_0}\Big|_{\bx=0},\frac{\partial\Theta}{\partial\bx_1}\Big|_{\bx=0},\cdots,\frac{\partial\Theta}{\partial\bx_m}\Big|_{\bx=0}\right).$$
		Note that \begin{align}
			\frac{\partial\Theta}{\partial\bx_0}=\frac{\partial(\bx_0h)}{\partial\bx_0}=h+\bx_0\frac{\partial h}{\partial\bx_0}.
		\end{align}
		Hence we have,
		\begin{align}
			\frac{\partial\Theta}{\partial\bx_0}\Big|_{\bx=0}=h(0).
		\end{align}
		For $1\leq i\leq m$,
		\begin{align}
			\frac{\partial\Theta}{\partial\bx_i}&=\frac{\partial(\bx_0h+a_i\Psi_{i}^{[0]}+\cdots+a_m\Psi_{m}^{[0]})}{\partial\bx_i}\\
			&=\bx_0\frac{\partial h}{\partial \bx_i}+a_i\frac{\partial\Psi_i^{[0]}}{\partial \bx_i}+\cdots +a_m\frac{\partial\Psi_m^{[0]}}{\partial \bx_i}.
		\end{align}
		From (\ref{psi_i0}), we have for all $k\geq i$, $$\frac{\partial\Psi_{k}^{[0]}}{\partial\bx_i}=\pi^{i-1}q^{k-i}\bx_i^{q^{k-i}-1}.$$
		Hence putting $\bx=0$, we obtain 
		\begin{align}
			\frac{\partial\Theta}{\partial\bx_i}\Big|_{\bx=0}&=\bx_0\frac{\partial h}{\partial \bx_i}\Big|_{\bx=0}+a_i\frac{\partial\Psi_i^{[0]}}{\partial \bx_i}\Big|_{\bx=0}+\cdots +a_m\frac{\partial\Psi_m^{[0]}}{\partial \bx_i}\Big|_{\bx=0}\\
			&=\pi^{i-1}a_i.
		\end{align} 
		Therefore \begin{equation}
			D\Theta=\left( h(0),a_1,\pi a_2,\cdots,\pi^{m-1}a_m\right).
		\end{equation}
		Hence if $\Theta\in \Hom^\triangledown(J^mA,\hG)$, then $D\Theta=0$, which implies $h(0)=0$ and $a_i=0$ for all $1\leq i\leq m.$ Therefore $\Theta=\bx_0h$, but note that $\iota_0^*\Theta=0$ and $\iota_0^*$ is an injection, hence $\Theta=0$ and we are done.
	\end{proof}

	\color{black}
	Hence we obtain the exact sequence
	\begin{equation}\label{Hodge}
		\xymatrix{0\ar[r] & \bX_m(A)\ar[r] & \Lie(J^mA)^\vee\ar[r] & \Ext^{\sharp}(J^mA,\hG)\ar[r] & \Ext(J^mA,\hG).}
	\end{equation}

	\begin{proposition}\label{Dphi_k}
		For every $n,k\geq 0$,  $D(\phi^{\circ k}):\Lie J^{k+n}A\rightarrow \Lie J^nA$ is given by the matrix\begin{equation}
			D(\phi^{\circ k})=\begin{pmatrix}
				\mathbf{0}_g & \cdots & \mathbf{0}_g & \pi^k\mathbf{I}_g & \mathbf{0}_g & \cdots & \mathbf{0}_g\\
				\mathbf{0}_g & \cdots & \mathbf{0}_g & \mathbf{0}_g & \pi^k\mathbf{I}_g & \cdots & \mathbf{0}_g\\
				\vdots &\ddots & \vdots & \vdots & \vdots &\ddots & \vdots\\
				\mathbf{0}_g & \cdots & \mathbf{0}_g & \mathbf{0}_g & \mathbf{0}_g & \cdots &\pi^k \mathbf{I}_g
			\end{pmatrix}_{(n+1)g,(n+k+1)g.}
		\end{equation}
	\end{proposition}
	
	\begin{proof}
		Let $\bx_0=(x_{01},\cdots,x_{0g})$ be an \'etale coordinate system of $A$ around the identity section. From the discussion in the  section (\ref{localcoordinates}), recall that $(\bx_0,\cdots,\bx_{n})$ denote the local Witt coordinates for $ J^nA$ around the identity section where $\bx_i=(x_{i1},\cdots,x_{ig}).$
		In this coordinate system, if the lift of Frobenius map $\phi:J^{n+1}A\rightarrow J^nA$ is given by \begin{equation}
			\phi(\bx_0,\cdots,\bx_{n+1})=(\by_0,\cdots,\by_n),
		\end{equation} 
		then from \cite{BoSa1}, Page 410 (6.1), we have $$d\by_n=\pi d\bx_{n+1}.$$
		Therefore the derivative matrix of $\phi$ at the identity section is given by the block matrix \begin{equation}
			D\phi=\begin{pmatrix}
				\mathbf{0}_g & \pi\mathbf{I}_g & \mathbf{0}_g & \cdots & \mathbf{0}_g\\
				\mathbf{0}_g & \mathbf{0}_g & \pi\mathbf{I}_g & \cdots & \mathbf{0}_g\\
				\vdots & \vdots & \vdots &\ddots &\vdots\\
				\mathbf{0}_g & \mathbf{0}_g & \mathbf{0}_g & \cdots & \pi\mathbf{I}_g
			\end{pmatrix}_{(n+1)g, (n+2)g}
		\end{equation}
		where $\mathbf{0}_g$ is the zero matrix of size $g\times g$ and $\mathbf{I}_g$ is the identity matrix of size $g\times g$. Hence 
		$$D(\phi^{\circ k})=\underbrace{D\phi\circ\cdots\circ D\phi}_{k-\text{many}}=\begin{pmatrix}
			\mathbf{0}_g & \cdots & \mathbf{0}_g & \pi^k\mathbf{I}_g & \mathbf{0}_g & \cdots & \mathbf{0}_g\\
			\mathbf{0}_g & \cdots & \mathbf{0}_g & \mathbf{0}_g & \pi^k\mathbf{I}_g & \cdots & \mathbf{0}_g\\
			\vdots & \ddots & \vdots & \vdots & \vdots & \ddots& \vdots\\
			\mathbf{0}_g & \cdots & \mathbf{0}_g& \mathbf{0}_g & \mathbf{0}_g & \cdots &\pi^k \mathbf{I}_g
		\end{pmatrix}_{((n+1)g),(n+k+1)g.}$$
		
	\end{proof}
	
	Recall from (\ref{v_m_Lie}), that we have a map $v_{l,n}:\Lie J^lA\rightarrow \Lie J^{l+n}A$ for all $l,n\geq 0$ defined as $$v_{l,n}(\bx_0,\cdots,\bx_l)=(\bx_0,\cdots,\bx_l,0,\cdots,0).$$
	
	\begin{lemma}\label{Lie_map}
		Let $\Theta\in \bX_n(A)$ be a character of order $n$ and $\tilde{\Theta}:=\phi^{*\circ  k}\Theta$ for some $k\geq 1$. Then $D\tilde{\Theta}\circ v_{k-l-1,n+l+1}=0$ for all $0\leq l\leq k-1.$
	\end{lemma}
	\begin{proof}
		We have $\tilde{\Theta}=\phi^{* \circ k}\Theta:J^{n+k}A\rightarrow \hG,$ this implies 
		\begin{equation}
			D\tilde{\Theta}=D(\phi^{*\circ k}(\Theta))=D(\Theta\circ\phi^{\circ k})=D\Theta\circ D\phi^{\circ k}:\Lie J^{n+k}A\longrightarrow \Lie \hG.
		\end{equation}
		From Proposition \ref{Dphi_k}, we have 
		$$D(\phi^{\circ k})=\begin{pmatrix}
			\mathbf{0}_g & \cdots & \mathbf{0}_g & \pi^k\mathbf{I}_g & \mathbf{0}_g & \cdots & \mathbf{0}_g\\
			\mathbf{0}_g & \cdots & \mathbf{0}_g & \mathbf{0}_g & \pi^k\mathbf{I}_g & \cdots & \mathbf{0}_g\\
			\vdots & \ddots & \vdots & \vdots & \vdots &\ddots & \vdots\\
			\mathbf{0}_g & \cdots & \mathbf{0}_g & \mathbf{0}_g & \mathbf{0}_g & \cdots &\pi^k \mathbf{I}_g
		\end{pmatrix}_{((n+1)g),(n+k+1)g.}$$
		Note that $D\tilde{\Theta}\circ v_{k-l-1,n+l+1}=D\Theta\circ D\phi^{\circ k}\circ v_{k-l-1,n+l+1}$ for all $0\leq l\leq k-1$. In matrix form, 
		$$
		v_{k-l-1,n+l+1}=\begin{pmatrix}
			\vdots\\
			\mathbf{I}_{(k-l)g}\\
			\vdots\\
			\mathbf{0}_{(n+l+1)h,(k-l)g}\\
			\vdots\\
			
		\end{pmatrix}_{((n+k+1)g,(k-l)g}
		$$
		where the top $(k-l)g$ block is the identity matrix and the rest of the entries are zero. Therefore $D\phi^{\circ k}\circ v_l=0$ and hence $D\tilde{\Theta}\circ v_{k-l-1,n+1+l}=0$ for all $0\leq l\leq k-1$.
	\end{proof}

	Note that we have a map of exact sequences between (\ref{Hom_Ker_A}) and (\ref{Hodge}):
	\begin{equation}
		\xymatrix{
			0\ar[r]&\bX_m(A)\ar[r]^-{u_m^*}\ar@{=}[d]& \bX_{m+n}(A)
			\ar[d]^{\Upsilon_{[m]}^n}
			\ar[r]^-{\iota_m^*}\ar[d]& \Hom(N^{[m]n}A,\hG)\ar[r]^-{\partial}\ar[d]_{\Lambda_{[m]}^n:}^{\Psi\mapsto ((J^mA)_\Psi^*,s_{\Psi_m})}& \Ext(J^mA,\hG)\ar@{=}[d] \\
			0\ar[r]&\bX_m(A)\ar[r] &\Lie (J^mA)^\vee\ar[r]& \Ext^\sharp(J^mA,\hG)\ar[r]&\Ext(J^mA,\hG)
		}
	\end{equation}
	We describe the morphism $\bX_{m+n}(A)\rightarrow \Lie(J^mA)^\vee$ explicitly in the following proposition.
	
	\color{black}
	\begin{proposition}\label{Upsilon_map}
		Let $\Theta\in \bX_{m+n}(A)$ and set $\Psi=\iota_m^*\Theta\in \Hom(N^{[m]n}A,\hG)$. Then:
		\begin{enumerate}
			\item The map $\Upsilon_{[m]}^n:\bX_{m+n}(A)\to \Lie(J^mA)^\vee$ is given by
			
			$$
			\Upsilon_{[m]}^n(\Theta)=-D\Theta\circ v_{m,n}.
			$$
			
			\item If $\Theta=\phi^{*(m+1)}\widetilde{\Theta}$ for some $\widetilde{\Theta}\in \bX_{n-1}(A)$, then
			
			$$
			-D\Theta\circ v_{m,n}=0.
			$$
			
			\item If $\Psi\in \iota_m^*\phi^{*(m+1)}\bigl(\bX_{n-1}(A)\bigr)$, then the rigidified extension
			
			$$
			\bigl((J^mA)_\Psi^*,s_{\Psi_m}\bigr)\in \Ext^\sharp(J^mA,\hG)
			$$
			
			is trivial.
		\end{enumerate}
	\end{proposition}
	
	\begin{proof}
		We will follow a similar argument to that of the proof of 
		Proposition 6.1 of \cite{BoSa1}. We have $\Psi=\iota_m^*\Theta=\Theta\circ \iota_m \in \Hom(N^{[m]n}A,\hG).$ 
		
		(1) The push-forward extension in (\ref{Extn}) is given as
		$$(J^mA)_{\Psi}^*:=\frac{J^{m+n}A\times \hG}{\Gamma_{\Psi}(N^{[m]n}A)}.$$ We first identify this extension with the split extension $J^mA\times\hG$. Consider the trivialization
		
		$$
		\xymatrix{
			(J^mA)_\Psi^*
			=\frac{J^{m+n}A\times\hG}
			{\Gamma_\Psi(N^{[m]n}A)}
			\ar[r]^-{\sim}
			&
			J^mA\times\hG
		}
		$$
		
		given by
		
		$$
		[\bx,b]\longmapsto
		\bigl(u_m(\bx),\Theta(\bx)+b\bigr).
		$$
		
		Passing to Lie algebras gives the induced isomorphism
		
		$$
		\xymatrix{
			\Lie (J^mA)_\Psi^*
			=
			\frac{\Lie J^{m+n}A\times\Lie\hG}
			{\Lie(\Gamma_\Psi(N^{[m]n}A))}
			\ar[r]^-{\sim}
			&
			\Lie J^mA\times\Lie\hG
		}
		$$
		
		which sends $[\bx,y]$ to $\bigl(u_m(\bx),D\Theta(\bx)+y\bigr)$. Its inverse is the map
		
		$$
		H_m:\Lie J^mA\times\Lie\hG
		\longrightarrow
		\Lie (J^mA)_\Psi^*
		$$
		
		defined by
		
		$$
		H_m(\bx,y)
		=
		\left[
		v_{m,n}(\bx),
		y-D\Theta\circ v_{m,n}(\bx)
		\right].
		$$ Hence the composition $$
		\xymatrix{
			\Lie J^mA \ar[r] & \Lie J^mA\times \Lie \hG\ar[r] & \Lie (J^mA)_{\Psi_m}^*\ar[r]& \Lie \hG
		}
		$$
		is explicitly given by $$\bx\mapsto  (\bx,0)\mapsto  [\bx,0]\mapsto -D\Theta\circ v_{m,n}(\bx)$$
		and $\Upsilon_{[m]}^n(\Theta)=-D\Theta\circ v_{m,n}$, as claimed.

		(2)  This follows from Lemma \ref{Lie_map}.
		
		(3) We have $\Psi\in \iota_m^*(\phi^{*(m+1)}\bX_{n-1}(A)\subset \Hom(N^{[m]n}A,\hG)$, this implies $\Psi\in \iota_m^*\bX_{n+m}(A)$.  Thus we may write
		
		$$
		\Psi=\Theta\circ\iota_m
		=
		\widetilde{\Theta}\circ\phi^{(m+1)}\circ\iota_m
		$$
		
		for some $\widetilde{\Theta}\in\bX_{n-1}(A)$. By part (2),
		
		$$
		-D\Theta\circ v_{m,n}=0.
		$$ Therefore, by (2) we have $-D\Theta\circ v_{m,n}=0$ and the extension $((J^mA)_{\Psi}^*,s_{\Psi_m})$ is   trivial.
		
	\end{proof}

	Let us recall 
	
	$$\bXp^{[m]}(A):=\varinjlim_n\frac{\bX_{m+n}(A)}{\phi^{* (m+1)}\bX_{n-1}(A)_{\phi^{m+1}}}$$
	and  $$\mH_{[m]}^\delta(A)=\varinjlim_n\mH_{[m]n}^\delta,$$ where $$\mH_{[m]n}^\delta (A):=\frac{\Hom(N^{[m]n}A,\hG)}{\iota_m^*\phi^{*(m+1)}\bX_{n-1}(A)_{\phi^{m+1}}}.$$ 
	Note that $\phi^*(\phi^{*m}\bX_{n-1}(A)_{\phi^m})=\phi^{*(m+1)}\bX_{n-1}(A)_{\phi^{m+1}}$, hence Frobenius descends to a morphism 
	\begin{equation}
		\phi^*:\bXp^{[m-1]}(A)\rightarrow \bXp^{[m]}(A)_{\phi}.
	\end{equation}
	
	\begin{theorem}\label{map_of_seq}
		We have the following $R$-module homomorphisms.
		\begin{enumerate}
			\item The map $\Upsilon_{[m]}:\bXp^{[m]}(A)\rightarrow \Lie (J^mA)^\vee$ is given by $[\Theta]\mapsto -D\Theta\circ v_{m,n}.$ 
			\item  The map $\Lambda_{[m]}:\mH_{[m]}^\delta(A)\rightarrow \Ext^\sharp(J^mA.\hG)$ is given by $\Psi\mapsto ((J^mA)_{\Psi}^*,s_{\Psi_m})$. 
			\item Moreover we have the following map of exact sequences
			\begin{equation}
				\xymatrix{
					0\ar[r] & \bX_m(A)\ar[r]^-{u_m^*}\ar@{=}[d] & \bXp^{[m]}(A)\ar[r]^-{\iota_m^*}\ar[d]_{\Upsilon_{[m]}} & \mH_{[m]}^\delta(A)\ar[r]^-{\partial}\ar[d]_{\Lambda_{[m]}} & \Ext(J^mA,\hG)\ar@{=}[d] \\
					0\ar[r]  & \bX_m(A)\ar[r] & \Lie(J^mA)^\vee\ar[r] & \Ext^{\sharp}(J^mA,\hG)\ar[r] & \Ext(J^mA,\hG).
				}  
			\end{equation}
		\end{enumerate}
	\end{theorem}
	\begin{proof}
		
		From (1) and (2) of Proposition \ref{Upsilon_map} , for every $m,n\geq 0$ we have the map 
		$$\frac{\bX_{m+n}(A)}{\phi^{*(m+1)}\bX_{n-1}(A)}\longrightarrow \Lie(J^mA)^\vee,$$ given by $[\Theta]\mapsto -D\Theta\circ v_{m,n}$ is well-defined. Hence passing to the limit over $n$, we obtain the map $$\Upsilon_{[m]}:\bXp^{[m]}(A)\longrightarrow \Lie(J^mA)^\vee.$$ 
		Recall from (\ref{Lambda_m}) the $R$-module map $\Lambda_{[m]}^n:\Hom(N^{[m]n}A,\hG)\rightarrow\Ext^{\sharp}(J^mA,\hG)$ given by $\Psi\mapsto ((J^mA)_{\Psi}^*,s_{\Psi_m}).$ By Proposition \ref{Upsilon_map} (3) this map factors 
		through $\mH_{[m]n}^\delta(A)$ and hence passing to the limit over $n$ we have $$\Lambda_{[m]}:\mH_{[m]}^\delta(A)\longrightarrow \Ext^{\sharp}(J^mA,\hG),$$ and the following diagram commutes 
		\begin{equation}\label{com1}
			\xymatrix{
				\bXp^{[m]}(A)\ar[r]^{\iota_m^*}\ar[d]_{\Upsilon_{[m]}} & \mH_{[m]}^\delta(A)\ar[d]^{\Lambda_{[m]}}\\
				\Lie(J^mA)^{\vee}\ar[r]& \Ext^{\sharp}(J^mA,\hG).
			}
		\end{equation}
		From equation (\ref{Hom_Ker_A}), we get the exact sequence 
		\begin{align}
			0\rightarrow \frac{\bX_m(A)}{\bX_m(A)\cap (\phi^{*(m+1)}\bX_{n-1}(A))}&\xrightarrow{u_m^*}  \frac{\bX_{m+n}(A)}{(\phi^{* (m+1)}\bX_{n-1}(A))}\\
			&\xrightarrow{\iota_m^*}\frac{\Hom(N^{[m]n}A,\hG)}{\iota_m^*(\phi^{*(m+1)}\bX_{n-1}(A))}\xrightarrow{\partial}\Ext(J^mA,\hG).\notag
		\end{align}
		Note that all the characters of $A$ in $\phi^{* (m+1)}\bX_{n-1}(A)$ have order at least $m+1$, hence $$\bX_m(A)\cap (\phi^{*(m+1)}\bX_{n-1}(A))=\{0\}.$$ Hence this gives us the exact sequence $$0\rightarrow \bX_m(A)\xrightarrow{u_m^*} \frac{\bX_{m+n}(A)}{\phi^{* (m+1)}\bX_{n-1}(A)}\xrightarrow{\iota_m^*}\frac{\Hom(N^{[m]n}A,\hG)}{\iota_m^*\phi^{* (m+1)}\bX_{n-1}(A)}\xrightarrow{\partial}\Ext(J^mA,\hG).$$ Passing to the limit over $n$, we have 
		\begin{equation}\label{Exact}
			0\rightarrow \bX_m(A)\xrightarrow{u_m^*} \bXp^{[m]}(A)\xrightarrow{\iota_m^*}\mH_{[m]}^\delta(A)\xrightarrow{\partial}\Ext(J^mA,\hG).
		\end{equation} From diagram (\ref{com1}), we have a map between (\ref{Exact}) and (\ref{Hodge}) given by $$
		\xymatrix{
			0\ar[r] & \bX_m(A)\ar[r]^-{u_m^*}\ar@{=}[d] & \bXp^{[m]}(A)\ar[r]^-{\iota_m^*}\ar[d]_-{\Upsilon_{[m]}} & \mH_{[m]}^\delta(A)\ar[r]^-{\partial}\ar[d]_-{\Lambda_{[m]}} & \Ext(J^mA,\hG)\ar@{=}[d] \\
			0\ar[r]  & \bX_m(A)\ar[r] & \Lie(J^mA)^\vee\ar[r] & \Ext^{\sharp}(J^mA,\hG)\ar[r] & \Ext(J^mA,\hG).
		}  
		$$
	\end{proof}
	\color{black}
	\begin{proposition}\label{Char_n}
		Let $\Theta\in \bX_n(A)$ be a delta character of order $n$ and $D\Theta$ be its derivative at the identity.  With respect to our chosen coordinates, let $D\Theta=(A_0,\cdots, A_n)$ where $A_j\in Mat_{1\times g}(R)$ for all $j=0,\dots, n$. Then we have $$\iota_m^*\phi^{*(m+1)}\Theta=\mf{f}_m^*\iota_m^*\phi^{*m}\Theta+\gamma.\Psi_{1}^{[m]},$$ where $\gamma=\pi^{m+1}A_0$ and $\Psi_{1}^{[m]}=(\Psi_{11}^{[m]},\cdots,\Psi_{1g}^{[m]})^t$.
	\end{proposition}
	\begin{proof}
		By Theorem 5.9 in \cite{Mi-Sa-1}, there exists an element $\gamma\in R^g$ such that $$(\iota_m^*\phi^{*(m+1)}-\mf{f}_m^*\iota_m^*\phi^{*m})\Theta=\gamma.\Psi^{[m]}_1.$$
		Then we have 
		\begin{equation}\label{Theta-Psi}
			\iota_m^*\phi^{*(m+1)}\Theta\equiv \gamma.\Psi^{[m]}_1\mod (\Psi_{21}^{[m]},\cdots,\Psi_{2g}^{[m]},\cdots,\Psi_{(n+1)1}^{[m]},\cdots, \Psi_{(n+1)g}^{[m]}) 
		\end{equation}
		in $\Hom(N^{[m]n+1}A,\hG),$ where $\Psi_{ij}^{[m]}:N^{[m]n+1}A\rightarrow\hG$ is the composition $$N^{[m]n+1}A\xrightarrow{\mff_m^{\circ (i-1)}}N^{[m]n-i+2}A\xrightarrow{u_m} N^{[m]1}A\xrightarrow{\Psi_{1j}^{[m]}}\hG,$$ for every $1\leq i\leq n+1$ and $1\leq j\leq g$. By Proposition \ref{Dphi_k} the derivative matrix $D(\Theta\circ\phi^{\circ (m+1)}\circ\iota_m)$ at the identity section is $$(\pi^{m+1}A_0,\cdots,\pi^{m+1}A_n).$$ 
		Differentiating both sides of the equation \eqref{Theta-Psi} at the identity section and comparing the first $g$ coordinates of the resulting derivative vectors gives
		\begin{equation}
			\gamma=\pi^{m+1}A_0.
		\end{equation}
	\end{proof}
	\begin{theorem}\label{isom_phi}
		For all $m\geq 0$ and $n\geq 1$, the following diagram  commutes
		\begin{equation}
			\xymatrix{
				\bX_{n}(A)/u_{n-1}^*\bX_{n-1}(A) \ar[r]^-{\phi^{*(m+1)}}\ar@{^{(}->} [d]_-{\iota_m^*\phi^{*m}}  &\bX_{m+n+1}(A)/u_{m+n}^*\bX_{m+n}(A)\ar@{^{(}->} [d]^-{\iota_m^*}\\
				\Hom(N^{[m]n}A,\hG)/\Hom(N^{[m]n-1}A,\hG)\ar [r]^-{\mf{f}_m^*}&  \Hom(N^{[m]n+1}A,\hG)/\Hom(N^{[m]n}A,\hG),
			}
		\end{equation}
		where $\mff_m^*$ is  an $R$-module isomorphism and $\iota_m^*$ and $\phi^{*(m+1)}$ are injections. For all $n\geq \beta_0,$ $\phi^{*(m+1)}$ is a bijection.
	\end{theorem}
	\begin{proof}
		For all $m\geq 0$ and $n\geq 1$, commutativity of the diagram follows from Proposition \ref{Char_n}. From Proposition 6.4 in \cite{BoSa1}, we have $\phi^*:\bX_{n}(A)/u_{n-1}^*\bX_{n-1}(A)\rightarrow\bX_{n+1}(A)/u_n^*\bX_{n}(A)$ is injective for all $n\geq 1$ and this implies $\phi^{*(m+1)}:\bX_{n}(A)/u_{n-1}^*\bX_{n-1}(A)\rightarrow\bX_{m+n+1}(A)/u_{m+n}^*\bX_{m+n}(A)$ is also injective.  From Proposition \ref{Psi_ij} , the family $\{\Psi_{ij}^{[m]}\}_{ij}$ forms an $R$-basis of $\Hom(N^{[m]n}A,\hG)$, where $1\leq i\leq n$ and $1\leq j\leq g.$ Hence $\Hom(N^{[m]n}A,\hG)/\Hom(N^{[m]n-1}A,\hG)$ has the basis $\{\Psi_{nj}^{[m]}\}$  and $\Hom(N^{[m]n+1}A,\hG)/\Hom(N^{[m]n}A,\hG)$ has the basis $\{\Psi_{(n+1)j}^{[m]}\}$ for all $1\leq j\leq g.$ By definition it is clear that $$\mff_m^*(\Psi_{nj}^{[m]})=\Psi_{nj}^{[m]}\circ \mff_m=\Psi_{(n+1)j}^{[m]}$$
		for all $1\leq j\le g$ and $m\geq 0.$		Hence this implies $\mff_m^*$ is an isomorphism of $R$-modules. For $n\geq \beta_0$, Theorem 8.5 of \cite{PaSa}, gives us that $\phi^*:\bX_{n}(A)/u_{n-1}^*\bX_{n-1}(A)\rightarrow\bX_{n+1}(A)/u_n^*\bX_{n}(A)$ is an isomorphism of $R$-modules and hence $\phi^{*(m+1)}:\bX_{n}(A)/u_{n-1}^*\bX_{n-1}(A)\rightarrow\bX_{m+n+1}(A)/u_{m+n}^*\bX_{m+n}(A)$ is also an isomorphism of $R$-modules.
	\end{proof}
	
	
	\section{System of Filtered Isocrystals}
	
	Let $A$ be a $\pi$-formal abelian scheme of relative dimension $g$ over $S$, maintaining our standing assumptions. In this section, we construct a system of filtered isocrystals $\Iso(\mH_{[m]}(A))$ for every $m\geq 0$, extending the construction of \cite{BoSa1}.
	
	\subsection{The connecting map and splitting numbers}
	\label{m_prim_diff_char}
	
	
	For any non-negative integers $m$ and $n$, recall the generalized kernel sequence delineated in \eqref{Ker_A}:
	$$
	0\rightarrow N^{[m]n}A\xrightarrow{\iota_m}J^{m+n}A\xrightarrow{u_m}J^mA\rightarrow 0.
	$$
	Applying the contravariant functor $\Hom(-,\hG)$ yields the fundamental exact sequence of $R$-modules as in \eqref{Hom_Ker_A}:
	$$
	0\rightarrow \bX_m(A)\xrightarrow{u_m^*}\bX_{m+n}(A)\xrightarrow{\iota_m^*}\Hom( N^{[m]n}A,\hG)\xrightarrow{\partial}\Ext(J^mA,\hG).
	$$
	We define the image of the connecting homomorphism to be the submodule $\bI_{[m]n}(A) := \text{Image}(\partial)$. This canonical construction naturally produces the following short exact sequence:
	\begin{equation}\label{Exact_Hom_Ker}
		0\rightarrow \bX_m(A)\xrightarrow{u_m^*}\bX_{m+n}(A)\xrightarrow{\iota_m^*}\Hom( N^{[m]n}A,\hG)\xrightarrow{\partial}\bI_{[m]n}(A)\rightarrow 0.
	\end{equation}
	
	Observe that for all integers $m,n\geq 0$, the projection $u$ induces an injection $u^*:\Hom(N^{[m]n}A,\hG)\hookrightarrow\Hom(N^{[m]n+1}A,\hG)$. This forces an ascending filtration of submodules $\bI_{[m]n}(A)\subseteq \bI_{[m]n+1}(A)$. We define the stable image module via the direct limit:
	$$\bI_{[m]}(A):=\varinjlim \bI_{[m]n}(A).$$ 
	To measure the sequential growth of these modules, we define the successive rank differences:
	$$h_{[m]i}=\rk_R(\bI_{[m]i}(A))-\rk_R(\bI_{[m]i-1}(A)) \quad \text{for all } i\geq 1.$$ 
	
	We define the \textit{$m$-th upper splitting number} to be the minimal integer $\beta_m\geq 1$ such that $h_{[m]n}=0$ for all $n\geq \beta_m$. It is pertinent to note that $\beta_0$ coincides precisely with the invariant $m_{\mathrm{u}}$
	investigated in \cite{BoSa1}. The existence of this stabilizing integer is guaranteed by the ascending chain condition on the finite-rank $R$-module:
	$$\bI_{[0]0}(A)\subseteq \bI_{[0]1}(A)\subseteq \cdots \subseteq 
	\Ext(A,\hG) \simeq H^1(A,\mO_A) \simeq R^g,$$ 
	where the terminal canonical isomorphism is derived from the classical theory of algebraic groups \cite[Page 185]{Se1}.
	
	In a dual manner, we define the \textit{lower splitting number} to be the unique integer $m_l$ characterized by the non-vanishing condition $\bX_{m_l}(A) \neq\{0\}$ subject to the strict vanishing $\bX_{m_l-1}(A) =\{0\}$. By virtue of Proposition 7.1 in \cite{BoSa1}, the value of $m_l$ is rigorously constrained to either $1$ or $2$.

	We define a $\delta$-character $\Theta\in \bX_{n}(A)_K$ to be \textit{primitive} of order $n$ if $$\Theta\notin (u_{n-1}^*)\bX_{n-1}(A)_K+\phi^{*}(\bX_{n-1}(A)_K)$$ 
	
	and we call $\Theta\in \bX_{m+n}(A)_K$ \textit{$m$-th primitive} if $$\Theta\notin (u_{n-1}^*)\bX_{n-1}(A)_K+\phi^{*(m+1)}(\bX_{n-1}(A)_K).$$
	For $i\geq 1$, a subset $\BB_i^{[m]}\subseteq \bX_{m+i}(A)_K$ is called the $m$-th primitive basis for $\bX_{m+i}(A)_K$ if its image in the quotient ${\bX_{m+i}(A)_K}/{\phi^{*(m+1)}\bX_{n-1}(A)_K}$ is a $K$-basis for the quotient $K$-module. 
	\begin{proposition}\label{lin_ind_S_n}
		For all $n$, $m\geq 0$ and $i_1< i_2< \cdots < i_l< n$, the set  $S_n(\BB_{i_1}^{[m]})\cup S_n(\BB_{i_2}^{[m]})\cup \cdots \cup S_n(\BB_{i_l}^{[m]})$ is $K$-linearly independent in $\bX_{m+n}(A)_K$.
	\end{proposition}
	\begin{proof}
		We argue by contradiction. Suppose that the set $S_n(\BB_{i_1}^{[m]})\cup \cdots \cup S_n(\BB_{i_l}^{[m]})$ is not $K$-linearly independent, then there exists a $K$-linear relation among the elements. Further, we may assume that $n$ is the minimal number for which there exists a non-trivial $K$-linear relation among the elements of $S_n(\BB_{i_1}^{[m]})\cup \cdots \cup S_n(\BB_{i_l}^{[m]})$. In other words, $n$ is the smallest number such that there are no non-trivial $K$-linear relations among the elements in $S_{n-1}(\BB_{i_1}^{[m]})\cup \cdots \cup S_{n-1}(\BB_{i_l}^{[m]})$. 
		
		Thus there exist constants $a_hs$, not all zero, such that \begin{equation}\label{01}
			\sum_{h=1}^l\left(\sum_{\Theta_{i_hs}\in\BB^{[m]}_{i_h}}a_{hs}\iota_m^*\phi^{*(m+n-i_h)}\Theta_{i_hs}\right)\equiv 0\mod\Hom(N^{[m]n-1}A,\hG).
		\end{equation}
		Let $l'\in\{1,\cdots,l\}$ be such that $a_{l'h_0}\neq 0$ for all $j>l$ and $a_{jh}=0$ for all $j>l'$ and all $h=1,\cdots,\mult(i_j).$ Therefore $i_{l'}$ is the largest order among the primitive elements that appear in the expression above with some non-zero coefficient. Hence we can truncate the left hand side of the expression and sum only up to $l'$ instead of $l$. Then we have \begin{equation}\label{02}
			\sum_{h=1}^{l'}\left(\sum_{\Theta_{i_hs}\in\BB^{[m]}_{i_h}}a_{hs}\iota_m^*\phi^{*(m+n-i_h)}\Theta_{i_hs}\right)\equiv 0\mod\Hom(N^{[m]n-1}A,\hG).
		\end{equation}
	$R$ is a $\pi$-adically complete discrete valuation ring and from Lemma \ref{phi_bij}, we know  $\phi$ to be an automorphism on $R$. Hence, we define the character \begin{equation}
			\Theta:=\sum_{h=1}^{l'}\left(\sum_{\Theta_{i_hs}\in\BB^{[m]}_{i_h}}a_{hs}^{-\phi^{-(n-i_{l})}}\phi^{*(m+i_{l'}-i_h)}\Theta_{i_hs}\right)\in \bX_{m+i_{l'}}(A)_K.
		\end{equation}
		Note that $\mff_m^{*(n-i_{l'})}\iota_m^*\Theta$ is the left hand side of the expression in (\ref{02}) and therefore we have \begin{equation}\label{04}
			\mff_m^{*(n-i_{l'})}\iota_m^*\Theta\equiv 0\mod\Hom(N^{[m]n-1}A,\hG).
		\end{equation}
		Observe that the image of $\Theta$ in the quotient $\bX_{m+i_{l'}}(A)_K/(u_m^*\bX_{i_{l'}-1}(A)_K+\phi^{*m}\bX_{i_{l'}-1}(A)_K)$ is non-zero because the image of $\BB^{[m]}_{i_{l'}}$ generates the quotient as a basis and at least one of the coefficients $a_{l'h_0}$ is non-zero. Therefore, we must have $\Theta\in \bX_{m+i_{l'}}(A)_K/\bX_{i_{l'}-1}(A)_K$ and hence  $$\iota_m^*\Theta\not\equiv 0\mod\Hom(N^{[m]i_{l'}-1}A,\hG).$$ Recall from Theorem \ref{isom_phi} that $$\mff_m^{*(n-i_{l'})}(\phi^{*m})\iota_m^*\Theta\equiv \iota_m^*\phi^{*(m+n-i_{l'})}\Theta\mod\Hom(N^{[m]n-1}A,\hG).$$
		But then the above implies \begin{equation}
			\mff_m^{*(n-i_{l'})}(\phi^{*m})\iota_m^*\Theta\not\equiv 0\mod\Hom(N^{[m]n-1}A,\hG)
		\end{equation} 
		which is a contradiction to (\ref{04}) and we are done.
	\end{proof}       \color{black}
	Let us recall some  results  regarding the $R$-module $\bXp^{[0]}(A)$.

	\begin{lemma}\label{Xm+n}
		The $R$-module $\bX_{n}(A)$ is freely generated by $S_n(\BB_1^{[0]})\cup  \cdots \cup S_n(\BB_n^{[0]})$.
	\end{lemma}
	\begin{proof}
		We refer to Lemma 8.6 of \cite{PaSa} for the proof.
	\end{proof}
	\begin{theorem}\label{X_inf}
		For any $\pi$-formal abelian scheme $A$ of dimension $g$, $\bX_\infty(A)$ is freely $R\{\phi^*\}$-generated by $g$ primitive $\delta$-characters of highest order $\beta_0$ and $\beta_0\leq g +1$.
	\end{theorem}
	\begin{proof}
		We refer to Theorem 7.6 of \cite{BoSa1} and Theorem 8.7 of \cite{PaSa} for the proof.
	\end{proof}
	
	Let $\Theta_1,\cdots,\Theta_g$ be the primitive characters generating the $R$-module $\bXp^{[0]}(A)$ and $o_i$ be the exact order of $\Theta_i\in \bX_{o_i}(A).$ Without loss of generality, we can assume $o_1\leq o_2\leq\cdots\leq o_g=\beta_0$ and denote $o=o_1+\cdots+o_g$ to be the total primitive order of $A$. Then we have \begin{equation}
		\bXp^{[0]}(A)= R\langle [\Theta_1],\cdots,[\Theta_g]\rangle,
	\end{equation}
	where $[\Theta_i]$ denotes the class of $\Theta_i$ in $\bXp^{[0]}(A)$ for all $i\in \{1,\cdots,g\}$. 
	\begin{theorem}\label{X^m_prim}
		For all $m\geq 0$
		$$
		\bXp^{[m]}(A)= R\langle [\Theta_1],\cdots,[\Theta_g],\cdots, \phi^{*(m)}[\Theta_1],\cdots,\phi^{*(m)}[\Theta_g]\rangle.
		$$
		Moreover $\bXp^{[m]}(A)$ is a free $R$-module of rank $(m+1)g.$
	\end{theorem}
	
	\begin{proof}
		By definition $$\bXp^{[m]}(A)=\varinjlim_n\frac{\bX_{m+n}(A)}{\phi^{*(m+1)}\bX_{n-1}(A)}=\frac{\bX_\infty(A)}{\phi^{*(m+1)}\bX_\infty(A)}.$$
		From Theorem \ref{X_inf}, we have $\bX_\infty(A)= R\{\phi^*\}\langle\Theta_1,\cdots,\Theta_g\rangle.$ Hence this implies 
		\begin{align}
			\bXp^{[m]}(A)= \frac{\bX_\infty(A)}{\phi^{*(m+1)}\bX_\infty(A)}&=\frac{R\langle \Theta_1,\cdots,\Theta_g,\phi^*\Theta_1,\cdots,\phi^*\Theta_g,\cdots \rangle }{\phi^{*(m+1)}(R\langle \Theta_1,\cdots,\Theta_g,\phi^*\Theta_1,\cdots,\phi^*\Theta_g,\cdots \rangle) }\\
			&= R\langle [\Theta_1],\cdots,[\Theta_g],\cdots, \phi^{*(m)}[\Theta_1],\cdots,\phi^{*(m)}[\Theta_g]\rangle.
		\end{align}
	\end{proof}
	
	\begin{corollary}\label{phiXprim}
		For all $m\geq 0$, $$\bXp^{[m]}(A)=\sum_{i=0}^m\phi^{*(i)}\bXp^{[0]}(A).$$
	\end{corollary}
	\begin{proof}
		We have $$\bXp^{[0]}(A)= R\langle [\Theta_1]\cdots,[\Theta_g]\rangle$$
		and hence for all $i\geq 0$,  $$\phi^{*(i)}\bXp^{[0]}(A)= R\langle \phi^{*(i)}[\Theta_1]\cdots,\phi^{*(i)}[\Theta_g]\rangle.$$
		
		Therefore the above assertion follows directly from Theorem \ref{X^m_prim}.
	\end{proof}	
	
	\color{black}
	
	Let us define the $R$-module
	$$\DD_{[m]}(A):= \frac{\bXp^{[m]}(A)}{u_m^*\bX_m(A)},$$
	where $u_m^*$ is as in Theorem \ref{map_of_seq}.
	Then we can rewrite the exact sequence (\ref{Exact}) as
	\begin{equation}\label{exact_sequence_Hm_X0}
		\xymatrix{
			0\ar[r] & \DD_{[m]}(A)\ar[r] & \mH^\delta_{[m]}(A)\ar[r] & \bI_{[m]}(A)\ar[r] & 0.
		}
	\end{equation}

	\begin{corollary}
		\label{rk_D_m}
		The following statements are true.
		\begin{enumerate}
			\item For all $m\geq 0$, we have $\rk_R\DD_{[m]}(A)\geq g$.
			\item For all $m\geq \beta_0-1,$ $$\DD_{[m]}(A)= R\langle[\phi^{*(m-o_1+1)}\Theta_1],\cdots,[\phi^{*(m)}\Theta_1],\cdots,[\phi^{*(m-o_g+1)}\Theta_g],\cdots,[\phi^{*(m)}\Theta_g] \rangle,$$ where for all $1\leq j\leq g$ and $m-o_j+1\leq k_j\leq m$,  $[\phi^{*(k_j)}\Theta_j]$ denotes the class of  $\phi^{*(k_j)}\Theta_j$ in $\DD_{[m]}(A).$ Moreover, $\rk_R(\DD_{[m]}(A))=o_1+\cdots+o_g=o$ for all $m\geq \beta_0-1.$
		\end{enumerate} 
	\end{corollary}
	\begin{proof}
		By definition, we have $\DD_{[m]}(A)=\bXp^{[m]}(A)/u_m^*\bX_m(A)$. From Theorem \ref{X^m_prim}, we know that $\rk_R(\bXp^{[m]}(A))=(m+1)g$ and from the exact sequence 
		$$
		0\rightarrow \bX_m(A)\rightarrow \Hom(N^{[0]m}A,\hG)\rightarrow \bI_{[0]m}(A)\rightarrow 0,
		$$ we can conclude  that $\rk_R(\bX_m(A))\leq mg$. Hence $$\rk_R(\DD_{[m]}(A))=\rk_R(\bXp^{[m]}(A))-\rk_R(\bX_m(A))\geq (m+1)g-mg=g.$$
		For all $m\geq \beta_0$, we have $$\bX_m(A)= R\langle \Theta_1,\cdots,\phi^{*(m-o_1)}\Theta_1,\cdots, \Theta_g,\cdots,\phi^{*(m-o_g)}\Theta_g\rangle.$$ Hence from Lemma \ref{X^m_prim}, 
		\begin{align}
			\DD_{[m]}(A)&=\frac{\bXp^{[m]}(A)}{u_m^*\bX_{m}(A)}=\frac{R\langle [\Theta_1],\cdots,[\Theta_g],\cdots, [\phi^{*(m)}\Theta_1],\cdots,[\phi^{*(m)}\Theta_g]\rangle}{R\langle [\Theta_1],\cdots,[\phi^{*(m-o_1)}\Theta_1],\cdots, [\Theta_g],\cdots,[\phi^{*(m-o_g)}\Theta_g]\rangle}\\
			&= R\langle[\phi^{*(m-o_1+1)}\Theta_1],\cdots,[\phi^{*(m)}\Theta_1],\cdots,[\phi^{*(m-o_g+1)}\Theta_g],\cdots,[\phi^{*(m)}\Theta_g] \rangle.
		\end{align}
		The last equality holds because $\bXp^{[m]}(A)$ and $\bX_m(A)$ are free $R$-modules and the chosen basis of $\bX_{m}(A)$ as a free $R$-module, extends to a basis of $\bXp^{[m]}(A)$. Therefore for all $m\geq \beta_0$, $$\rk_R(\DD_{[m]}(A))=o_1+\cdots+o_g=o.$$
		From Lemma \ref{Xm+n} we obtain,  $$\bX_{\beta_0-1}(A)= R\langle \Theta_1,\cdots, \phi^{(\beta_0-o_1-1)}\Theta_1,\cdots,\Theta_{g-1},\cdots,\phi^{*(\beta_0-o_{g-1}-1)}\Theta_{g-1}\rangle.$$ Therefore, from Theorem \ref{X^m_prim}, we can conclude \begin{align}
			\DD_{[\beta_0-1]}(A)&=\frac{\bXp^{[\beta_0-1]}(A)}{u_{\beta_0-1}^*\bX_{[\beta_0-1]}(A)}=\frac{R\langle [\Theta_1],\cdots,[\Theta_g],\cdots, [\phi^{*(\beta_0-1)}\Theta_1],\cdots,[\phi^{*(\beta_0-1)}\Theta_g]\rangle}{R\langle [\Theta_1],\cdots,[\phi^{*(\beta_0-1-o_1)}\Theta_1],\cdots, [\Theta_g],\cdots,[\phi^{*(\beta_0-1-o_g)}\Theta_g]\rangle}\\
			&= R\langle[\phi^{*(\beta_0-o_1)}\Theta_1],\cdots,[\phi^{*(\beta_0-1)}\Theta_1],\cdots,[\Theta_g],\cdots,[\phi^{*(\beta_0-1)}\Theta_g] \rangle
		\end{align} and hence $\rk_R(\DD_{[\beta_0-1]}(A))=o_1+\cdots+o_g=o.$ 
	\end{proof}
	
	\cblue{
		
		\subsection{Filtered isocrystals in $p$-adic Hodge theory}
		\label{m_delta_Iso}
We briefly recall some of the basic definitions of filtered isocrystals
		in $p$-adic Hodge theory.
		For a comprehensive exposition, we direct the reader to the excellent survey by Brinon and Conrad \cite{BriCon}.
		
		Let $R$ be a commutative ring. A \textit{filtered module} over $R$ is defined as an $R$-module $M$ endowed with a descending filtration, which is a specified collection of submodules $\{M_i\}_{i\in\mathbb{Z}}$ satisfying the containment condition $M_{i+1}\subseteq M_i$ for all integers $i\in \mathbb{Z}$. 
		
		Over the fraction field $K$, we define a \textit{non-degenerate isocrystal} to be a finite-dimensional $K$-vector space $D$ equipped with a bijective, Frobenius-semilinear endomorphism $\phi_D:D\rightarrow D$. 
		
Combining the above two definitions, a \textit{filtered isocrystal} over $K$ is formalized as a tuple $(D, \phi_D, \{D_i\})$, wherein $(D, \phi_D)$ constitutes a non-degenerate isocrystal over $K$, and $\{D_i\}_{i\in \ZZ}$ provides a descending filtration on the underlying $K$-vector space $D$.
		
		One of the 
		principal objectives of the following discussion is to establish that, for any fixed integer $m \geq 0$, $\mH_{[m]}^\delta(A)_K$ canonically admits the structure of a filtered isocrystal.
	}
	
	\begin{proposition}\label{H_m^delta}
		The morphism $$u^*:\mH_{[m]n}^\delta(A)\longrightarrow \mH_{[m]n+1}^{\delta}(A)$$
		is injective. For all $n\geq \beta_0$, it is a bijection.
	\end{proposition}
	
	\begin{proof}
		Let us consider the following diagram of exact sequences
		
		\begin{equation}
			\xymatrix{
				& 0& 0& & \\
				& \frac{(\bX_n(A))}{u_{n-1}^*(\bX_{n-1}(A))}\ar[r]^-{\iota_m^*(\phi^*)^{m+1}}\ar[u] & \frac{\Hom(N^{[m]n+1}A,\hG)}{\Hom(N^{[m]n}A,\hG)}\ar[u]& & \\
				0\ar[r] & \bX_n(A)\ar[r]^-{\iota_m^*(\phi^*)^{m+1}}\ar[u] & \Hom(N^{[m]n+1}A,\hG)\ar[r]\ar[u]& \bH_{[m]n+1}^\delta(A)\ar[r] & 0\\ 
				0\ar[r] & \bX_{n-1}(A)\ar[r]^-{\iota_m^*(\phi^*)^{m+1}}\ar[u]^{u^*} & \Hom(N^{[m]n}A,\hG)\ar[r]\ar[u]^{u^*}& \bH_{[m]n}^\delta(A)\ar[r]\ar[u]^{u^*} & 0\\
				& 0\ar[u]& 0\ar[u]& &
			}
		\end{equation}
		Then $\iota_m^*(\phi^*)^{m+1}:\bX_{n}(A)/\bX_{n-1}(A)\rightarrow \Hom(N^{[m]n+1}A,\hG)/\Hom(N^{[m]n}A,\hG)$ is injective by Theorem \ref{isom_phi} and hence by the snake lemma, $\bH_{[m]n}^\delta(A)\xrightarrow{u^*}\bH_{[m]n+1}^\delta(A)$ is injective for all $n$.
		
		It remains to show that $u^*:\mH_{[m]n}^{\delta}(A)\rightarrow \mH_{[m]n+1}^\delta(A)$ is surjective for all $n\geq \beta_0$. From Lemma \ref{Xm+n}, putting $m=0$, we have $\bX_n(A)/u^*\bX_{n-1}(A)$ is freely $R$-generated by the images of $(S_n(\BB_1^{[0]})\setminus S_{n-1}(\BB_1^{[0]}))\cup\cdots\cup (S_n(\BB_{\beta_0}^{[0]})\setminus S_{n-1}(\BB_{\beta_0}^{[0]}))$ for all $n\geq \beta_0$. Hence  $\rk(\bX_n(A)_K/u^*\bX_{n-1}(A))=g=\rk(\Hom(N^{[m]n+1}A,\hG)/\Hom(N^{[m]n}A,\hG)),$ which implies $\iota_m^*(\phi^*)^{m+1}$ is surjective for $n\geq \beta_0$ and therefore $u^*:\bH_{[m]n}^\delta(A)\rightarrow \bH_{[m]n+1}^{\delta}(A)$ is surjective.
	\end{proof}

	\begin{theorem}\label{filt_isocrys}
		Let $A$ be a $\pi$-formal abelian scheme of dimension $g$ over $R$. 
		\begin{enumerate}
			\item[(i)]For all $m\geq 0$,  $\mH_{[m]}^\delta(A)$ is a free $R$-module with $g\leq \rk_R\mH_{[m]}^\delta(A)\leq 2g$. 
			\item [(ii)] For any fixed $m\geq 0$,  $\mH_{[m]}^\delta(A)_K$ is  a filtered isocrystal with filtration \begin{equation}\label{filt_0}
				\mH_{[m]}^\delta(A)_K\supseteq  \DD_{[m]}(A)_K\supseteq \{0\}
			\end{equation}
			and semi-linear endomorphism $$\mff_m^*:\mH_{[m]}^\delta(A)\rightarrow\mH_{[m]}^\delta(A).$$
		\end{enumerate}
	\end{theorem}
	\begin{proof}

		From the exact sequence (\ref{exact_sequence_Hm_X0}), we have 
		\begin{align}
			\rk_R(\mH_{[m]}^\delta(A))&=\rk_R(\DD_{[m]}(A))+\rk_R(\bI_{[m](A)})\\
			&\geq g+\rk_R(\bI_{[m](A)})\text{ (From Corollary \ref{rk_D_m})}.
		\end{align}
		Hence, we have $\rk_R(\mH_{[m]}^\delta(A))\geq g $ for all $m\geq 0$. For $m=0$, we know that $\rk_R(\bI_{[0](A)})\leq \rk_R(\Ext(A,\hG))=g.$
		Therefore $\rk_R(\mH_{[0]}^\delta(A))\leq 2g.$ From Corollary \ref{rk_H_del}, we obtain $\rk_R(\mH_{[m+1]}^\delta(A))\leq \rk_R(\mH_{[m]}^\delta(A))$ for all $m\geq 0$, hence for all $m\geq 1$, $\rk_R(\mH_{[m]}^\delta(A))\leq \mH_{[0]}^\delta(A)\leq 2g.$ Therefore this gives us  $$g\leq \rk_R(\mH_{[m]}^\delta(A))\leq 2g$$ for all $m\geq 0.$
		
		Hence for a fixed $m$, $\mH_{[m]}^\delta(A)_K$ is a finite dimensional vector space over $K$. The filtration above follows from the exact sequence in equation (\ref{exact_sequence_Hm_X0}) and $\mff_m^*$ is defined as in (\ref{Semilin_end}). Hence $\mH_{[m]}^\delta(A)_K$ is a filtered isocrystal over $K$.
	\end{proof}
	We define $$\Iso(\mH_{[m]}^\delta(A)):=(\mH_{[m]}^\delta(A)_K,\mff_m^*,\mH_{[m]}^\delta(A)_K^\bullet),$$ as an object in the category of filtered isocrystals where $\mff_m^*$ is the semilinear operator on $\mH_{[m]}^\delta(A)_K$ and $\mH_{[m]}^\delta(A)_K^\bullet$ denotes the filtration given by $\mH_{[m]}^\delta(A)_K\supseteq \DD_{[m]}(A)_K\supseteq \{0\}$. We call these filtered isocrystals as $m$-th delta isocrystals.

	Finally we would like to show that for any fixed $m\geq 0$, the $m$-th delta isocrystal $\Iso(\mH_{[m]}^\delta(A))$ is a non-degenerate isocrystal over $K$, that is, we claim that the semilinear endomorphism $$\mff_m^*:\mH_{[m]}^\delta(A)_K\rightarrow \mH_{[m]}^\delta(A)_K$$ is  a bijection.
	By Proposition \ref{H_m^delta}, we have $$\mH_{[m]}^\delta(A)_K\simeq \mH_{[m]\beta_0}^\delta(A)_K=\frac{\Hom(N^{[m]\beta_0}A,\hG)_K}{\iota_m^*\phi^{*(m+1)}(\bX_{\beta_0-1}(A)_K)}$$
	From Corollary \ref{comm_phi*1} we obtain the following exact sequence 
	\begin{equation}
		\xymatrix{
			0\ar[r]& \DD_{[m]}(A)_K\ar[r]^-{\iota_m^*} & \mH_{[m]}^\delta(A)_K\ar[r]^-{\partial}& \bI_{[m]}(A)_K\ar[r] & 0
		}
	\end{equation}
	
	
	Let $\{\Theta_1,\cdots, \Theta_g\}$ be a set of primitive characters of $\bX_\infty(A)$ and $o_i$ be the exact order of $\Theta_i$. We also assume $o_1\leq \cdots\leq o_g=\beta_0.$ Hence from Theorem \ref{X^m_prim}, we have
	$$
	\bXp^{[m]}(A)_K\simeq K\langle \Theta_1,\cdots,\Theta_g,\cdots, \phi^{*(m)}\Theta_1,\cdots,\phi^{*(m)}\Theta_g\rangle.
	$$
	
	For any $\Theta\in \bX_{\beta_0}(A)$,	let us denote $\tilde{\Theta}^{[m]}:=\phi^{*(m+\beta_0-r_\Theta)}\Theta, $ where $r_\Theta$ is the exact order of $\Theta$.  Hence $\tilde{\Theta}_i^{[m]}\in \bX_{m+\beta_0}(A)\setminus\bX_{m+\beta_0-1}(A) $ for every $1\leq  i\leq g.$ 
	
	For any $\lambda=(\lambda_1,\cdots,\lambda_g)\in Mat_{1\times g}(R)$ and $\Psi=(\Psi_1,\cdots,\Psi_g),$ where $\Psi_i\in \Hom(N^{[m]n}A,\hG)$, let us denote $\lambda.\Psi=\lambda_1\Psi_1+\cdots+\lambda_g.\Psi_g.$ Then \begin{equation}
		\iota_m^*\tTheta_i^{[m]}=\lambda_{i\beta_0}^{[m]}.\Psi_{\beta_0}^{[m]}-\lambda_{i(\beta_0-1)}^{[m]}.\Psi_{\beta_0-1}^{[m]}-\cdots-\lambda_{i1}^{[m]}.\Psi_1^{[m]}\in \Hom(N^{[m]n}A,\hG)
	\end{equation}
	for some $\lambda_{ij}^{[m]}\in Mat_{1\times g}(R)$ and $\Psi_{ij}^{[m]}$ as defined in (\ref{Psiijm}) for all  $1\leq j\leq \beta_0$ and $1\leq i\leq g$.  Applying $\mff_m^*$ to both sides we have \begin{equation}\label{f_mi_m}
		\mff_m^*\iota_m^*\tTheta_i^{[m]}=(\lambda_{i\beta_0}^{[m]})^\phi.\Psi_{\beta_0+1}^{[m]}-(\lambda_{i(\beta_0-1)}^{[m]})^\phi.\Psi_{\beta_0}^{[m]}-\cdots-(\lambda_{i1}^{[m]})^\phi.\Psi_2^{[m]}.
	\end{equation} 
	From Proposition \ref{Char_n}, we have
	\begin{equation}\label{i_mphi_m}
		\iota_m^*\phi^{*(m+1)}\tTheta_i^{[0]}=\iota_m^*\tTheta_i^{[m+1]}=\mff_m^*\iota_m^*\tTheta_i^{[m]}+\gamma_i^{[m]}.\Psi_1^{[m]}
	\end{equation}
	where $\gamma_i^{[m]}=\pi^{m+1}A_{i0}$ and $D\tTheta_i^{[0]}=(A_{i0},\cdots, A_{i\beta})$
	for some $A_{ij}\in Mat_{1\times g}(R).$
	Hence from (\ref{f_mi_m}) and (\ref{i_mphi_m}) we obtain \begin{equation}\label{i_mphi_m+1}
		\iota_m^*\phi^{*(m+1)}\tTheta_i^{[0]} =(\lambda_{i\beta_0}^{[m]})^\phi.\Psi_{\beta_0+1}^{[m]}-(\lambda_{i(\beta_0-1)}^{[m]})^\phi.\Psi_{\beta_0}^{[m]}-\cdots-(\lambda_{i1}^{[m]})^\phi.\Psi_2^{[m]}+\gamma_i^{[m]}.\Psi_1^{[m]}.
	\end{equation}
	
	For all $j=1,\cdots,\beta_0$, we define the following $g\times g$ matrices
	\begin{equation}
		\Delta_j^{[m]}:=\begin{pmatrix}
			(\lambda_{1j}^{[m]})^\phi\\
			\vdots\\
			(\lambda_{gj}^{[m]})^\phi
		\end{pmatrix}
	\end{equation}
	and
	\begin{equation}
		\Gamma^{[m]}:=\begin{pmatrix}
			\gamma_{1}^{[m]}\\
			\vdots\\
			\gamma_{g}^{[m]}
		\end{pmatrix}.
	\end{equation}
	Hence the induced map on $K$-modules $$[\iota_m^*\phi^{*(m+1)}]:\frac{\bX_{\beta_0}(A)_K}{u_0^*\bX_{\beta_0-1}(A)_K}\rightarrow \frac{ \Hom(N^{[m]\beta_0+1}A,\hG)_K}{u_m^*\Hom(N^{[m]\beta_0}A,\hG)_K}$$ for all $n\geq \beta_0$ is given by $[\iota_m^*\phi^{*(m+1)}]\tTheta_i^{[0]}=(\lambda_{i\beta_0}^{[m]})^\phi.\Psi_{\beta_0+1}^{[m]}$. Therefore, $[\iota_m^*\phi^{*(m+1)}]$ in terms of the basis $\{\tTheta_i^{[0]}\}_{i=1}^g$ of $\frac{\bX_{\beta_0}(A)_K}{u_0^*\bX_{\beta_0-1}(A)_K}$ and $\Psi_{\beta_0+1}^{[m]}=\{\Psi_{(\beta_0+1)1}^{[m]},\cdots,\Psi_{(\beta_0+1)g}^{[m]}\}$ for $\frac{\Hom(N^{[m]\beta_0+1}A,\hG)_K}{\Hom(N^{[m]\beta_0}A,\hG)_K}$  
	can be represented by the matrix \begin{equation}
		[\iota_m^*\phi^{*(m+1)}]=\Delta_{\beta_0}^{[m]}.
	\end{equation}
	From the proof of Proposition \ref{H_m^delta}, $[\iota_m^*\phi^{*(m+1)}]$ is an isomorphism and that implies the matrix $\Delta_\beta^{[m]}$ is invertible.
	Hence from (\ref{i_mphi_m+1}), we have \begin{equation}\label{psi_beta+1}
		\Psi^{[m]}_{\beta_0+1}\equiv (\Delta_{\beta_0}^{[m]})^{-1}(\Delta_{\beta_0-1}^{[m]}\Psi_{\beta_0}^{[m]}+\cdots+\Delta_{1}^{[m]}\Psi_2^{[m]}-\Gamma^{[m]}\Psi_1^{[m]}) \mod \iota_m^*\phi^{*(m+1)}(\bX_{\beta_0-1}(A)_K).
	\end{equation} 
	
	
	\begin{theorem}\label{non-deg}
Suppose $R = W(\kappa)$. Let $m\geq 0$ and $A$ be a $\pi$-formal abelian scheme of relative dimension $g$. Then the operator $$\mff_m^*:\mH_{[m]}^\delta(A)_K\rightarrow \mH_{[m]}^\delta(A)_K$$ is bijective. Hence $\mH_{[m]}^\delta(A)_K$ is a non-degenerate filtered isocrystal.
	\end{theorem}
	
	\begin{proof}
		Recall from Proposition \ref{phm_iso}, we obtain $$\mH_{[m]}^\delta(A)_K\simeq \mH_{[m]\beta_0}^\delta(A)_K=\frac{\Hom(N^{[m]\beta_0}A,\hG)_K}{\iota_m^*\phi^{*(m+1)}(\bX_{\beta_0-1}(A)_K)}$$ and $\Hom(N^{[m]\beta_0}A,\hG)_K\simeq K\langle\Psi_1^{[m]},\cdots,\Psi_{\beta_0}^{[m]}\rangle.$ Hence $$\mH_{[m]}^\delta(A)_K\simeq \frac{K\langle\Psi_1^{[m]},\cdots,\Psi_{\beta_0}^{[m]}\rangle}{\iota_m^*\phi^{*(m+1)}(\bX_{\beta_0-1}(A)_K)}.$$
		From the definition of $\Psi_i^{[m]}$ and equation (\ref{psi_beta+1}) we get \begin{align}
			& \mff_m^*\Psi_i^{[m]}=\Psi_{i+1}^{[m]} \text{ for all $i=1,\cdots,\beta_0-1$,}\\
			& \mff_m^*\Psi_{\beta_0}^{[m]}=\Psi^{[m]}_{\beta_0+1}=(\Delta_{\beta_0}^{[m]})^{-1}(\Delta_{\beta_0-1}^{[m]}\Psi_{\beta_0}^{[m]}+\cdots+\Delta_{1}^{[m]}\Psi_2^{[m]}-\Gamma^{[m]}\Psi_1^{[m]}).
		\end{align}
		Hence the $\beta_0 g\times\beta_0 g$-matrix associated to $\mff_m^*$ with respect to the basis $\{\Psi_1^{[m]},\cdots,\Psi_{\beta_0}^{[m]}\}$ will be 
		\begin{equation}
			[\mff_m^*]=\begin{pmatrix}
				\mathbf{0}_g & \mathbf{0}_g  & \cdots &  -(\Delta_{\beta_0}^{[m]})^{-1}\Gamma^{[m]}\\
				\mathbf{I}_g & \mathbf{0}_g & \cdots & (\Delta_{\beta_0}^{[m]})^{-1}\Delta_1^{[m]} \\
				\mathbf{0}_g & \mathbf{I}_g & \cdots & (\Delta_{\beta_0}^{[m]})^{-1}\Delta_2^{[m]} \\
				\vdots & \vdots & \ddots & \vdots\\
				\mathbf{0}_g & \mathbf{0}_g & \cdots & 
				(\Delta_{\beta_0}^{[m]})^{-1}\Delta_{\beta_0-1}^{[m]} 
			\end{pmatrix}.
		\end{equation}
		Hence, we obtain $\det([\mff_m^*])=(-1)^{\beta_0}\det((\Delta_{\beta_0}^{[m]})^{-1}\Gamma^{[m]})$. Note that $\Gamma^{[m]}$ is the $g\times g$ matrix  
		\begin{equation}
			\label{gamma_mat}
		\Gamma^{[m]}=\begin{pmatrix}
			\gamma_{1}^{[m]}\\
			\vdots\\
			\gamma_{g}^{[m]}
		\end{pmatrix},
		\end{equation}
		 where $\gamma_i^{[m]}=\pi^{m+1}A_{i0}$  and $D\tTheta_i^{[0]}=(A_{i0},\cdots, A_{i\beta_0})$
		for some $A_{ij}\in Mat_{1\times g}(R).$ Therefore for $m=0$, we get $\gamma^{[0]}=\pi A_{i0}$. Therefore we conclude $\gamma_i^{[m]}=\pi^m\gamma_i^{[0]}$ for all $1\leq i\leq g.$ This implies $$\Gamma^{[m]}=\pi^m\Gamma^{[0]}.$$ From equation (6.10) in the proof of Theorem 5.2 in \cite{Delta-crys}, $\Gamma^{[0]}$ is invertible over $K$. Hence this implies $\Gamma^{[m]}$ is invertible over $K$. Therefore the matrix $[\mff_m^*]$ is invertible and hence $\mff_m^*$ is bijective.
	\end{proof}
	
	\section{Transition Morphisms and Stabilization of Filtered Isocrystals}\label{sec:frob_isocrys}

	We now return to the  exact sequence \eqref{Exact_Hom_Ker} for any $m,n\geq 0$:
	$$
	0\rightarrow \bX_m(A)\xrightarrow{u_m^*}\bX_{m+n}(A)\xrightarrow{\iota_m^*}\Hom( N^{[m]n}A,\hG)\xrightarrow{\partial}\bI_{[m]n}(A)\rightarrow 0.
	$$
	In this section, we study the Frobenius compatibility between the filtered isocrystals associated to the different levels of the generalized kernel sequence. In particular, we show that the transition morphisms \(\Phi_{[m]}\) are compatible with Frobenius pullback and the connecting homomorphisms, thereby inducing morphisms between the corresponding filtered isocrystals. We then analyze these transition morphisms and establish the stabilization of the resulting system for sufficiently large \(m\).
	\subsection{Frobenius pullbacks on extensions}\label{subsec:frob_pullbacks}
	For any smooth, commutative $\pi$-formal group scheme $G$, we define the base change along Frobenius as $G_\phi:=G\times_{S,\phi}S$. 
	
	Let $\phi_{\mathrm{rel}}: J^{m}A_\phi \map J^{m-1}A$ denote the relative Frobenius morphism naturally induced by the absolute Frobenius $\phi :J^mA \map J^{m-1}A$. This factors canonically through the commutative diagram:
	\begin{equation}\label{phi_rel}
		\xymatrix{
			J^mA\ar@/^1.5pc/[rrd]^{\phi}\ar@/_1.5pc/[rdd]_-u \ar@{.>}[rd]^{\phi_{\text{rel}}}& & \\
			& J^{m-1}A_\phi=J^{m-1}A\times_{S,\phi}S\ar[r]\ar[d] & J^{m-1}A\ar[d]^{u_S}\\
			& S\ar[r]_{\phi}& S,
		}
	\end{equation}
	which may be evaluated on points as $\phi_{\text{rel}}(\bx)=(\phi(\bx),u_S(\bx))$. Here, for all $n\geq 0$, $u_S:J^nA\rightarrow S$ signifies the structure map naturally prolonged from the base structure map $u_S:A\rightarrow S$.
	
	Via pullback, the absolute Frobenius $\phi$ induces a canonical morphism on 
	extensions, $$\phi^* : \Ext(J^{m-1}A,\hG) \map \Ext(J^m A,\hG),$$ 
	which is constructed explicitly as follows.
	
	Suppose $C$ is an extension of $J^{m-1}A$ by $\hG$. By pulling this extension back via the Frobenius lift on the base $\phi : S \map S$, we obtain a twisted short exact sequence of $\pi$-formal schemes:
	$$
	0 \map \hat{\bb{G}}_{a,\phi} \map C_\phi \map J^{m-1}A_\phi \map 0.
	$$
	Since the additive formal group $\hG$ fundamentally descends to $\Spf \ZZ_p$, it is canonically isomorphic to its own Frobenius base change, providing the identification $\hat{\mathbb{G}}_{a,\phi}=\hG\times_{S,\phi}S\simeq \hG$. Pulling the extension back further along the relative Frobenius map $\phi_{\mathrm{rel}}: J^mA \map J^{m-1}A_\phi$ produces the lower exact row in the following commutative diagram:
	\begin{align}
		\xymatrix{
			0\ar[r]& \hG\ar[r] & C_\phi\ar[r]& J^{m-1}A_\phi\ar[r] & 0\\
			0\ar[r]& \hG\ar[r]\ar@{=}[u] & \phi^*C\ar[r]\ar[u]& J^{m}A\ar[r]\ar[u]^{\phi_{\text{rel}}} & 0,\\
		}
	\end{align}
	where the total space of the pullback extension is given by the Cartesian product $\phi^*C = C_\phi \times_{J^{m-1}A_\phi,\phi_{\mathrm{rel}}} J^m A$. Consequently, the operator $\phi^* : \Ext(J^{m-1}A,\hG) \map \Ext(J^mA,\hG)$ on isomorphism classes is formally dictated by the assignment $[C] \mapsto [\phi^*C]$.

	For any $m,n\geq 0$, let us designate the twisted image of the connecting homomorphism via the map \eqref{maps_of_phm} as the $R$-module:
	$$\bI_{[m]n}(A_{\phi}):=\text{Im}(\partial:\Hom(N^{[m]n}A,\hG)_{\phi}\rightarrow \Ext(J^{m}A_{\phi},\hG)).$$
	
	\begin{theorem}\label{phi*Ext}
		For any $\pi$-formal abelian scheme $A$ over $S$ and every $n\geq 0$, the Frobenius pullback  $\phi^*:\Ext(J^{m-1}A_\phi,\hG)\rightarrow \Ext(J^mA,\hG)$ maps $\bI_{[m-1]n}(A_\phi)$ to $\bI_{[m]n}(A)$  that is $$\phi^*:\bI_{[m-1]n}(A_\phi)\rightarrow \bI_{[m]n}(A)$$ is well defined. More precisely, the following diagram of $R$-modules commutes
		\begin{equation}
			\xymatrix{
				\Hom(N^{[m]n}A,\hG)\ar[r]^-{\partial}& \bI_{[m]n}(A)\ar@{^{(}->}[r] & \Ext(J^mA,\hG)\\
				\Hom(N^{[m-1]n}A,\hG)_\phi\ar[r]^-{\partial}\ar[u]_{\phm^*} & \bI_{[m-1]n}(A_\phi)\ar@{^{(}->}[r]\ar[u]_{\phi^*} & \Ext(J^{m-1}A_\phi,\hG)\ar[u]_{\phi^*}.
			}
		\end{equation}
	\end{theorem}
	
	\begin{proof}
		Let $\Psi\in \Hom(N^{[m-1]n}A,\hG)_\phi$. We need to show that $\phi^*(\partial\Psi)\in \bI_{[m]n}(A)$. By definition the $\partial$ map is given
		by $$\partial\Psi=\Psi_*(J^{m-1+n}A)=\frac{J^{m-1+n}A\times \hG}{\Gamma_{\Psi}(N^{[m-1]n}A)}$$ where $\Gamma_\Psi(N^{[m-1]n}A)=\{(\iota_{m-1}(z),-\Psi(z))|z\in N^{[m-1]n}A\}.$ We claim that $$\phi^*(\partial \Psi)=\partial(\phm^*\Psi)=\partial(\Psi\circ\phm).$$ Note that $$\phi^*(\partial\Psi)=\left(\frac{J^{m-1+n}A\times \hG}{\Gamma_{\Psi}(N^{[m-1]n}A)}\right)_\phi\times_{J^{m-1}A_\phi,\phi_{\text{rel}}}J^mA $$ and $$\partial(\Psi\circ\phm)=\frac{J^{m+n}A\times \hG}{\Gamma_{\Psi\circ\phm}(N^{[m]n}A)}.$$
		Let us consider the following map
		$$\mF:J^{m+n}A\times\hG\longrightarrow \left(\frac{J^{m-1+n}A\times \hG}{\Gamma_{\Psi}(N^{[m-1]n}A)}\right)_\phi\times_{J^{m-1}A_\phi,\phi_{\text{rel}}}J^mA $$
		which sends $\mF(\bx,y)=(([\phi(\bx),y],u(\bx)),u_m(\bx))$. This map is well defined because $u_{m-1}(\phi(\bx))=\phi(u_m(\bx))$ and $u(\bx)=u(u_m(\bx))$. Note that $u:A\rightarrow S$ denotes the structure map of $A$ over $S$.
		
		First we will show that $\mF$ is surjective and then $\Ker\mF=\Gamma_{\Psi\circ \phm}(N^{[m]n}A)$ implying that both extension classes $\phi^*(\partial\Psi)$ and $\partial(\Psi\circ \phm)$ are the same.

		Let $(([\ba,y],s),\cb)\in \left(\frac{J^{m-1+n}A\times \hG}{\Gamma_{\Psi}(N^{[m-1]n}A)}\right)_\phi\times_{J^{m-1}A_\phi,\phi_{\text{rel}}}J^mA.$ Then we have $u_{m-1}(\ba)=\phi_{\text{rel}}(\cb)$ and $\phi(s)=u(a)$ by definition of fibre product. For brevity, we drop the third coordinate as it plays no role in the computations that follow. With a slight abuse of notation, we denote $\phi_{\text{rel}}=\phi$.  Let us consider  $([\ba,y],\cb)\in \left(\frac{J^{m-1+n}A\times \hG}{\Gamma_{\Psi}(N^{[m-1]n}A)}\right)_\phi\times_{J^{m-1}A_\phi,\phi_{\text{rel}}}J^mA.$  We need to find $(\tilde{\ba},\tilde{y})\in J^{m+n}A\times\hG$ such that $\mF(\tilde{\ba},\tilde{y})=([\phi(\tilde{\ba}),\tilde{y}],u_m(\tilde{\ba}))=([\ba,y],\cb)$. Note that we have $\ba\in J^mA$ and $u_{m-1}(\ba)=\phi(\cb)$. Let us choose $\tilde{\ba}\in J^{m+n}A$ to be any preimage of $\cb$. We have the following commutative diagram
		
		\begin{equation}
			\xymatrix{
				\tilde{\ba}\in J^{m+n}A\ar[rr]^-{u_m}\ar[d]_{\phi} & & {J^mA\ni \cb}\ar[d]^{\phi}\\
				\phi(\tilde{\ba})\in J^{m-1+n}A\ar[rr]^-{u_{m-1}} & & {J^{m-1}A\ni u_{m-1}(\ba)}
			}
		\end{equation}
		
		which implies $u_{m-1}(\phi(\tilde{\ba}))=u_{m-1}(\ba)$, that is $\phi(\tilde{\ba})-\ba\in \Ker (u_{m-1})=\text{Im}(\iota_{m-1})$. Hence there exists $z\in N^{[m-1]n}A$, such that  $\phi(\tilde{\ba})=\ba+\iota_{m-1}(z)$. We choose $\tilde{y}=y-\Psi(z)$. Then we have $$\mF(\tilde{\ba},\tilde{y})=([\phi(\tilde{\ba}),\tilde{y}],u_m(\tilde{\ba}))=([\ba+\iota_{m-1}(z),y-\Psi(z)],\cb)=([\ba,y],\cb).$$
		This implies $\mF$ is a surjective map.

		
		Let $(\bx,y)\in \Gamma_{\Psi\circ\phm}(N^{[m]n}A)$, then $(\bx,y)=(\iota_m(z),-\Psi\circ\phm(z))$ for some $z\in N^{[m]n}A$. Then 
		\begin{align}
			\mF(\bx,y)&=\mF(\iota_m(z),-\Psi\circ\phm(z))=([\phi(\iota_m(z)),-\Psi\circ \phm(z)]),u_m(\iota_m(z)))\\ &=([\iota_{m-1}(\phm(z)),-\Psi\circ \phm(z)],0)=([0,0],0)
		\end{align} 
		Hence we have $\Gamma_{\Psi\circ\phm}(N^{[m]n}A)\subseteq \Ker \mF.$ 
		
		Let $(\bx,y)\in \Ker\mF$, then we have $\mF(\bx,y)=([\phi(\bx),y],u_m(\bx))=0.$ Therefore  $(\phi(\bx),y)\in \Gamma_\Psi(N^{[m-1]n}A)$, which implies there exists $z\in N^{[m-1]n}A$ such that $\phi(\bx)=\iota_{m-1}(z)$ and $y=-\Psi(z)$. We also have $u_m(\bx)=0$ and hence $\bx\in $ Im$(\iota_m)$. This implies, there exists $\tilde{z}\in N^{[m]n}A$ such that $\iota_m(\tilde{z})=\bx.$ Hence $\phi(\bx)=\phi\circ \iota_m(\tilde{z})=\iota_{m-1}\circ \phm(\tilde{z})=\iota_{m-1}(z).$ Therefore $\iota_{m-1}$ is injective and hence that implies  $\phm(\tilde{z})=z.$ Hence $y=-\Psi(z)=-\Psi\circ\phm(\tilde{z})$ and this implies $(\bx,y)=(\iota_m(\tilde{z}),-\Psi\circ\phm(\tilde{z}))\in \Gamma_{\Psi\circ\phm}(N^{[m]n}A)$. This implies  $ \Ker \mF\subseteq \Gamma_{\Psi\circ\phm}(N^{[m]n}A)$. 
		
		Hence we have $ \Ker \mF= \Gamma_{\Psi\circ\phm}(N^{[m]n}A)$  and finally this gives us our expected isomorphism to conclude that $\phi^*(\partial\Psi)$ and $\partial(\phm^*\Psi) $ represent the same extension class in $\Ext(J^mA,\hG)$.
	\end{proof}
	\subsection{Morphism in the system of isocrystals}
	Combining Theorem \ref{phi*Ext}  with Theorem \ref{JnNm},
	we have the following commutative diagram of $R$-modules for all $m\geq 1$
	\begin{align}\label{maps_of_phm}
		\xymatrix{0\ar[r]& \bX_m(A)\ar[r]^-{u_m^*}&\bX_{m+n}(A)\ar[r]^-{\iota_m^*}& \Hom(N^{[m]n}A,\hG)\ar[r]^-{\partial} & \Ext(J^mA,\hG)\\
			0\ar[r]& \bX_{m-1}(A)_\phi \ar[r]^-{u_{m-1}^*}\ar[u]_{\phi^*}&\bX_{m-1+n}(A)_\phi \ar[r]^-{\iota_{m-1}^*}\ar[u]_{\phi^*}& \Hom(N^{[m-1]n}A,\hG)_\phi \ar[r]^-{\partial}\ar[u]_{\phm^*} & \Ext(J^{m-1}A_\phi,\hG).\ar[u]_{\phi^*}}
	\end{align}

	Recall from Proposition \ref{Psi_ij}, for all $m,n\geq 0$, $\Hom(N^{[m]n}A)\simeq R\langle \Psi_1^{[m]},\cdots, \Psi_n^{[m]}\rangle,$ where $\Psi_i^{[m]}=\{\Psi_{ij}^{[m]}\}_{j=1}^g$ and $\Psi_{ij}^{[m]}$ is defined as in (\ref{Psiijm}). We denote this basis by $\BB_{[m]n}.$
	\begin{corollary}\label{comm_phi*}
		For every $m\geq 1$ and $n\geq 0$, we have the following commutative diagram of $R$-modules
		\begin{equation}
			\xymatrix{0\ar[r]& \frac{\bX_{m+n}(A)}{u_m^*\bX_m(A)}\ar[r]^-{\iota_m^*}& \Hom(N^{[m]n}A,\hG)\ar[r]^-{\partial} & \bI_{[m]n}(A)\ar[r] & 0\\
				0\ar[r]& \frac{(\bX_{m-1+n}(A))_{\phi}}{(u_{m-1}^*\bX_{m-1}(A))_{\phi}}\ar[r]^-{\iota_{m-1}^*}\ar[u]_{\phi^*}^{u_{m-1+n}^*}& \Hom(N^{[m-1]n}A,\hG)_{\phi}\ar[r]^-{\partial}\ar[u]_{\phm^*}^{u_{m-1+n}^*} & \bI_{[m-1]n}(A_{\phi})\ar[r]\ar[u]_{\phi^*} ^{u_{m-1+n}^*}& 0. }
		\end{equation}
		
		Moreover, we have  the following:
		\begin{enumerate}
			\item[(i)] $\phm^*:\Hom(N^{[m-1]n}A,\hG)_{\phi})\rightarrow \Hom(N^{[m]n}A,\hG)$ is  an injection defined by multiplication of $\pi$ with respect to the basis $\BB_{[m-1]n}$ and $\BB_{[m]n}$. Tensoring with $K$ gives us an isomorphism $\phm^*:(\Hom(N^{[m-1]n}A,\hG)_{\phi})_K\rightarrow \Hom(N^{[m]n}A,\hG)_K$ of $K$-modules.
			
			\item[(ii)] $\phi^*:\frac{(\bX_{m-1+n}(A))_{\phi}}{(u_{m-1}^*\bX_{m-1}(A))_{\phi}}\rightarrow\frac{\bX_{m+n}(A)}{u_m^*\bX_m(A)}$ is an injection. 
			\item[(iii)] Tensoring with $K$, we get $\phi^*:\bI_{[m-1]n}(A_\phi)_K\rightarrow \bI_{[m]n}(A)_K$ is a surjection. 
			\item[(iv)] We have a surjective map $u_{m-1+n}^*:\bI_{[m-1]n}(A)\rightarrow \bI_{[m]n}(A)$ induced from the projection $u_{m-1+n}:J^{m+n}A\rightarrow J^{m+n-1}A$.
		\end{enumerate}
		
	\end{corollary}
	
	\begin{proof}
		The first square in the above diagram commutes by the definition of $\phm$, while the last square commutes by Theorem \ref{phi*Ext}.
		
		By Corollary \ref{phm-pi-com}, the map
		$$
		\phm^* : \Hom(N^{[m-1]n}A,\hG)_\phi \longrightarrow \Hom(N^{[m]n}A,\hG)
		$$
		is injective as a morphism of $R$-modules with $\mathrm{Coker}_3$ as the cokernel. Consider the following morphism of exact sequences:
		\begin{equation}\label{maps_of_phm_3}
			\xymatrix{
				& \mathrm{Coker}_1 & \mathrm{Coker}_3 & \mathrm{Coker}_2 & \\
				0 \ar[r] &
				\frac{\bX_{m+n}(A)}{u_m^*\bX_m(A)} \ar[r]^-{\iota_m^*} \ar[u] &
				\Hom(N^{[m]n}A,\hG) \ar[r]^-{\partial} \ar[u] &
				\bI_{[m]n}(A) \ar[r] \ar[u] &
				0 \\
				0 \ar[r] &
				\frac{(\bX_{m-1+n}(A))_\phi}{(u_{m-1}^*\bX_{m-1}(A))_\phi}
				\ar[r]^-{\iota_{m-1}^*} \ar[u]_{\phi^*} &
				(\Hom(N^{[m-1]n}A,\hG))_\phi \ar[r]^-{\partial} \ar[u]_{\phm^*} &
				\bI_{[m-1]n}(A_\phi) \ar[r] \ar[u]_{\phi^*} &
				0 \\
				& \mathrm{Ker}_1 \ar[u] & 0 \ar[u] & \mathrm{Ker}_2 \ar[u] &
			}
		\end{equation}
		where $\mathrm{Ker}_1$ and $\mathrm{Coker}_1$ denote the kernel and cokernel of
		$$
		\phi^* :
		\frac{(\bX_{m-1+n}(A))_\phi}{(u_{m-1}^*\bX_{m-1}(A))_\phi}
		\longrightarrow
		\frac{\bX_{m+n}(A)}{u_m^*\bX_m(A)},
		$$
		and $\mathrm{Ker}_2$ and $\mathrm{Coker}_2$ denote the kernel and cokernel of
		$$
		\phi^* : \bI_{[m-1]n}(A_\phi) \longrightarrow \bI_{[m]n}(A).
		$$
		By the snake Lemma, we obtain $\mathrm{Ker}_1 = 0$. Consequently, the map
		$$
		\phi^* :
		\frac{(\bX_{m-1+n}(A))_\phi}{(u_{m-1}^*\bX_{m-1}(A))_\phi}
		\longrightarrow
		\frac{\bX_{m+n}(A)}{u_m^*\bX_m(A)}
		$$
		is injective.
		
		After tensoring with $K$, Corollary \ref{phm-pi-com} implies that $\phm^*$ becomes an isomorphism, and hence $\mathrm{Coker}_3 = 0$. Therefore,
		$$
		\phi^* : \bI_{[m-1]n}(A_\phi)_K \longrightarrow \bI_{[m]n}(A)_K
		$$
		is surjective, since $\mathrm{Coker}_2 = 0$.
		Finally, note that the map
		$$
		u^* : \Hom(N^{[m-1]n}A,\hG)
		\longrightarrow
		\Hom(N^{[m]n}A,\hG)
		$$
		is an isomorphism of $R$-modules. This forces the induced map
		$$
		u^* : \bI_{[m-1]n}(A) \longrightarrow \bI_{[m]n}(A)
		$$
		to be surjective.
		
	\end{proof}

	\textbf{Remark.}
	Note that for any $\pi$-formal abelian scheme $A$ over $S$, the $m$-th upper splitting number $\beta_m$ exists because from Corollary \ref{comm_phi*}, $\phi^*:\bI_{[m-1]}(A)_K\rightarrow \bI_{[m]}(A)_K$ is a surjection for all $m\geq 1$ and hence $\rk_K(\bI_{[m]}(A)_K)\leq \rk_K(\bI_{[0]}(A)_K)\leq g.$ Hence there exist $\beta_m\geq 1$ such that $$\bI_{[m]0}(A)\subseteq \bI_{[m]1}(A)\subseteq \cdots  \subseteq \bI_{[m]\beta_0-1}(A)=\bI_{[m]\beta_0}(A)= \bI_{[m]}(A).$$
	
	\color{black}

	\begin{lemma}\label{phi*_isom}
		For all $m>0$ and $n\geq 0$, the following diagram of $R$-modules commutes
		\begin{equation}
			\xymatrix{
				\frac{\bX_{m-1+n}(A)_\phi}{u_{m-1}^*\bX_{m-1}(A)_\phi}\ar[r]_{\sim}^{\phi^*}\ar[d]^{\iota_0^*}_{\rotatebox{90}{$\sim$}} & \frac{\bX_{m+n}(A)}{u_m^*\bX_m(A)}\ar[d]^{\iota_0^*}_{\rotatebox{90}{$\sim$}}\\
				\frac{\Hom(N^{[0]m-1+n}A,\hG)_\phi}{u_m^*\Hom(N^{[0]m-1}A,\hG)_\phi}\ar[r]^{\mff_0^*}_{\sim} & \frac{\Hom(N^{[0]m+n}A,\hG)}{u_m^*\Hom(N^{[0]m}A,\hG)}.
			}
		\end{equation}
		In particular, if $m\geq \beta_0$, all the maps are isomorphisms of $R$-modules. 
	\end{lemma}
	\begin{proof}
		For any $m> 0$ and $n\geq 0$ the commutativity of the diagram above follows directly from Proposition \ref{Char_n}. Let $\Theta\in \bX_{m-1+n}(A)_\phi$, then Proposition \ref{Char_n} implies that $$\iota_0^*\phi^*\Theta=\mff_0^*\iota_0^*\Theta+\gamma.\Psi_1^{[0]}$$
		for some $\gamma\in R^g$. Note that $\Psi_1^{[0]}\in u_m^*\Hom(N^{[0]m}A,\hG)$ and hence we have the required commutativity. 
		
		For any $m>\beta_0$, let us consider the map of exact sequences
		\begin{equation}
			\xymatrix{
				0\ar[r] & \bX_{m+n}(A)\ar[r]^-{\iota_0^*}& \Hom(N^{[0]m+n}A,\hG)\ar[r]^-{\partial} & \bI_{[0]m+n}(A)\ar[r] &0\\
				0\ar[r] & \bX_{m}(A)\ar[r]^-{\iota_0^*}\ar[u]^-{u_m^*}& \Hom(N^{[0]m}A,\hG)\ar[u]^{u_m^*}\ar[r]^-{\partial} & \bI_{[0]m}(A)\ar[u]\ar[r] &0.
			}
		\end{equation}
		Taking the quotient of both the exact sequences, we have 
		\begin{equation}\label{Exact_Xmn_Hmn}
			\xymatrix{
				0\ar[r] & \frac{\bX_{m+n}(A)}{u_m^*\bX_m(A)}\ar[r]^-{\iota_0^*}& \frac{\Hom(N^{[0]m+n}A,\hG)}{u_m^*\Hom(N^{[0]m}A,\hG)}\ar[r]^-{\partial} & \frac{\bI_{[0]m+n}(A)}{\bI_{[0]m}(A)}\ar[r] &0.
			}
		\end{equation}
		Since $\bI_{[0]m+n}(A)$ is free, there is a section $s_{m+n} : \bI_{[0]m+n}(A)\rightarrow \Hom(N^{[0]m+n}A,\hG)$. Hence this section gives us the following exact sequence of $R$-modules
		\begin{equation}
			\xymatrix{0\ar[r] & \frac{\bX_{m+n}(A)}{u_m^*\bX_m(A)}\ar[r]^-{\iota_0^*}& \frac{\Hom(N^{[0]m+n}A,\hG)}{u_m^*\Hom(N^{[0]m}A,\hG)}\ar[r]^-{\partial} & \frac{\bI_{[0]m+n}(A)}{\bI_{[0]m}(A)}\ar[r]\ar@/^1.8pc/[l]_{s_{m+n}} &0}
		\end{equation}
		to split. By Proposition \ref{Psi_ij},  $\Hom(N^{[0]m+n}A,\hG)/u_m^*\Hom(N^{[0]m}A,\hG)$ is a free $R$-module of rank $ng$. 
		As we assume $m>\beta_0$, for all $n\geq 0$, we have $\rk_R(\bI_{[0]m+n}(A))=\rk_R(\bI_{[0]m}(A))$ because $h_{[0]m+1}=\rk_R(\bI_{[0]m+1}(A))-\rk_R(\bI_{[0]m}(A))=0$ for all $m> \beta_0 $. But  both $\bI_{[0]m+n}(A)$ and $\bI_{[0]m}(A)$ are free $R$-modules and therefore ${\bI_{[0]m+n}(A)}/{\bI_{[0]m}(A)}$ has only torsion points, this implies $s_{m+n}:{\bI_{[0]m+n}(A)}/{\bI_{[0]m}(A)}\rightarrow{\Hom(N^{[0]m+n}A,\hG)}/{u_m^*\Hom(N^{[0]m}A,\hG)}$ is a zero map. Hence from (\ref{Exact_Xmn_Hmn}),  we have  $\iota_0^*$ to be an isomorphism between 
		\begin{equation}\label{sim_i0}
			\frac{\bX_{m+n}(A)}{u_m^*\bX_m(A)}\simeq \frac{\Hom(N^{[0]m+n}A,\hG)}{u_m^*\Hom(N^{[0]m}A,\hG)}.
		\end{equation}
		Note that $\mff_0^*$ is an isomorphism by Proposition \ref{Psi_ij}. From (\ref{sim_i0}), we have $\iota_0^*$ to be an isomorphism of $R$-modules. Hence $\phi^*$ is an $R$-module isomorphism and we are done.
	\end{proof}
	\begin{theorem}\label{Stab_beta}
		For any $\pi$-formal abelian scheme $A$ over $S$, 
		
		\begin{enumerate}
			
			\item The Frobenius map $$\phi^*:\bI_{[m-1]n}(A_\phi)_K\rightarrow\bI_{[m]n}(A)_K$$ is a bijection for all $m\geq\beta_0$ and $n\geq 0$ . \label{(1)}
			\item The sequence $\{\beta_m\}$ is a decreasing sequence. 
			
			\item For all $m\geq \beta_0$, the sequence $\{\beta_m\}$ stabilizes, in particular, for all $m> \beta_0$, $\beta_{m-1}= \beta_{m}.$

		\end{enumerate}
	\end{theorem} 
	\begin{proof}
		
		$(1)$ From Lemma \ref{phi*_isom}, $\phi^*$ is an isomorphism for all $m\geq\beta_0$.  Hence from (\ref{maps_of_phm_3}) by the snake lemma we have $\phi^*:\bI_{[m-1]n}(A_\phi)_K\rightarrow\bI_{[m]n}(A)_K$ is an isomorphism for all $n\geq 0$ and $m\geq\beta_0$ and this proves (1).
		
		$(2)$ Let us fix some $m\geq 1,$ then by definition for all $n\geq \beta_{m-1}-1$, we must have $\bI_{[m-1]n}(A))\hookrightarrow \bI_{[m-1]n+1}(A)$ and for all $n\geq \beta_{m-1}-1$, $\bI_{[m-1]n}(A)=\bI_{[m-1]n+1}(A).$ 
		Hence by Corollary \ref{comm_phi*} (iv), we must have the following commutative diagram
		$$
		\xymatrix{
			\bI_{[m]\beta_{m-1}-1}(A)\ar@{^{(}->}[r]&\bI_{[m]\beta_{m-1}}(A)\ar@{^{(}->}[r]&\cdots\\
			\bI_{[m-1]\beta_{m-1}-1}(A)\ar@{=}[r]\ar[u]^{u^*}& \bI_{[m-1]\beta_{m-1}}(A)\ar[u]^{u^*}\ar@{=}[r] &\cdots
		}
		$$	
		where the 	vertical maps are surjective. Hence that forces the inclusion $\bI_{[m]n}(A)\hookrightarrow \bI_{[m]n+1(A)}$ to be an equality for all $n\geq \beta_{m-1}-1$ and that implies $\beta_{m-1}\geq \beta_m$ for all $m\geq 1.$	
		
		$(3)$ Since $\bI_{[m]n}(A)$ is a free $R$-module, for all $n\geq 0$ and $m\geq\beta_0$, we must have 
		\begin{equation}\label{equality_Imn} 
			\rk_R(\bI_{[m-1]n}(A))=\rk_K(\bI_{[m-1]n}(A)_K)=\rk_K(\bI_{[m]n}(A)_K)=\rk_R(\bI_{[m]n}(A)).
		\end{equation} 
		Let $n\geq \beta_{m-1}$, then by definition $\rk_R(\bI_{[m-1]n}(A))=\rk_R(\bI_{[m-1]n+1}(A))$. But by \eqref{equality_Imn}, we have $\rk_R(\bI_{[m-1]n}(A))=\rk_R(\bI_{[m]n}(A))=\rk_R(\bI_{[m]n+1}(A))=\rk_R(\bI_{[m-1]n+1}(A))$ for all $n\geq \beta_{m-1}$. Hence $\beta_m\leq \beta_{m-1}.$ Similarly for all $n\ge \beta_m$,  we have $\rk_R(\bI_{[m]n}(A))=\rk_R(\bI_{[m]n+1}(A))$ and that implies $\rk_R(\bI_{[m]n}(A))=\rk_R(\bI_{[m-1]n}(A))=\rk_R(\bI_{[m-1]n+1}(A))=\rk_R(\bI_{[m]n+1}(A))$ for all $n\geq \beta_m$. Hence $\beta_m\geq \beta_{m-1}$ and we have the stability of the sequence.	
	\end{proof}

	\textbf{Notational Convention.} In what follows, to alleviate the typographical burden, we shall systematically suppress the explicit indication of $\phi$-twists. Consequently, for an arbitrary $R$-module $M$, we will make no notational distinction between $M$ and its Frobenius pullback $M_\phi$. This slight abuse of notation is entirely benign and has no bearing on the validity of the ensuing results.

	\textbf{Remark.}  By  Corollary \ref{comm_phi*} and Theorem \ref{Stab_beta}, we have the following commutative diagram of $K$-modules	
	$$
	\xymatrix{\bI_{[\beta_0-1]0}(A)_K\ar@{^{(}->}[r]&\bI_{[\beta_0-1]1}(A)_K\ar@{^{(}->}[r]&\cdots\ar@{=}[r] &\bI_{[\beta_0-1]\beta_{0}-1}(A)_K\ar@{=}[r]&\bI_{[\beta_0-1]\beta_{0}}(A)_K=\bI_{[\beta_0-1]}(A)_K\\
		\bI_{[\beta_0-2]0}(A)_K\ar[u]_-{\phi^*}^{u^*}\ar@{^{(}->}[r]&\bI_{[\beta_0-2]1}(A)_K\ar[u]^{u^*}_-{\phi^*}\ar@{^{(}->}[r]&\cdots\ar@{=}[r] &\bI_{[\beta_0-2]\beta_{0}-1}(A)_K\ar[u]^{u^*}_-{\phi^*}\ar@{=}[r]&\bI_{[\beta_0-2]\beta_{0}}(A)_K=\bI_{[\beta_0-2]}(A)_K\ar[u]^{u^*}_-{\phi^*}\\
		\vdots\ar[u] & \vdots\ar[u] &  & \vdots\ar[u] & \vdots\ar[u] \\
		\bI_{[1]0}(A)_K\ar[u]^{u^*}_{\phi^*}\ar@{^{(}->}[r]&\bI_{[1]1}(A)_K\ar[u]^{u^*}_{\phi^*}\ar@{^{(}->}[r]&\cdots\ar@{=}[r] &\bI_{[1]\beta_0-1}(A)_K\ar[u]^{u^*}_{\phi^*}\ar@{=}[r]&\bI_{[1]\beta_0}(A)_K=\bI_{[1]}(A)_K\ar[u]^{u^*}_{\phi^*}\\
		\bI_{[0]0}(A)_K\ar[u]^{u^*}_{\phi^*}\ar@{^{(}->}[r]&\bI_{[0]1}(A)_K\ar[u]^{u^*}_{\phi^*}\ar@{^{(}->}[r]&\cdots\ar@{^{(}->}[r] &\bI_{[0]\beta_0-1}(A)_K\ar[u]^{u^*}_{\phi^*}\ar@{=}[r]&\bI_{[0]\beta_0}(A)_K=\bI_{[0]}(A)_K.\ar[u]^{u^*}_{\phi^*}
	}
	$$	
	We will show in Theorem \ref{rk_H_delta} that all modules in the  topmost row of the above 
	diagram are zero.
	
	\begin{corollary}\label{comm_phi*1}
		For every $m\geq 1$, we have the following  commutative diagram of exact sequences
		\begin{equation}
			\xymatrix{
				0\ar[r]& \DD_{[m]}(A)\ar[r]^-{\iota_m^*}& \mH_{[m]}^\delta(A)\ar[r]^-{\partial} & \bI_{[m]}(A)\ar[r] & 0\\
				0\ar[r] &\DD_{[m-1]}(A)\ar[r]^-{\iota_{m-1}^*}\ar[u]_{\phi^*}& \mH_{[m-1]}^{\delta}(A)\ar[r]^-{\partial}\ar[u]_{\phm^*} & \bI_{[m-1]}(A)\ar[u]_{\phi^*}\ar[r] & 0.
			}
		\end{equation}
	\end{corollary}
	
	\begin{proof}
		This follows from Corollary \ref{comm_phi*} and equation (\ref{Exact}) as we take direct limit over $n$.
	\end{proof}
	\begin{corollary}\label{Hodge_Seq 1}
		The map of exact sequences in Theorem  \ref{map_of_seq} can be realized as 
		\begin{equation}
			\xymatrix{
				0\ar[r] & \DD_{[m]}(A)_K\ar[r]^-{\iota_m^*}\ar[d]_{\Upsilon_{[m]}} & \mH_{[m]}^\delta(A)_K\ar[r]^-{\partial}\ar[d]_{\Lambda_{[m]}} & \Ext(J^mA,\hG)_k\ar@{=}[d] \\
				0\ar[r]  & \frac{\Lie(J^mA)^\vee_K}{\bX_m(A)_K}\ar[r] & \Ext^{\sharp}(J^mA,\hG)_K\ar[r] & \Ext(J^mA,\hG)_K.
			}  
		\end{equation}
	\end{corollary}
	\begin{proof}
		This follows directly from  Theorem \ref{map_of_seq} (3).
	\end{proof}
	
	\begin{proposition}\label{phi*X0}
		The Frobenius pull-back  $$\phi^*:\DD_{[m-1]}(A)\rightarrow\DD_{[m]}(A)$$ is an injection for all $m>0$ and is a bijection for all $m\geq \beta_0$.
	\end{proposition}
	\begin{proof}
		By definition for any $m\geq 0$, we have $$\bXp^{[m]}(A)=\varinjlim\frac{\bX_{m+n}(A)}{\phi^{*(m+1)}\bX_{n-1}(A)}.$$ Hence 
		\begin{equation}\label{X_0_prim}
			\DD_{[m]}(A)=\frac{\bXp^{[m]}(A)}{u_m^*\bX_{m}(A)}=\frac{\varinjlim_n\frac{\bX_{m+n}(A)}{\phi^{*(m+1)}\bX_{n-1}(A)}}{u_m^*\bX_{m}(A)}\simeq \varinjlim_n \frac{\frac{\bX_{m+n}(A)}{\phi^{*(m+1)}\bX_{n-1}(A)}}{u_m^*\bX_{m}(A)},
		\end{equation}
the last isomorphism follows from the exactness of filtered colimits of $R$-modules. Note that for all $n\geq 0$, 
		\begin{equation}\label{third-isom-thm}
			\frac{\frac{\bX_{m+n}(A)}{\phi^{*(m+1)}\bX_{n-1}(A)}}{u_m^*\bX_{m}(A)}\simeq \frac{\frac{\bX_{m+n}(A)}{u_m^*\bX_{m}(A)}}{\phi^{*(m+1)}\bX_{n-1}(A)} 
		\end{equation} because  $u_m^*\bX_{m}(A)\cap\phi^{*(m+1)}\bX_{n-1}(A)=\{0\}$ for all $n\geq 0$. From Corollary \ref{comm_phi*}, we obtain $$\phi^*:\frac{\bX_{m-1+n}(A)}{u_{m-1}^*\bX_{m-1}(A)}\rightarrow \frac{\bX_{m+n}(A)}{u_m^*\bX_{m}(A)}$$ to be an injection for all $m\geq 1$ and $n\geq 0$ and $\phi^*:\phi^{*(m)}\bX_{n-1}(A)\rightarrow\phi^{*(m+1)}\bX_{n-1}(A)$ is also bijective. Hence
		$$\phi^*:\frac{\frac{\bX_{m-1+n}(A)}{u_{m-1}^*\bX_{m-1}(A)_K}}{\phi^{*(m)}\bX_{n-1}(A)}\rightarrow \frac{\frac{\bX_{m+n}(A)}{u_m^*\bX_{m}(A)}}{\phi^{*(m+1)}\bX_{n-1}(A)}$$ is an injection and taking the limit over $n$, from equation (\ref{X_0_prim}), we know $$\phi^*:\DD_{[m-1]}(A)\rightarrow\DD_{[m]}(A)$$ is an injection for all $m>0.$ 
		
		Note that from Lemma \ref{phi*_isom},  if  $m\geq\beta_0$ and $n\geq 0$, then $$\phi^*:\frac{\bX_{m-1+n}(A)}{u_{m-1}^*\bX_{m-1}(A)}\rightarrow \frac{\bX_{m+n}(A)}{u_m^*\bX_{m}(A)}$$ is an isomorphism of $R$-modules and $\phi^*:\phi^{*(m)}\bX_{n-1}(A)\rightarrow\phi^{*(m+1)}\bX_{n-1}(A)$ is also a bijection. Hence  we get an $R$-module isomorphism $$\phi^*:\frac{\frac{\bX_{m-1+n}(A)}{u_{m-1}^*\bX_{m-1}(A)}}{\phi^{*(m)}\bX_{n-1}(A)}\rightarrow \frac{\frac{\bX_{m+n}(A)}{u_m^*\bX_{m}(A)}}{\phi^{*(m+1)}\bX_{n-1}(A)}.  $$ From (\ref{third-isom-thm}), we know that the isomorphism is $$\phi^*:\frac{\frac{\bX_{m-1+n}(A)}{\phi^{*(m)}\bX_{n-1}(A)}}{u_{m-1}^*\bX_{m-1}(A)}\rightarrow \frac{\frac{\bX_{m+n}(A)}{\phi^{*(m+1)}\bX_{n-1}(A)}}{u_m^*\bX_{m}(A)}.$$ Finally taking the limit over $n$, from equation (\ref{X_0_prim}), $$\phi^*:\DD_{[m-1]}(A)\rightarrow\DD_{[m]}(A)$$ is a bijection for all $m\geq\beta_0$.
	\end{proof}
	
	\begin{theorem}\label{phm_iso}
		For every $m\geq 1$, we have the  following commutative diagram of exact sequences of 
		$K$-modules
		\begin{equation}
			\xymatrix{0\ar[r]&\DD_{[m]}(A)_K\ar[r]^-{\iota_m^*}& \mH^{\delta}_{[m]}(A)_K\ar[r]^-{\partial} & \bI_{[m]}(A)_K\ar[r] & 0\\
				0\ar[r]&\DD_{[m-1]}(A)_K\ar[r]^-{\iota_{m-1}^*}\ar[u]_{\phi^*}& \mH^{\delta}_{[m-1]}(A)_K\ar[r]^-{\partial}\ar[u]_{\phm^*} & \bI_{[m-1]}(A)_K\ar[r]\ar[u]_{\phi^*} & 0. }
		\end{equation}
		Moreover we have,
		\begin{enumerate}
			\item[(i)] $\phm^*: \mH^{\delta}_{[m-1]}(A)_K\rightarrow  \mH^{\delta}_{[m]}(A)_K$ is a surjection.
			\item[(ii)]  $\phi^*:\bI_{[m-1]}(A)_K\rightarrow \bI_{[m]}(A)_K$ is also a surjection.
			
			\item[(iii)]  $\phi^*:\DD_{[m-1]}(A)_K\rightarrow\DD_{[m]}(A)_K$ is an injection. 
			\item[(iv)] For all $m\geq\beta_0,$  all the 
			vertical maps in  above diagram are  isomorphisms of $K$-modules.   
		\end{enumerate}
	\end{theorem}
	\begin{proof}
		The commutativity of the above diagram follows from  Corollary \ref{comm_phi*1}. For every $m\geq 1$, let us consider the following commutative diagram of exact sequences
		\begin{equation}\label{H_delta_m_phm}
			\xymatrix{
				0\ar[r]&\iota_m^*\phi^{*(m+1)}\bX_{n-1}(A)_K\ar[r]&\Hom(N^{[m]n}A,\hG)_K\ar[r]& \mH_{[m]}^\delta(A)_K\ar[r]& 0\\
				0\ar[r]&\iota_{m+1}^*\phi^{*(m)}\bX_{n-1}(A)_K\ar[r]\ar[u]_{\phm^*}& \Hom(N^{[m-1]n}A,\hG)_K\ar[r]\ar[u]_{\phm^*=[\pi]}&\mH_{[m-1]}^\delta(A)_K\ar[r]\ar[u]_{\phm^*}& 0.
			}
		\end{equation}
		From Corollary \ref{comm_phi*}, we have $$\phm^*:\Hom(N^{[m-1]n}A,\hG)_K\rightarrow\Hom(N^{[m]n}A,\hG)_K$$ is an isomorphism and hence by the snake lemma, we obtain $\phm^*:\mH_{[m-1}^\delta(A)_K\rightarrow\mH_{[m]}^\delta(A)_K$ to be a surjection.
		
		Note that $\phi^*:\DD_{[m-1]}(A)_K\rightarrow\DD_{[m]}(A)_K$ is an  injection from Proposition \ref{phi*X0} and $\phi^*:\bI_{[m-1]}(A)_K\rightarrow\bI_{[m]}(A)_K$ is a surjection from Corollary \ref{comm_phi*}.
		
		If $m\geq \beta_0,$ from Theorem \ref{Stab_beta}, we have $\phi^*:\bI_{[m-1]}(A)_K\rightarrow \bI_{[m]}(A)_K$ is an isomorphism and from Corollary \ref{phi*X0}  $\phi^*:\DD_{[m-1]}(A)\rightarrow \DD_{[m]}(A)$ is an isomorphism. Hence this implies $\phm^*:\mH_{[m-1]}^\delta(A)_K\rightarrow\mH_{[m]}^\delta(A)_K$ is an isomorphism.
	\end{proof}
	\begin{corollary}\label{rk_H_del}
		For every $m\geq 1,$ we have
 $\rk_R(\mH_{[m-1]}^\delta(A))\geq\rk_R(\mH_{[m]}^\delta(A))$.\\
		In particular, for $m\geq\beta_0$, we obtain  $\rk_R(\mH_{[m-1]}^\delta(A))=\rk_R(\mH_{[m]}^\delta(A))$.
	\end{corollary}
	\begin{proof}
		From Theorem \ref{phm_iso}, for every $m\geq 0$,  $\phm^*: \mH^{\delta}_{[m]}(A)_K\rightarrow  \mH^{\delta}_{[m+1]}(A)_K$ is a surjection, hence $\rk_K(\mH^{\delta}_{[m]}(A)_K)\geq \rk_K(\mH^{\delta}_{[m+1]}(A)_K)$. But as $\mH^{\delta}_{[m]}(A)$ is a free $R$-module, we have $\rk_K(\mH^{\delta}_{[m]}(A)_K)=\rk_R(\mH^{\delta}_{[m]}(A))$ and hence we obtain $$\rk_R(\mH^{\delta}_{[m]}(A))\geq \rk_R(\mH^{\delta}_{[m+1]}(A)).$$
		Again from Theorem \ref{phm_iso}, for every $m\geq\beta_0$,  $\phm^*:\mH_{[m]}^\delta(A)_K\rightarrow\mH_{[m+1]}^{\delta}(A)_K$ is an isomorphism and hence we have $\rk_K(\mH^{\delta}_{[m]}(A)_K)=\rk_R(\mH^{\delta}_{[m]}(A))=\rk_R(\mH^{\delta}_{[m+1]}(A))=\rk_K(\mH^{\delta}_{[m+1]}(A))_K.$
	\end{proof}
	
	\color{black}
	\color{magenta}
	
	\color{black}

	\section{Rank of the Filtered Isocrystals and Applications in $p$-adic Hodge Theory}\label{rk_H_m}
	
	Let us recall the following exact sequence of $R$-modules
	$$
	\xymatrix{0\ar[r] & \DD_{[0]}(A)\ar[r]& \mH_{[0]}^\delta(A)\ar[r]& \bI_{[0](A)}\ar[r] &0.}
	$$
	Set $\rk_R(\bI_{[0]}(A))=k$. Since $\bI_{[0]}(A)\subseteq \Ext(A,\hG)$. Hence, $k=\rk_R(\bI_{[0]}(A))\leq \rk_R(\Ext(A,\hG))=g$ as $A$ is of relative dimension $g$ over $S$ and $\beta_0$ denotes the $0$-th upper splitting number. Then we have the following increasing sequence of $R$-modules $$0=\bI_{[0]0}(A)\subseteq \bI_{[0]1}(A)\subseteq\cdots\subseteq \bI_{[0]\beta_0-1}(A)=\bI_{[0]\beta_0}(A)=\bI_{[0]}(A).$$ Let $k_i$ be the rank of $\bI_{[0]_i}(A)$. Then by Lemma 7.3 of \cite{BoSa1}, we have $0=k_1<k_2<\cdots <k_{\beta_0-2}<k_{\beta_0-1}=k.$ Let us recall the following system of exact sequences of $R$-modules
	
	\begin{equation}\label{sys_X_mn}
		\xymatrix{
			0\ar[r] & \bX_1(A)\ar[r]\ar[d]_{u_1^*} & \Hom(N^{[0]1}A,\hG)\ar[r]\ar[d]_{u_1^*} & \bI_{[0]1}(A)\ar[d]_{u_1^*}\ar[r] & 0\\
			0\ar[r] & \bX_2(A)\ar[r]\ar[d]_{u^*_2} & \Hom(N^{[0]2}A,\hG)\ar[r]\ar[d]_{u_2^*} & \bI_{[0]2}(A)\ar[d]_{u_2^*}\ar[r] & 0\\
			& \vdots\ar[d]& \vdots \ar[d] &\vdots \ar[d] &\\
			0\ar[r] & \bX_{\beta_0-2}(A)\ar[r]\ar[d]_{u_{\beta_0-2}^*} & \Hom(N^{[0]\beta_0-2}A,\hG)\ar[r]\ar[d] _{u_{\beta_0-2}^*}& \bI_{[0]\beta_0-2}(A)\ar[d]_{u_{\beta_0-2}^*}\ar[r] & 0\\
			0\ar[r] & \bX_{\beta_0-1}(A)\ar[r] & \Hom(N^{[0]\beta_0-1}A,\hG)\ar[r] & \bI_{[0]\beta_0-1}(A)\ar[r] & 0.\\
		}
	\end{equation}
	From Theorem \ref{X_inf} we have,  $$\DD_{[0]}(A)=\bXp^{[0]}(A)\simeq R\langle\Theta_1,\cdots,\Theta_g\rangle.$$ Let  $\BB=\{\Theta_1,\cdots,\Theta_g\}$ be a set of primitive characters and $o_i$ be the exact order of $\Theta_i$.   Let us recall from \cite{BoSa1}, Page 419, $l_i$ denote the number of  primitive characters in $\BB$ of exact order $i$ and $h_i$ denote the difference of $k_i$ and $k_{i-1}$  for all $i\geq 1$. Note that $l_i=\rk_K\left({\bX_i(A)_K}/{u^*\bX_{i-1}(A)_K+\phi^*\bX_{i-1}(A)_K}\right)$ and $k_i=\sum_{j=0}^ih_j$. Then for all $m\geq 0$, from \cite{BoSa1}  equation (7.5), we obtain
	\begin{equation}\label{rk_RX_m}
		\rk_R(\bX_m(A))=ml_1+(m-1)l_2+\cdots+2l_{m-1}+l_m.
	\end{equation}

	\begin{proposition} \label{l_ik_i}
		The following equalities are true.
		\begin{enumerate}
			\item[(i)] $l_1=g-k_1$.
			\item[(ii)] $l_2=2k_1-k_2$.
			\item[(iii)] $l_i=2k_{i-1}-k_{i-2}-k_i$ for all $i\geq 3$.
			In particular, $l_{\beta_0}=k-k_{\beta_0-2}$ and $l_i=0$ for all $i>\beta_0$.
		\end{enumerate}
	\end{proposition}
	\begin{proof}
		Note that from Proposition \ref{Psi_ij},  $\rk_R(\Hom(N^{[0]i}A,\hG)=ig$  for all $i\geq 0$.
		\begin{enumerate}
			
			\item[(i)] Since $\rk_R(\Hom(N^{[0]1}A,\hG)=g$, by  the exact sequence $$
			\xymatrix{
				0\ar[r]&\bX_1(A)\ar[r]&\Hom(N^{[0]1}A,\hG)\ar[r]&\bI_{[0]1}A\ar[r]&0
			}
			$$ we obtain, $$\rk_R(\bX_1(A))=l_1=g-k_1.$$ 
			\item[(ii)]From the following exact sequence
			$$
			\xymatrix{
				0\ar[r]&\bX_2(A)\ar[r]&\Hom(N^{[0]2}A,\hG)\ar[r]&\bI_{[0]2}A\ar[r]&0,
			}
			$$
			we have $\rk_R(\bX_2(A))=2l_1+l_2=2g-k_2.$ Hence
			it follows that
			\begin{align}
				l_2&=2g-k_2-2l_1\\
				&=2g-k_2-2(g-k_1)\\
				&=2k_1-k_2.
			\end{align}
			\item[(iii)] Recall from \cite{BoSa1}, Lemma 7.2, that for all $i\geq 2$, $l_i=h_{i-1}-h_i$. We know that $k_i=\sum_{j=0}^ih_i$, therefore  for all $i\geq 3$, we obtain
			\begin{align}
				2k_{i-1}-k_{i-2}-k_{i}&=2\sum_{j=0}^{i-1}h_j-\sum_{j=0}^{i-2}h_i-\sum_{j=0}^ih_i\\
				&=\left(\sum_{j=0}^{i-2}2h_j-h_i-h_j\right)+2h_{i-1}-h_{i-1}-h_i\\
				&=h_{i-1}-h_i=l_i \textrm{ (From Lemma 7.2 in \cite{BoSa1})}.
			\end{align}
			Hence for all $i\geq 3$, we conclude $$l_i=2k_{i-1}-k_{i-2}-k_i.$$
			For $i=\beta_0$, $$l_{\beta_0}=2k_{\beta_0-1}-k_{\beta_0-2}-k_{\beta_0}$$ and we have $k_i=k$ for all $i\geq \beta_{0}-1$, therefore we obtain $$l_{\beta_0}=k-k_{\beta_0-2}.$$ Finally for all $i\geq \beta_0+1$, we get $l_i=2k-k-k=0$ and we are done. 
		\end{enumerate}
	\end{proof}
	
	\begin{corollary}
		\label{rk_D_m_k}
		For every integer $m \geq 0$, we have
		$$
		\rk_R\bigl(\DD_{[m]}(A)\bigr) = g + k_m.
		$$
		In particular, for all $m \geq \beta_0 - 1$, we obtain
		$$
		g + k = o_1 + o_2 + \cdots + o_g.
		$$
	\end{corollary}
	\begin{proof}
		Note that from Corollary \ref{rk_D_m}, we have for all $m\geq 0$ 
		\begin{equation}\label{rk_D_beta}
			\DD_{[m]}(A)=\frac{\bXp^{[m]}(A)}{u_m^*\bX_m(A)}.
		\end{equation}
		Hence using Theorem \ref{X^m_prim} and equation (\ref{rk_RX_m}), for all $m\geq 1$, we obtain $$\rk_R(\DD_{[m]}(A))=(m+1)g-ml_1-(m-1)l_2-\cdots-2l_{m-1}-l_m.$$
		For $m=1$ and $m=2$, 
		\begin{align}		
			\rk_R(\DD_{[1]}(A))&=2g-l_1=2g-(g-k_1)=g+k_1\\
			\rk_R(\DD_{[2]}(A))&=3g-2(g-k_1)-(2k_1-k_2)=g+k_2.
		\end{align}
		For $m\geq 3$, we have 
		\begin{align}
			\rk_R(\DD_{[m]}(A))&=(m+1)g-ml_1-(m-1)l_2-\cdots-2l_{m-1}-l_m\\
			&=(m+1)g-m(g-k_1)-(m-1)(2k_1-k_2)\\
			&-\left(\sum_{i=2}^{m-1}(m-i)(2k_i-k_{i-1}-k_{i+1})\right) \text{ (By Proposition \ref{l_ik_i})}\\
			&=g-(m-2)k_1+(m-1)k_2-\left(\sum_{i=2}^{m-1}(m-i)(2k_i-k_{i-1}-k_{i+1})\right)\\
			&=g-(m-2)k_1+(m-1)k_2\\
			&-\left(\sum_{i=2}^{m-1}2(m-i)k_i-\sum_{i=1}^{m-2}(m-(i+1))k_i-\sum_{i=3}^{m}(m-(i-1))k_i\right)\\
			&=g-(m-2)k_1+(m-1)k_2-2(m-2)k_2+(m-2)k_1+(m-3)k_2\\
			&-\left(\sum_{i=3}^{m-2}(2(m-i)-(m-(i+1))-(m-(i-1)))k_i\right)-2k_{m-1}+2k_{m-1}+k_m\\
			&=g+k_m.
		\end{align}
		Therefore, putting $m=\beta_0-1$, we get $\rk_R(\DD_{[\beta_0-1]}(A))=g+k_{\beta_0-1}=g+k$ and from Proposition \ref{phi*X0}, we conclude $\rk_R(\DD_{[m-1]}(A))=\rk_R(\DD_{[m]}(A))$ for all $m\geq \beta_0.$ Hence from Corollary \ref{rk_D_m}, we deduce $\rk_R(\DD_{[m]}(A))=g+k=o_1+\cdots+o_g$ for all $m\geq\beta_0-1.$
	\end{proof}
	
	Recall that we have a set of  primitive characters $\BB=\{\Theta_1,\cdots, \Theta_g\}$, such that $\bXp^{[0]}(A)\simeq R\langle \Theta_1,\cdots\Theta_g\rangle$
	and $o_i$ denotes the exact order of $\Theta_i$ for all $1\leq i\leq g.$
	
	\begin{theorem}\label{rk_H_delta}
		Let $A$ be an abelian scheme of relative dimension $g$ over $S$. For all $m\geq 0$,  we have 
		\begin{enumerate}
			\item
			$\mH^\delta_{[m]}(A)_K\simeq \mH^\delta_{[m+1]}(A)_K$ for all $m\geq 0$.
			\item
			$\rk_R(\mH^\delta_{[m]}(A))=o_1+\cdots+o_g.$
			\item 
			$\rk_R(\bI_{[m]}(A))=k-k_m$, where $k=\rk_{R}(\bI_{[0]}(A))$ and $k_m=\rk_R(\bI_{[0]m}(A))$ for all $m\geq 1$. 
		\end{enumerate}
		
		In particular, for all $m\geq \beta_0-1$, $\bI_{[m]}(A)=\{0\}$ and we have
		$$
		\xymatrix{
			\DD_{m+1}(A)\ar[r]^-\sim & \mH_{[m+1]}^\delta(A)\\
			\DD_{m}(A)\ar[r]^-\sim \ar[u]_-{\rotatebox{90}{$\sim$}}^-{\phi^*} & \mH_{[m]}^\delta(A)\ar[u]^-{\rotatebox{90}{$\sim$}}_-{\Phi_{[m+1]}^*}.
		}
		$$
	\end{theorem}
	\begin{proof}
		Let us recall the following diagram of $R$-modules:
		\begin{equation}\label{sys_isocrys}
			\xymatrix{
				0\ar[r]& \DD_{[\beta_0]}(A)\ar[r]^-{\iota_{\beta_0}^*}& \mH_{[\beta_0]}^\delta(A)\ar[r] & \bI_{[\beta_0]}(A)\ar[r]& 0\\
				0\ar[r]& \DD_{[\beta_0-1]}(A)\ar[r]^-{\iota_{\beta_0-1}^*}\ar[u]_{\phi^*}& \mH_{[\beta_0-1]}^\delta(A)\ar[r]\ar[u] _{\Phi_{[\beta_0]}^*}& \bI_{[\beta_0-1]}(A)\ar[r]\ar[u]_{\phi^*}& 0\\
				& \vdots\ar[u] & \vdots\ar[u] & \vdots\ar[u] &\\
				0\ar[r]& \DD_{[1]}(A)\ar[r]^-{\iota_1^*}\ar[u]_{\phi^*}& \mH_{[1]}^\delta(A)\ar[r]\ar[u] _{\Phi_{[2]}^*}& \bI_{[1]}(A)\ar[r]\ar[u]_{\phi^*}& 0\\
				0\ar[r]& \DD_{[0]}(A)\ar[r]^-{\iota_0^*}\ar[u]_{\phi^*}& \mH_{[0]}^\delta(A)\ar[r]\ar[u] _{\Phi_{[1]}^*}& \bI_{[0]}(A)\ar[r]\ar[u]_{\phi^*}& 0.	}
		\end{equation}
		From Corollary \ref{rk_D_m} and Corollary \ref{rk_D_m_k}, we obtain $\rk_R(\DD_{[\beta_0-1]}(A))=\rk_R(\DD_{[\beta_0]}(A))=o_1+\cdots+o_g=g+k$ and hence $\rk_R(\mH_{[\beta_0-1]}^\delta(A))\geq g+k$. We know that $\rk_R(\DD_{[0]}(A))=\rk_R(\bXp^{[0]}(A))=g$ and hence $\rk_R(\mH_{[0]}^\delta(A))=\rk_K(\mH_{[0]}^\delta(A)_K)=g+k$. From Theorem \ref{phm_iso} (i), $\phm^*$ is a surjection for any $m\geq 1.$ This implies $\rk_K(\mH_{[\beta_0-1]}^\delta(A)_K)\leq g+k$. Hence,
		$$\rk_K(\mH_{[\beta_0-1]}^\delta(A)_K)=g+k=o_1+\cdots+o_g=\rk_K(\mH_{[0]}^\delta(A)_K).$$ 
		Therefore by Theorem \ref{phm_iso} (i) and (iv), we obtain for all $m\geq 0$,
		$$\rk_K(\mH_{[m]}^\delta(A)_K)=g+k=o_1+\cdots+o_g.$$ 
		Since $\mH_{[m]}^\delta(A)$ is a free $R$-module, for all $m\geq 0$, we have
		$$\rk_R(\mH_{[m]}^\delta(A))=g+k=o_1+\cdots+o_g,$$
		and consequently for all $m\geq 0$, we obtain
		$$\mH_{[m]}^\delta(A)_K\simeq \mH_{[m+1]}^\delta(A)_K.$$
		Note that for all $m\geq 0$, one has \begin{align}
			\rk_R(\bI_{[m]}(A))&=\rk_R(\mH_{[m]}^\delta(A))-\rk_R(\DD_{[m]}(A))\\
			&=(g+k)-(g+k_m)=k-k_m.
		\end{align}
		In particular, for $m=\beta_0-1,$ we have $\rk_R(\bI_{[\beta_0-1)]}(A))=k-k_{\beta_0-1}=k-k=0.$ Therefore, $\bI_{[m]}(A)=\{0\}$ for all $m\geq \beta_0-1$ as $\phi^*:\bI_{[m-1]}(A)_K\rightarrow \bI_{[m]}(A)_K$ is a surjection by Theorem \ref{phm_iso} (ii) and this gives us the commutative diagram
		$$
		\xymatrix{
			\DD_{[m+1]}(A)\ar[r]^-\sim & \mH_{[m+1]}^\delta(A)\\
			\DD_{[m]}(A)\ar[r]^-\sim \ar[u]^-{\rotatebox{90}{$\sim$}}_-{\phi^*} & \mH_{[m]}^\delta(A)\ar[u]^-{\rotatebox{90}{$\sim$}}_-{\Phi_{[m+1]}^*}
		}
		$$
		for all $m\geq \beta_0-1.$ 
	\end{proof}

	\begin{corollary}
		For all $m\geq 0$, we have $$\mH_{[m]}^\delta(A)_K\simeq K\langle \phi^{*(\beta_0-o_1)}\Theta_1,\cdots,\phi^{*\beta_0}\Theta_1,\cdots,\Theta_g,\cdots,\phi^{*\beta_0-1}\Theta_g \rangle.$$
	\end{corollary}
	\begin{proof}
		The proof follows directly from Theorem \ref{rk_H_delta} and Corollary \ref{rk_D_m}.
	\end{proof}
	
	
	\begin{proposition}\label{delta_prim}
		For any abelian scheme $A$ of relative dimension $g$ over $S$, we obtain 
		\begin{equation}
			\mH_{[\beta_0]}^\delta(A)\simeq \sum_{k=0}^{\beta_0}\phi^{* (k)}\bXp^{[0]}(A).
		\end{equation}
		Moreover, tensoring with $K$, for any $m\geq0$, we get $$\mH_{[m]}^\delta(A)_K\simeq  \sum_{k=0}^{\beta_0-1}\iota_k^*\phi^{* (k)}\bXp^{[0]}(A)_K.$$
	\end{proposition}
	\begin{proof}
		
		Note that from Corollary \ref{phiXprim}, we obtain $$\bXp^{[\beta_0]}(A)= \sum_{k=0}^{\beta_0}\phi^{* (k)}\bXp^{[0]}(A)$$ and hence $$\DD_{[\beta_0]}(A)=\frac{\bXp^{[\beta_0]}(A)}{u_{\beta_0}^*X_{\beta_0(A)}}=\sum_{k=0}^{\beta_0}\frac{\phi^{* (k)}\bXp^{[0]}(A)}{u_{\beta_0}^*\bX_{\beta_0(A)}}.$$
		Theorem \ref{rk_H_delta} gives us that $\iota_{\beta_0}^*$ induces an isomorphism $$\DD_{[\beta_0]}(A)\simeq \mH_{[\beta_0]}^\delta(A)$$ and that implies $$\mH_{[\beta_0]}^\delta(A)= \sum_{k=0}^{\beta_0}\frac{\iota_{\beta_0}^*\phi^{*(k)}\bXp^{[0]}(A)}{\iota_{\beta_0}^*u_{\beta_0}^*X_{\beta_0(A)}}=\sum_{k=0}^{\beta_0}\iota_{\beta_0}^*\phi^{* (k)}\bXp^{[0]}(A).$$ Since $\iota_{\beta_0}^*$ is an isomorphism, this yields $$\mH_{[\beta_0]}^\delta(A)\simeq \sum_{k=0}^{\beta_0}\phi^{*(k) }\bXp^{[0]}(A).$$
		Tensoring both sides with $K$ and using Theorem \ref{rk_H_delta}, we conclude for any $m\geq 0$,
		$$\mH_{[m]}^\delta(A)_K\simeq \sum_{k=0}^{\beta_0}\phi^{* (k)}\bXp^{[0]}(A)_K.$$	
		Note that $\iota_k^*:\phi^{*(k)}\bXp^{[0]}(A)\subset \DD_{[k]}(A)_K\hookrightarrow \mH_{[k]}^\delta (A)_K$ is an injection and from Theorem \ref{rk_H_delta} we know that $\mH_{[m]}^\delta(A)_K\simeq \mH_{[m+1]}^\delta(A)_K$ for all $m\geq 0.$ Hence for all $m\geq 0$, it follows that 
		\begin{equation}\label{mH_phi*}
			\mH_{[m]}^\delta(A)_K\simeq \sum_{k=0}^{\beta_0}\iota_k^*\phi^{* (k)}\bXp^{[0]}(A)_K.
		\end{equation}
		Recall that from Theorem \ref{phm_iso}, we also know that $\phi^*:\DD_{[\beta_0-1]}(A)_K\rightarrow \DD_{[\beta_0]}(A)_K$ is an isomorphism and hence that implies $\phi^{*(\beta_0-1)}\bXp^{[0]}(A)_K\simeq \phi^{*(\beta_0)}\bXp^{[0]}(A)_K$, therefore (\ref{mH_phi*}) can be rewritten as 
		\begin{equation}\label{mH_phi*-1}
			\mH_{[m]}^\delta(A)_K\simeq \sum_{k=0}^{\beta_0-1}\iota_k^*\phi^{* (k)}\bXp^{[0]}(A)_K.
		\end{equation}
	\end{proof}
	
	Recall again that we have fixed a set $\BB=\{\Theta_1,\cdots,\Theta_g\}$ such that $\bXp^{[0]}(A)_K\simeq K\langle \Theta_1,\cdots,\Theta_g \rangle$ and $o_i$ denotes the exact order of $\Theta_i$. Note that for all $i\in\{1,\cdots,g\}$, if we consider $[\Theta_i]$ to be the class of $\Theta_i$ in $\bXp^{[0]}(A)$, then we can write
	$$\bXp^{[0]}(A)_K=K\langle [\Theta_1],\cdots,[\Theta_g] \rangle.$$ Hence from Corollary \ref{rk_D_m}, we have  $$\DD_{[\beta_0-1]}(A)_K=K\langle[\phi^{*(\beta_0-o_1)}\Theta_1],\cdots, [\phi^{*(\beta_0-1)}\Theta_1],\cdots,[\Theta_g],\cdots.[\phi^{*(\beta_0-1)}\Theta_g]\rangle,$$ where for all $1\leq j\leq g$ and $\beta_0-o_j\leq k_j\leq \beta_0-1$, we denote $[\phi^{*(k_j)}\Theta_j]$ the class of $\phi^{*(k_j)}\Theta_j$ in $\DD_{[\beta_0-1]}(A)_K$  and hence this implies  $$\mH^\delta_{[\beta_0-1]}(A)_K=K\langle\iota_{\beta_0-1}^*[\phi^{*(\beta_0-o_1)}\Theta_1],\cdots, \iota_{\beta_0-1}^*[\phi^{*(\beta_0-1)}\Theta_1],\cdots,\iota_{\beta_0-1}^*[\Theta_g],\cdots\iota_{\beta_0-1}^*[\phi^{*(\beta_0-1)}\Theta_g]\rangle.$$ 
	Let us denote for all $k\geq 1$, $$\Phi_{(k)}^*:=\Phi_{[k]}^*\circ\cdots\circ \Phi_{[1]}^*:\mH_{[0]}^\delta(A)_K\xrightarrow{\sim}\mH_{[k]}^\delta(A)_K.$$
	Recall form Proposition \ref{Psi_ij}, for all $1\leq i\le n$ and $1\leq j\leq g$, $\Hom(N^{[m]n}A,\hG)$ has $\{\Psi_{ij}^{[m]}\}$ as an $R$-basis, where $\Psi_{ij}^{[m]}$ is defined as in  (\ref{Psiijm}) and we denote $\Psi_i^{[m]}=(\Psi_{i1}^{[m]},\cdots,\Psi_{ig}^{[m]}).$

\begin{lemma}\label{Psi_Gamma}
Assume $R= W(\kappa)$.
		For all $k\geq 1,$ in $\mH_{[0]}^\delta(A)_K$, we have $$
		[\Psi_k^{[0]}]=\left(-{(\Gamma^{[0]})}^{\phi^{(k-1)}}\right)^{-1}\begin{pmatrix}
			\mff_0^{*(k)}\iota_0^*[\Theta_1]\\
			\vdots\\
			\mff_0^{*(k)}\iota_0^*[\Theta_g]
		\end{pmatrix}
		$$
		where the matrix $\Gamma^{[0]}$ is defined as in (\ref{gamma_mat}) and ${(\Gamma^{[0]})}^{\phi^{(k-1)}}$ denotes the matrix obtained by $\phi$-twisting each entry of $\Gamma^{[0]}$, $k-1$ times. In particular, for all $k\geq 1$, we have $$[\Psi_k^{[0]}]\in\mff_0^{*(k)}\iota_0^*\bXp^{[0]}(A)_K\subseteq \mH_{[0]}^\delta(A)_K.$$
	\end{lemma}
	\begin{proof}
		
		Note that from Proposition \ref{Char_n}, we obtain for all $j\in\{1,\cdots,g\}$ 
		$$
		\mff_0^*\iota_0^*\Theta_j=\iota_0^*\phi^*\Theta_j-\gamma_j^{[0]}.\Psi_1^{[0]}\in\Hom(N^{[0]o_j+1}A,\hG)
		$$
		where $\gamma_j^{[0]}$ is as defined in (\ref{i_mphi_m}). Therefore in $\mH_{[0]}^\delta(A)_K$ we obtain
		$$[\mff_0^*\iota_0^*\Theta_j]=\mff_0^*\iota_0^*[\Theta_j]=-\gamma_j^{[0]}.[\Psi_1^{[0]}].$$ Hence for all $k\geq 1$, we have $$
		\mff_0^{*(k)}\iota_0^*[\Theta_j]=\mff_0^{*(k-1)}(\mff_0^*\iota_0^*[\Theta_j])=-(\gamma_j^{[0]})^{\phi^{k-1}}[\Psi_{k}^{[0]}].
		$$ This implies 
		$$
		\begin{pmatrix}
			\mff_0^{*(k)}\iota_0^*[\Theta_1]\\
			\vdots\\
			\mff_0^{*(k)}\iota_0^*[\Theta_g]
		\end{pmatrix}=-\begin{pmatrix}
			\left(\gamma_1^{[0]}\right)^{\phi^{k-1}}\\
			\vdots \\
			\left(\gamma_g^{[0]}\right)^{\phi^{k-1}}
		\end{pmatrix}[\Psi_k^{[0]}]=-\begin{pmatrix}
			\left(\gamma_1^{[0]}\right)\\
			\vdots \\
			\left(\gamma_g^{[0]}\right)
		\end{pmatrix}^{\phi^{k-1}}[\Psi_k^{[0]}]=-{(\Gamma^{[0]})}^{\phi^{(k-1)}}[\Psi_k^{[0]}].
		$$ We have $\Gamma^{[0]}$ to be invertible over $K$ from Theorem 5.2 in \cite{Delta-crys}. Hence for all $k\geq 1$,    ${(\Gamma^{[0]})}^{\phi^{(k-1)}}$ is also invertible and consequently we have $$
		[\Psi_k^{[0]}]=\left(-{(\Gamma^{[0]})}^{\phi^{(k-1)}}\right)^{-1}\begin{pmatrix}
			\mff_0^{*(k)}\iota_0^*[\Theta_1]\\
			\vdots\\
			\mff_0^{*(k)}\iota_0^*[\Theta_g]
		\end{pmatrix}
		$$ and we are done.
	\end{proof}	
	
\begin{theorem}
\label{delta_prim_final} 
Assume $R = W(\kappa)$.
For any abelian scheme $A$ of relative dimension $g$ over $S$, we have
		$$\mH_{[0]}^\delta(A)_K=
		\sum_{k\geq 0}\mff_0^{*(k)}\iota_0^*\bXp^{[0]}(A)_K.$$
	\end{theorem}	
	\begin{proof}
		Clearly we have $$\sum_{k\geq 0}\mff_0^{*(k)}\iota_0^*\bXp^{[0]}(A)_K\subseteq \mH_{[0]}^\delta(A)_K.$$ We need to show that
		$$
		\mH_{[0]}^\delta(A)_K\subseteq \sum_{k\geq 0}\mff_0^{*(k)}\iota_0^*\bXp^{[0]}(A)_K.
		$$
		By definition, we have  $$\mH_{[0]}^\delta(A)_K=\varinjlim_n\mH_{[0]n}^\delta(A)_K=\varinjlim_n\frac{\Hom(N^{[0]n}A,\hG)_K}{\iota_0^*\phi^*\bX_{n-1}(A)_K}=\varinjlim_{n}\frac{K\langle \Psi_1^{[0]},\cdots,\Psi_n^{[0]}\rangle}{\iota_0^*\phi^*\bX_{n-1}(A)_K}.$$
		Therefore $\mH_{[0]}^\delta(A)_K$ is generated by $\{[\Psi_k^{[0]}]\}_{k\geq 0}$ as an $R$-module, where $[\Psi_k^{[0]}]$ denotes the image of $\Psi_k^{[0]}\in \Hom(N^{[0]n}A,\hG)$ in $\mH_{[0]n}^\delta(A)_K$ for all $k\geq 1$. 
		But from Lemma \ref{Psi_Gamma}, we have $[\Psi_k^{[0]}]\in \mff_0^{*(k)}\iota_0^*\bXp^{[0]}(A)_K.$ Hence, this implies $$
		\mH_{[0]}^\delta(A)_K\subseteq \sum_{k\geq 0}\mff_0^{*(k)}\iota_0^*\bXp^{[0]}(A)_K
		$$ and we are done.
	\end{proof}
	\begin{corollary}
Assume $R = W(\kappa)$.
		For every $m\geq 0$, we have $$\mH_{[m]}^\delta(A)_K=\sum_{k\geq 0}\mff_m^{*(k)}\iota_m^*\phi^{*(m)}\bXp^{[0]}(A)_K.$$
	\end{corollary}
	\begin{proof}
		The proof follows directly from Theorem \ref{delta_prim_final} and Theorem 5.8 and Theorem 5.5 in \cite{Mi-Sa-1}.
		
		From  Theorem \ref{delta_prim_final}, we have 
		$$
		\mH_{[0]}^\delta(A)_K=\sum_{k\geq 0}\mff_0^{*(k)}\iota_0^*\bXp^{[0]}(A)_K.
		$$
		This implies for all $m\geq 0$,
		\begin{align}
			\mH_{[m]}^\delta(A)_K &=\Phi_{(m)}^*\left(\mH_{[0]}^\delta(A)_K\right)=\Phi_{(m)}^*\left(\sum_{k\geq 0}\mff_0^{*(k)}\iota_0^*\bXp^{[0]}(A)_K\right)\\
			&=\Phi_{[m]}^*\circ\cdots\circ\Phi_{[1]}^*\left(\sum_{k\geq 0}\mff_0^{*(k)}\iota_0^*\bXp^{[0]}(A)_K\right)\\
			&=\sum_{k\geq 0}\mff_m^{*(k)}\circ \Phi_{[m]}^*\circ\cdots\circ\Phi_{[2]}^*\circ\Phi_{[1]}^*\circ\iota_0^*\bXp^{[0]}(A)_K \text{  (Using Theorem 5.8 in \cite{Mi-Sa-1})}\\
			&=\sum_{k\geq 0}\mff_m^{*(k)}\iota_m^*\phi^{*(m)}\bXp^{[0]}(A)_K \text{  (Using Theorem 5.5 in \cite{Mi-Sa-1})}
		\end{align}
	\end{proof}

Assume $R=W(\kappa)$.	
In Section 8 of \cite{BoSa1}, Borger and Saha constructed a canonical morphism of $K$-modules $\Phi: \bHd{0}(A)_K \map \Hcr^1(A)_K$, which fits into a natural commutative diagram of short exact sequences:
	\begin{equation}
		\label{mapisocrys}
		\xymatrix{
			0 \ar[r] & \bXp(A)_K \ar[d]^\Upsilon \ar[r] & \bH^\delta_{[0]}(A)_K \ar[d]^{\Phi} \ar[r] & \bI_{[0]}(A)_K \ar[r] \ainj{d} & 0  \\
			0 \ar[r] & H^0(A,\Omega_A)_K \ar[r] & \Hcr^1(A)_K \ar[r] & H^1(A,\Omega_A)_K \ar[r] & 0,
		}
	\end{equation}
	where the lower row constitutes the classical Hodge filtration sequence associated to $A$. 
	
	Let $\Fcr(\Omega_A)$ denote the minimal subisocrystal of $\Hcr^1(A)_K$ generated by the space of invariant differentials $H^0(A,\Omega_A)_K$. By virtue of Theorem 1.1 in \cite{Delta-crys}, the restriction $\Upsilon$ is an isomorphism of $K$-modules, which consequently forces the  comparison map $\Phi$ to be 
	injective.
	
	Pandit establishes in \cite{Pandit}, Theorem 1.2, that $\Phi$ in 
	(\ref{mapisocrys}) is in fact a  morphism of filtered isocrystals. 
	In this context, the target module $\Hcr^1(A)_K$ is naturally endowed with the crystalline Frobenius operator $\Fcr$ alongside the descending Hodge filtration, whose first step is given by $H^0(A,\Omega_A)_K$. As a direct consequence of this compatibility, $\Phi$ induces an exact isomorphism of isocrystals:

	\begin{corollary}
		\label{rkFcris}
		Let $A$ be an abelian scheme of relative dimension $g$ over $S$. Then $\Phi$
		induces the following isomorphism of isocrystals:
		$$
		\bHd{m}(A)_K \simeq  \bHd{0}(A)_K \stk{\Phi}{\longisomap}
		\Fcr(\Omega_A).
		$$
		In particular we have
		\begin{align}
			\dim_K(\Fcr(\Omega_A)) = o_1+\cdots+o_g.  
		\end{align} 
	\end{corollary}
	\begin{proof}
		The first isomorphism follows from Theorem \ref{rk_H_delta}. For the second 
		isomorphism via $\Phi$ note that 
		$\Fcr(\Omega_A) = \sum_k \Fcr^{\circ k} (H^0(A,\Omega_A)_K)$.
		Since by Theorem $1.2$ in \cite{Pandit}, $\Phi$ is a map of isocrystals,
		the isomorphism follows from Theorem \ref{delta_prim_final}.
		
		The second result is true by Theorem \ref{rk_H_delta} applied to the case $m=0$.
	\end{proof}

	\color{black}

	
	\section{The Low-Dimensional Cases }\label{Examples}
Assume $R=W(\kappa)$.
	In this section, we illustrate our results through some examples of abelian schemes of dimensions \(1\) and \(2\), namely the cases \(g=1\) and \(g=2\).

	\subsection{Elliptic Curves ($g=1$)}
	We take $A$ to be an elliptic curve over $R.$ Let us consider the following exact sequence
	\begin{equation}
		0\rightarrow \bX_n(A)\xrightarrow{\iota_0^*}\Hom(N^{[0]n}A,\hG)\xrightarrow{\partial}\Ext(A,\hG).
	\end{equation}
	By Corollary 7.8 of \cite{BoSa1}, the $0$-th upper splitting number satisfies $\beta_0\leq 2$. Thus there are two possibilities: $\beta_0=1$ or $\beta_0=2$. Recalling from Section 9 of \cite{BoSa1} 
	we have $$\bX_\infty(A)=R\{\phi^*\}\langle\Theta_{\beta_0}\rangle,$$ we obtain the following description in each case:
	\begin{enumerate}
		\item[(i)] If $A$ admits a canonical lift (CL), then $\beta_0=1$ and
		$$
		\iota_0^*\Theta_1=\Psi_1^{[0]}.
		$$
		
		\item[(ii)] If $A$ does not admit a canonical lift (Non-CL), then $\beta_0=2$ and
		$$
		\iota_0^*\Theta_2=\Psi_2^{[0]}-\lambda\Psi_1^{[0]},
		$$
		for some $\lambda\in R$. In this case, the map $\partial$ is non-zero, and consequently $\partial\Psi_1^{[0]}\neq 0$. Since $\iota_0^*\Theta_2$ lies in the kernel of $\partial$, it follows that
		$
		\lambda=\frac{\partial\Psi_2^{[0]}}{\partial\Psi_1^{[0]}}.
		$
	\end{enumerate}
	Finally, Proposition 9.1 of \cite{BoSa1} yields
	\begin{equation}\label{H_E}
		\mH_{[0]}^\delta(A)_K \simeq \begin{cases}
			K\langle\Psi_1^{[0]}\rangle, & \text{if $E$ has CL}\\
			K\langle\Psi_1^{[0]},\Psi_2^{[0]}\rangle, & \text{if $E$ is Non-CL.}
		\end{cases}
	\end{equation}

	\begin{theorem}\label{H_delta_E}
		For every elliptic curve $A$ over $S$ and every $m\geq 0$,
		$
		\mH_{[m]}^\delta(A)_K\simeq \mH_{[m+1]}^\delta(A)_K.
		$
		If $A$ has a canonical lift, then
		$
		\mH_{[m]}^\delta(A)_K\simeq K\langle\Psi_1^{[0]}\rangle
		$
		for all $m\geq 0$, and
		$
		\Iso(\mH_{[m]}^\delta(A))
		\simeq
		\Iso(\mH_{[0]}^\delta(A))
		$
		for all $m\geq 1$.
		
		If $A$ is Non-CL, then
		$
		\mH_{[m]}^\delta(A)_K
		\simeq
		K\langle\Psi_1^{[0]},\Psi_2^{[0]}\rangle
		$
		for all $m\geq 0$, while
		\begin{align}
			\Iso(\mH_{[0]}^\delta(A))
			&\not\simeq
			\Iso(\mH_{[1]}^\delta(A)),\\
			\Iso(\mH_{[m]}^\delta(A))
			&\simeq
			\Iso(\mH_{[1]}^\delta(A))
		\end{align}
		for all $m\geq 2$.

	\end{theorem}

	\begin{proof}
		This follows directly from Theorem \ref{rk_H_delta}.
	\end{proof} 
	
	\begin{corollary}
		If the elliptic curve $A$ does not admit a canonical lift (Non-CL), then
		$$
		\Ext(A,\hG)_K\simeq K\langle\Theta_2\rangle.
		$$
	\end{corollary}
	\begin{proof}
		If $A$ is Non-CL, $\bI_{[0]}(A)=\Ext(A,\hG).$  We have the following commutative diagram: 
		\begin{equation}\label{maps_of_phm_E_3}
			\xymatrix{&\text{Coker}_1&0&0&\\
				0\ar[r]&\DD_{[1]}(A)_K\ar[u]\ar[r]^-{\iota_1^*}& \mH^{\delta}_{[1]}(A)_K\ar[u]\ar[r]^-{\partial} & \bI_{[1]}(A)_K\ar[r]\ar[u] & 0\\
				0\ar[r]&\DD_{[0]}(A)_K\ar[r]^-{\iota_{0}^*}\ar[u]_{\phi^*}& \mH^{\delta}_{[0]}(A)_K\ar[r]^-{\partial}\ar[u]_{\phm^*} & \bI_{[0]}(A)_K\ar[r]\ar[u]_{\phi^*} & 0\\
				& 0\ar[u]& 0\ar[u]& \text{Ker}_3.\ar[u]& }
		\end{equation}
		By the snake lemma, we obtain $\text{Ker}_3\simeq \text{Coker}_1$. From Theorem \ref{H_delta_E} we have, $\bI_{[1]}(A)=\{0\}.$ Hence $\text{Ker}_3=\bI_{[0]}(A)_K=\Ext(A,\hG)_K.$ From corollary \ref{rk_D_m}, we conclude $\DD_{[0]}(A))\simeq R\langle \Theta_2\rangle$ and $\DD_{[1]}(A)\simeq R\langle \Theta_2,\phi^*\Theta_2\rangle.$ Hence the cokernel of the map $\phi^*:\DD_{[0]}(A)_K\rightarrow\DD_{[1]}(A)_K$ is isomorphic to $K\langle \Theta_2\rangle.$ Hence we have $\Ext(A,\hG)_K\simeq K\langle \Theta_2\rangle$.
	\end{proof}
	\subsection{Abelian Surfaces ($g=2$) }
	We consider $A$ to be an abelian scheme of relative dimension 2 over $S$. This implies $\rk_R\Ext(A,\hG)=2$ and we have three possible values for the rank of $\bI_{[0]}(A)$ namely $0,1$ and $2$. Note that $$\DD_{[0]}(A)=\bXp^{[0]}(A)\simeq R\langle \Theta_1,\Theta_2\rangle,$$ where $\Theta_1$ and $\Theta_2$ represent classes of primitive characters of $A$ in $\bXp^{[0]}(A)$. We denote $o_i$ to be the exact order of $\Theta_i $ for $i=1,2$ and assume $o_1\leq o_2.$ Let us recall the  following commutative diagram from Theorem \ref{phm_iso} for every $m\geq 1$.
	\begin{equation}\label{maps_of_phm_A_m}
		\xymatrix{0\ar[r]&\DD_{[m]}(A)_K\ar[r]^-{\iota_m^*}& \mH^{\delta}_{[m]}(A)_K\ar[r]^-{\partial} & \bI_{[m]}(A)_K\ar[r] & 0\\
			0\ar[r]&\DD_{[m-1]}(A)_K\ar[r]^-{\iota_{m-1}^*}\ar[u]_{\phi^*}& \mH^{\delta}_{[m-1]}(A)_K\ar[r]^-{\partial}\ar[u]_{\phm^*} & \bI_{[m-1]}(A)_K\ar[r]\ar[u]_{\phi^*} & 0. }
	\end{equation}
	From Theorem \ref{phm_iso}, we have the map $\phi^*:\bI_{[m-1]}(A)_K\rightarrow \bI_{[m]}(A)_K$ as a surjection for all $m\geq 1$, which implies $\{\rk_R(\bI_{[m]}(A))\}_{m\geq 0}$ is a decreasing sequence and from Theorem \ref{rk_H_delta}, we know it vanishes for all $m\geq \beta_0-1$. Let us explore each possible rank of $\bI_{[0]}(A)$ one by one.

	\underline{Case 1 ($\rk_R(\bI_{[0]}(A))=0)$}:  Note that in this case, by definition we have $\beta_0=1$, hence  $o_1=o_2=1$. Therefore by Theorem \ref{rk_H_delta} all the vertical maps in (\ref{maps_of_phm_A_m}) are isomorphisms and this gives us $$\Iso(\mH_{[m]}^\delta(A))\simeq \Iso(\mH_{[0]}^\delta(A))$$ for all $m\geq 0$. An example of this situation is given by
	$
	A=E_1\times E_2,
	$
	where $E_1$ and $E_2$ are elliptic curves with canonical lifts.
	
	\underline{Case 2 ($\rk_R(\bI_{[0]}(A))=1)$}: We have $\rk_K(\DD_{[0]}(A)_K)=2$, therefore, we obtain $\rk_K(\mH_{[0]}^\delta(A)_K)=3$ and by Theorem \ref{rk_H_delta}, we conclude $\rk_K(\mH_{[m]}(A)_K)=3$ for all $m\geq 0$.
	
	Note that $\bI_{[0]}(A)=\varinjlim_n\bI_{0{[n]}}(A)$ and hence we have $\rk_R(\bI_{[0]1}(A))\leq 1$. Therefore if $\rk_R(\bI_{[0]1}(A))=0 $, by Lemma 7.3 of \cite{BoSa1} we have $\rk_R(\bI_{[0]}(A))=0$ which takes us back to our first case. This implies we must have $\rk_R(\bI_{[0]1}(A))=\rk_R(\bI_{[0]2}(A))=\rk_R(\bI_{[0]}(A))=1$ and that gives us  $\beta_0=2$. 
	
	Therefore by Proposition 7.2 of \cite{BoSa1}, we have $o_2=\beta_{0}=2$. Hence by Theorem \ref{rk_H_delta} we obtain $o_1=1$  and  $\DD_{[1]}(A)_K\simeq K\langle\phi^*\Theta_1,\Theta_2,\phi^*\Theta_2\rangle,$ that is, $\rk_K(\DD_{[1]}(A)_K)=3$.
	Moreover by Theorem \ref{rk_H_delta}, we have 
	$$\DD_{[m]}(A)_K\simeq \mH_{[m]}^\delta(A)_K\simeq \mH_{[0]}(A)_K.$$ Therefore we conclude $$\Iso(\mH_{[0]}^\delta(A))\not\simeq \Iso(\mH_{[1]}^\delta(A))$$ as both the isocrystals have distinct filtrations and for all $m\geq 1$, we have $$\Iso(\mH_{[m]}^\delta(A))\simeq \Iso(\mH_{[1]}^\delta(A)).$$  
	
	An example illustrating this case is given by the product surface of elliptic curves, $A=E_1\times E_2$, where the elliptic curve $E_1$ has a canonical lift whereas $E_2$ does not.

	\underline{Case 3 ($\rk_R(\bI_{[0]}(A))=2)$}:
	We have $\rk_K(\DD_{[0]}(A)_K)=2$, which implies $\rk_K(\mH_{[0]}^\delta(A)_K)=4.$ Hence from Theorem \ref{rk_H_delta}, we obtain  $\rk_{K}(\mH_{[m]}(A)_K)=4$ for all $m\geq 0$.	Note that in this case, it is not possible to have $\beta_0=1$, because that would imply $\bI_{[0]}(A)=\bI_{[0]1}(A)=\bI_{[0]0}(A)=\{0\}$, giving us $\rk_R(\bI_{[0]}(A))=\{0\}$ which contradicts our assumption. Hence we are left with two subcases, either $\beta_0=2$ or $3$. 
	
	\underline{Case 3.1 $(\beta_0=2)$}: If $\beta_0=2$, from Proposition 7.2 of \cite{BoSa1}, we have $o_2=2$. Hence by Theorem \ref{rk_H_delta} we obtain $o_1=4-2=2$ and as isocrystals,  $$\Iso(\mH_{[0]}(A))\not\simeq\Iso(\mH_{[1]}(A))$$ and
	$$\Iso(\mH_{[m]}(A))\simeq\Iso(\mH_{[1]}(A)).$$
	Note that in this subcase, as we have $\bX_1(A)=\{0\}$, therefore we obtain, $$\DD_{[1]}(A)_K=\bXp^{[1]}(A)_K/u_1^*\bX_1(A)_K= \bXp^{[1]}(A)_K\simeq K\langle\Theta_1,\Theta_2,\phi^*\Theta_1.\phi^*\Theta_2\rangle\simeq \mH_{[m]}(A)_K,$$ for all $m\geq 0$. 	This subcase can be illustrated by the product abelian surface $A=E_1\times E_2$, where neither $E_1$ nor $E_2$ admits a canonical lift. 
	
	\underline{Case 3.2 $(\beta_0=3)$}: If $\beta_0=3$, then by Proposition 7.2 of \cite{BoSa1}, we obtain $o_2=3$ and hence by Theorem \ref{rk_H_delta} we have $o_1=1$.
	Therefore this implies $$\DD_{[1]}(A)_K=\bXp^{[1]}(A)_K/u_1^*(\bX_1(A)_K)=K\langle\phi^*\Theta_1,\Theta_2,\phi^*\Theta_2\rangle$$ and  $\rk_K(\DD_{[1]}(A)_K)=3$. Similarly we have $$\DD_{[2]}(A)_K=\bXp^{[2]}(A)_K/u_2^*(\bX_1(A)_K)=K\langle \phi^{*2}\Theta_1,\Theta_2,\phi^*\Theta_2,\phi^{*2}\Theta_2\rangle$$ and $\rk_K(\DD_{[2]}(A)_K)=4$. We have $\rk_K(\mH_{[1]}^\delta(A)_K)= 4$ and hence $\rk_K(\bI_{[1]}(A))= 1.$  Hence for $m=1$, we obtain a proper filtration $$\mH_{[1]}^\delta(A)_K^\bullet=\{\mH_{[1]}^\delta(A)_K\supsetneq\DD_{[1]}(A)_K\supsetneq\{0\}\}.$$
	For $m\geq 2,$ by Theorem \ref{rk_H_delta}, we obtain 
	$$\DD_{[m]}(A)_K\simeq \mH_{[m]}^\delta(A)_K$$ for all $m\geq 2$ and $\DD_{[m]}(A)_K\simeq \DD_{[m+1]}(A)_K$, giving us, $$\Iso(\mH_{[m]}(A))\simeq\Iso(\mH_{[2]}(A))$$ and $\Iso(\mH_{[0]}(A)),\Iso(\mH_{[1]}(A)) \text{ and }\Iso(\mH_{[2]}(A))$ give us three distinct isocrystals connected by the Frobenius map.	
	Combining the cases above we have the following theorem:
	
	\begin{theorem}
		For any abelian surface $A$ over $S$, we have two classes of primitive characters $\Theta_1,\Theta_2\in \bXp^{[0]}(A)$ such that $\bXp^{[0]}(A)\cong R\langle \Theta_1,\Theta_2\rangle.$  Let $o_1$ and  $o_2$ be the exact orders of $\Theta_1$ and $\Theta_2$ respectively and assume $o_1\leq o_2$. Then we have the following cases:
		\begin{center}
			\begin{tabular}{|l|l|l|l|}
				\hline
				$\rk_K \bI_{[0]}(A)_K$ & $0$  & $1$   & $2$  \\
				\hline
				Value of $\beta_0$ & $1$ & $2$ & $\leq 3$ \\
				\hline
				Values of $o_1$ and $o_2$ & $o_1 =1,~ o_2 = 1$ & $o_1 = 1, ~ o_2 = 2$ & $\beta_0 = 1$ is not possible  \\
				& & & If $\beta_0 = 2 $, then $o_1 = o_2 = 2$ \\
				& & & If $\beta_0 = 3 $, then $o_1 = 1, ~ o_2 = 3$ \\
				\hline
				$\rk_K\Fcr H^0(A,\Omega_A)_K=$ & 2 & 3  & 4 \\
				$\rk_K\mH^\delta_{[m]}(A) = o_1 + o_2$ & & & \\
				\hline
			\end{tabular}
		\end{center}
	\end{theorem}

	\textbf{Acknowledgments.} The first author would like to thank CSIR for  financial support under the scheme 09/1031(0011)/2021-EMR-I.
	The first author was also partially supported by the Postdoctoral fellowship at the Tata Institute of Fundamental Research (TIFR). 
The second author would like to thank the Lodha Mathematical Sciences Institute (LMSI) for their hospitality and support. He was partially supported by the MATRICS grant ANRF/ARGM/2025/001294/MTR. He would also like to thank Netan Dogra for many insightful discussions.

	\color{black}
	

\begin{thebibliography}{B}
		\bibitem{Bar}M.~A. Barcau, Isogeny covariant differential modular forms and the space of elliptic curves up to isogeny, Compositio Math. {\bf 137} (2003), no.~3, 237--273.
		
		\bibitem{BePrSa} A. Bertapelle, E. Previato and A. Saha, Arithmetic jet spaces, J. Algebra {\bf 623} (2023), 127--153.
		
		\bibitem{BeOg} P. Berthelot and A. Ogus, $F$-isocrystals and de Rham cohomology. I, Invent. Math. {\bf 72} (1983), no.~2, 159--199.
		

		\bibitem{Bo1}J. Borger, The basic geometry of Witt vectors, I: The affine case, Algebra Number Theory {\bf 5} (2011), no.~2, 231--285.		
		
		\bibitem{Bo2}J. Borger, The basic geometry of Witt vectors. II: Spaces, Math. Ann. {\bf 351} (2011), no.~4, 877--933; 
		
		\bibitem{BoGu} J. Borger and L. Gurney, Canonical lifts and $\delta$-structures, Selecta Math. (N.S.) {\bf 26} (2020), no.~5, Paper No. 67, 48 pp.
		
		\bibitem{BoSa1} J. Borger and A. Saha, Isocrystals associated to arithmetic jet spaces of abelian schemes, Adv. Math. {\bf 351} (2019), 388--428.
		
		\bibitem{BoSa2} J. Borger and A. Saha, Differential characters of Drinfeld modules and de Rham cohomology, Algebra Number Theory {\bf 13} (2019), no.~4, 797--837.
		
		\bibitem{BriCon}O. Brinon, B. Conrad. CMI Summer school notes on p-adic Hodge theory, http://math .stanford .edu /~conrad /papers /notes .pdf.
		
		\bibitem{Bu2}
		A. Buium, Differential characters of abelian varieties over $p$-adic fields, Invent. Math. {\bf 122} (1995), no.~2, 309--340.		
		
		\bibitem{Bu3}
		A. Buium, Differential modular forms, J. Reine Angew. Math. {\bf 520} (2000), 95--167.
		
		\bibitem{Bu1} A. Buium, Geometry of $p$-jets, Duke Math. J. {\bf 82} (1996), no.~2, 349--367.
		
		\bibitem{Bu4} A. Buium. Arithmetic Differential equations. AMS, Volume 118, 2005.
		
		\bibitem{Bu-Mi-1}
		A. Buium and L.~E. Miller, Perfectoid spaces arising from arithmetic jet spaces, Amer. J. Math. {\bf 145} (2023), no.~1, 287--334.
		
		\bibitem{Bu-Mi-3}
		A. Buium and L.~E. Miller, {\it Purely arithmetic PDEs over a $p$-adic field: $\delta$-characters and $\delta$-modular forms}, Memoirs of the European Mathematical Society, 6, EMS Press, Berlin, 2023; 
		
		\bibitem{Bu-Mi-2}A. Buium and L.~E. Miller, Solutions to arithmetic differential equations in algebraically closed fields, Adv. Math. {\bf 375} (2020), 107342, 47.		
		
		\bibitem{Bu-Po} A. Buium and B. Poonen, Independence of points on elliptic curves arising from special points on modular and Shimura curves. I. Global results, Duke Math. J. {\bf 147} (2009), no.~1, 181--191;		
		
		\bibitem{BuSa1} A. Buium and A. Saha, Hecke operators on differential modular forms mod $p$, J. Number Theory {\bf 132} (2012), no.~5, 966--997.
		
		\bibitem{BuSa2} A. Buium and A. Saha, The ring of differential Fourier expansions, J. Number Theory {\bf 132} (2012), no.~5, 896--937.
		

\bibitem{DP25}
N.~Dogra and S.~Pandit.
\newblock A Buium--Coleman bound for the Mordell--Lang conjecture, 
\newblock arXiv:2504.10155. 

\bibitem{DP26}
N.~Dogra and S.~Pandit.
\newblock A Buium--Coleman bound for the Zilber--Pink conjecture for curves inside abelian varieties, 
\newblock arXiv:2608.24474.

\bibitem{DS26}
N.~Dogra and A.~Saha.
\newblock Arithmetic jet spaces and the unramified Zilber--Pink conjecture for subvarieties of abelian varieties,
\newblock arXiv:2610.17325.

		\bibitem{Drn} V.~G. Drinfeld, Coverings of $p$-adic symmetric domains, Funkcional. Anal. i Prilo\v zen. {\bf 10} (1976), no.~2, 29--40.
		
		
		\bibitem{Hurlburt} C. Hurlburt, Isogeny covariant differential modular forms modulo $p$, Compositio Math. {\bf 128} (2001), no.~1, 17--34.
		
		\bibitem{Joyal} A. Joyal, $\delta$-anneaux et vecteurs de Witt, C. R. Math. Rep. Acad. Sci. Canada {\bf 7} (1985), no.~3, 177--182.
		
		\bibitem{Mi-Sa-1}R.~K. Mishra and A.Saha,  Kernels of Arithmetic Jet Spaces and Frobenius Morphism, arXiv preprint arXiv.2601.22591
		
		\bibitem{Mi-Sa-2}R.~K. Mishra and A. Saha, Differential characters and $D$-group schemes, Int. Math. Res. Not. IMRN {\bf 2025}, no.~9, Paper No. rnaf109, 18 pp.
		
		\bibitem{Delta-crys} L.Gurney, S. Pandit, A. Saha. Delta characters and Filtered Isocrystals , arXiv preprint	arXiv:2609.28365
		
		\bibitem{Mazur}
		B. Mazur, W. Messing. Universal extensions and one dimensional crystalline cohomology. Springer,volume=370, 2006
		
\bibitem{MM25}
L. Miller and J. Morrow.
\newblock Higher dimensional geometry of $p$-jets.
\newblock arXiv:2510.00336, 2025.

\bibitem{Mum-66}
\newblock D. Mumford.
\newblock Lectures on curves on an algebraic surface,
\newblock Ann. of Math. Stud., No. 59
Princeton University Press, Princeton, NJ, 1966, xi+200 pp.

\bibitem{Pandit} S. Pandit, Delta Isocrystal and Crystalline Cohomology of Abelian Schemes, arXiv preprint, arXiv:2605.00713
		
		\bibitem{PaSa} S. Pandit and A. Saha, Delta characters and crystalline cohomology, Camb. J. Math. {\bf 13} (2025), no.~2, 301--358.
		
		\bibitem{PaSa2} S. Pandit and A. Saha, On de Rham cohomology of Drinfeld modules of rank 2, Int. J. Number Theory {\bf 21} (2025), no.~7, 1437--1457.
		
		
		\bibitem{Se1} Jean-Pierre Serre. Algebraic groups and class fields. Springer Science \& Business Media, volume=117(2012)
		
	\end{thebibliography}
\end{document}